\documentclass[12pt, a4paper]{amsproc}

\usepackage{setspace,hyperref,longtable}
\hypersetup{citecolor=red, linkcolor=blue, colorlinks=true}

\usepackage{latexsym}
\usepackage{amssymb, amsthm, amsfonts, amsmath, booktabs}
\usepackage[margin=2cm]{geometry}

\usepackage{enumerate}
\usepackage[shortlabels]{enumitem}
\usepackage{url}

\setlist[description]{leftmargin=*}

\newtheorem{thm}{Theorem}[section]
\newtheorem{cor}[thm]{Corollary}
\newtheorem{lem}[thm]{Lemma}

\newtheorem{hyp}[thm]{Hypothesis}
\newtheorem{prop}[thm]{Proposition}

\theoremstyle{remark}
\newtheorem{remark}[thm]{Remark}

\newcommand\Aut{{\rm Aut}}
\newcommand\Sz{{\rm Sz}}

\newcommand\GAP{\textbf{{\rm GAP}}}
\newcommand\PSL{{\rm PSL}}
\newcommand\SL{{\rm SL}}
\newcommand\GL{{\rm GL}}
\newcommand\PGL{{\rm PGL}}

\newcommand\PSU{{\rm PSU}}
\newcommand\Out{{\rm Out}}
\newcommand\Ree{{\rm Ree}}

\newcommand\Sp{{\rm Sp}}
\newcommand\SU{{\rm SU}}
\newcommand\GO{{\rm GO}}
\newcommand\SO{{\rm SO}}
\newcommand\Q{{\rm Q}}
\newcommand\GU{{\rm GU}}
\newcommand\PSp{{\rm PSp}}
\newcommand\W{{\rm W}}
\newcommand\herm{{\rm H}}
\newcommand\gqthreefive{\mathsf{GQ}(3,5)}
\newcommand\rmO{{\rm O}}

\newcommand\POmega{{\rm P\Omega}}

\DeclareMathOperator{\soc}{soc}
\newcommand\PG{{\rm PG}}
\newcommand\Cos{{\sf Cos}}
\newcommand\C{{\bf C}}
\newcommand\Wr{{\rm \,wr\, }}
\newcommand\calP{{\mathcal P}}
\newcommand\calL{{\mathcal L}}

\newcommand\calQ{{\mathcal Q}}
\newcommand\calC{{\mathcal C}}
\newcommand\calO{{\mathcal O}}

\newcommand\McL{{\rm McL}} 
 
\newcommand\E{{\rm E}} \newcommand\F{{\rm F}} \newcommand\D{{\rm D}} \newcommand\G{{\rm G}}
 
\newcommand\Co{{\rm Co}}

\renewcommand\le{\leqslant}
\renewcommand\ge{\geqslant}

\title{A classification of finite antiflag-transitive generalized quadrangles}

\author{John Bamberg}
\thanks{This paper forms part of an Australian Research Council Discovery Project (DP120101336), that supported the third author during his time at The University of Western Australia.
The first author was supported by an Australian Research Council Future Fellowship (FT120100036).}

\email{John.Bamberg@uwa.edu.au}

\author{Cai Heng Li}
\email{Cai.Heng.Li@uwa.edu.au}

\address{ %
Centre for the Mathematics of Symmetry and Computation\\
School of Mathematics and Statistics\\
The University of Western Australia\\
35 Stirling Highway, Crawley, W.A. 6009, Australia.}

\author{Eric Swartz}
\email{easwartz@wm.edu}

\address{Department of Mathematics\\ 
College of William and Mary\\ 
Williamsburg, VA 23185, USA.}


\begin{document}

\maketitle

\begin{abstract}
A generalized quadrangle is a point-line incidence geometry $\mathcal{Q}$ such that: (i) any two
points lie on at most one line, and (ii) given a line $\ell$ and a point $P$ not incident with
$\ell$, there is a unique point of $\ell$ collinear with $P$. The finite Moufang generalized
quadrangles were classified by Fong and Seitz (1973), and we study a larger class of generalized
quadrangles: the \emph{antiflag-transitive} quadrangles. An antiflag of a generalized quadrangle is
a non-incident point-line pair $(P, \ell)$, and we say that the generalized quadrangle $\mathcal{Q}$
is antiflag-transitive if the group of collineations is transitive on the set of all antiflags. We
prove that if a finite thick generalized quadrangle $\mathcal{Q}$ is antiflag-transitive, then
$\mathcal{Q}$ is either a
classical generalized quadrangle or is the unique generalized quadrangle of order $(3,5)$ or its dual.
\end{abstract}

\section{Introduction}

In this paper, we classify the finite generalized quadrangles having an automorphism group that acts
transitively on \emph{antiflags}, that is, non-incident point-line pairs.

\begin{thm}\label{maintheorem}
Let $\mathcal{Q}$ be a finite thick generalized quadrangle and suppose $G$ is a subgroup of
automorphisms of $\mathcal{Q}$ acting transitively on the antiflags of $\mathcal{Q}$. Then
$\mathcal{Q}$ is isomorphic to a classical generalized quadrangle or to the unique generalized
quadrangle of order $(3,5)$ or its dual.
\end{thm}

This result is a generalization of the classification of finite Moufang quadrangles. By a deep result of
Fong and Seitz (1973), the finite generalized polygons satisfying a symmetry property known as the
\emph{Moufang condition} are precisely the known generalized polygons arising from simple groups of
Lie type. The Moufang condition has a group theoretic counterpart known as a \emph{split
$(B,N)$-pair}: a $(B,N)$-pair of rank $2$ such that there exists a normal nilpotent subgroup $U$ of
$B$ with $B=U(B\cap N)$. 
A {\it generalized $n$-gon} ({\it polygon}) is a geometry of points and lines whose
incidence graph has girth $2n$ and diameter $n$. Alternatively, a generalized $n$-gon is an
irreducible spherical building of rank $2$. If we reverse the role of points and lines of a generalized $n$-gon,
then we obtain another generalized $n$-gon; its \emph{dual}.
We say a generalized $n$-gon is \emph{thick} if every
vertex of the incidence graph has valency at least $3$. By a result of Feit and Higman \cite{Feit:1964aa}, a thick
generalized $n$-gon has $n\in\{2,3,4,6,8\}$. For $n=2$, the incidence graph is nothing more than a
complete bipartite graph, and a generalized $3$-gon is precisely a projective plane. In this paper,
we will be primarily concerned with generalized $4$-gons (quadrangles). 

Ostrom and Wagner \cite{Ostrom:1958qv,Ostrom:1959ys} proved that 
if a group $G$ of automorphisms of a projective plane $\Pi$ of order $q$ acts transitively on the antiflags of $\Pi$,
then $\Pi$ is Desarguesian and $\PSL_3(q)\le G$. 
The incidence graph $\Gamma$ of a generalized $n$-gon has diameter $n$, and 
hence for $n>4$, the automorphism group of $\Gamma$ cannot act antiflag-transitively since there are
non-incident lines at distinct distances from a given point. 
So Theorem \ref{maintheorem} gives a complete classification of generalized $n$-gons having 
a group of automorphisms $G$ acting antiflag-transitively on their incidence graphs.

There are two notable generalizations of the result of Fong and Seitz that are due to Buekenhout and
Van Maldeghem \cite{BuekenhoutHvM1994b,BuekenhoutHvM1994}. They classified the generalized polygons
that are \emph{point-distance-transitive} and those that are \emph{geodesic-transitive}. For
generalized quadrangles, these results can be reinterpreted using the language of local symmetry;
their results imply the classification of the finite generalized quadrangles whose incidence graphs are locally
$4$-arc-transitive. An $s$-arc in a simple graph is a path of $s+1$ vertices $(v_0,v_1,\ldots v_s)$
with $v_i\ne v_{i+1}$ for all $0\le i < s$. The graph $\Gamma$ is locally $s$-arc-transitive with
respect to $G\le\Aut(\Gamma)$ if for each vertex $v\in V\Gamma$, the group $G_v$ acts transitively
on the set of $s$-arcs originating at $v$. In this language, a generalized quadrangle is
\emph{Moufang} if it satisfies:
\begin{quote}
For each $2$-path $(v_0,v_1,v_2)$, the group $G_{v_0}^{[1]}\cap G_{v_1}^{[1]}\cap G_{v_2}^{[1]}$
acts transitively on $\Gamma(v_2)\setminus\{v_1\}$.
\end{quote}
(If the reader is unfamiliar with the above notation, then we refer to Section
\ref{section:groupsgraphs} for the necessary background). We will be considering a much weaker
condition -- \emph{local $3$-arc-transitivity} -- which is equivalent to transitivity on antiflags
(Lemma \ref{anti-f-trans-local-3-a-trans}) for generalized quadrangles. It is not difficult to see
that every Moufang generalized quadrangle is locally $3$-arc-transitive; the converse does not hold,
however. A counter-example being the unique generalized quadrangle of order $(3,5)$ and its dual
which arise from the construction $T_2^*(\mathcal{O})$ (see \cite[\S3.1.3]{fgq}) where $\mathcal{O}$
is a regular hyperoval of the projective plane $\PG(2,4)$ of order $4$. Its full automorphism group
is isomorphic to $2^6:(3.A_6.2)$ and it acts locally $3$-arc-transitively.

By the definition of a generalized quadrangle, each antiflag determines a unique flag, and hence
antiflag-transitivity implies flag-transitivity. Kantor \cite{Kantor1991} conjectured that the
finite flag-transitive generalized quadrangles are known, and our result gives some evidence for the
validity of this conjecture. All of the classical generalized quadrangles are flag-transitive, and
there are only two other examples known up to duality: the unique generalized quadrangle of order
$(3,5)$ and a generalized quadrangle of order $(15,17)$, the \emph{Lunelli-Sce quadrangle}, obtained
by the construction $T_2^*(\mathcal{L})$ where $\mathcal{L}$ is the Lunelli-Sce hyperoval of the
Desarguesian projective plane $\PG(2,16)$ of order $16$. The latter example is not
antiflag-transitive, and hence Kantor's Conjecture simplifies: a finite flag-transitive generalized
quadrangle that is not antiflag-transitive is isomorphic to the \emph{Lunelli-Sce quadrangle} or its
dual.

Our main result relies on the Classification of Finite Simple Groups, but does not rely directly on the
earlier results of Fong and Seitz, or Buekenhout and Van Maldeghem. Hence we obtain the following
corollary:

\begin{cor}\leavevmode
\begin{enumerate}[(i)]
\item A finite Moufang thick generalized quadrangle is classical (see Table~\ref{tbl:rank3} for a list).
\item A finite locally $4$-arc-transitive thick generalized quadrangle is classical.
\item A finite antiflag-transitive generalized quadrangle that is not locally $4$-arc-transitive 
is isomorphic to the unique generalized quadrangle of order $(3,5)$ or its dual.
\end{enumerate}
\end{cor}

The strategy for proving the main result of our paper is summarized as follows. Let $\Gamma$ be the
incidence graph of a finite thick generalized quadrangle $\mathcal{Q}$, and let $G$ be a locally
$3$-arc-transitive group of automorphisms of $\Gamma$.
\begin{itemize}
\item We prove (Theorem \ref{thm:3arcNtrans}) that $G$ acts quasiprimitively on, at least one of,
  the set of points or the set of lines of $\mathcal{Q}$;
\item It is shown (Theorem \ref{thm:3arcLimprim}) that if $G$ acts primitively on points but
  imprimitively on lines, then $\mathcal{Q}$ is isomorphic to the unique generalized quadrangle of
  order $(3,5)$.
\item We then establish that apart from the above case, we have $G$ acting primitively on both
  points and lines of almost simple type (Lemma \ref{lem:3arcQuasi=Prim} and Theorem
  \ref{thm:characterization}).
\item The stabilizer of a point $P$ is then shown to be a \emph{large} subgroup of $G$ (Corollary
  \ref{cor:largesub}), and we use the characterization of large subgroups of simple groups of Lie
  type by Alavi and Burness \cite{largesubs} to determine possibilities for $G$ and $G_P$.
\end{itemize}

\section{Finite generalized quadrangles}

In this section we give the necessary background on finite generalized quadrangles. For a far more
thorough treatment of the combinatorics of finite generalized quadrangles, see \cite{fgq}.
Throughout this paper, we will let $\mathcal{Q}=(\calP,\calL)$ denote a finite generalized
quadrangle with point set $\mathcal{P}$ and line set $\mathcal{L}$.

\subsection{Parameters of generalized quadrangles}

A generalized quadrangle $\mathcal{Q}$ is said to have \textit{order} $(s,t)$ if every line is
incident with $s+1$ points and every point is incident with $t+1$ lines. The following lemmas
summarize some basic results concerning the parameters $s$ and $t$.

Let $\mathcal{Q}=(\calP,\calL)$ be a finite generalized quadrangle of order $(s,t)$, where $s,t >
1$. The following three lemmas give some basic properties about $\calQ$.

\begin{lem}[{\cite[1.2.1, 1.2.2, 1.2.3, 1.2.5]{fgq}}]
\label{lem:GQbasics} The following hold:
\begin{itemize}
 \item[(i)] $|\mathcal{P}| = (s+1)(st + 1)$ and $|\mathcal{L}| = (t+1)(st+1)$;
 \item[(ii)] $s + t$ divides $st(s+1)(t+1)$;
 \item[(iii)] $t \le s^2$ and $s \le t^2$;
 \item[(iv)] if $s < t^2$, then $s \le t^2 - t$, and if $t < s^2$, then $t \le s^2 - s$.
\end{itemize}
\end{lem}

\begin{lem}
\label{lem:ratio} Let $G$ be a group that is transitive on both $\mathcal{P}$ and $\mathcal{L}$.
Then, for $P\in\calP$ and $\ell\in\calL$, $$\frac{s+1}{t+1} = \frac{|G_\ell|}{|G_P|}.$$
\end{lem}

\begin{proof}
Using Lemma \ref{lem:GQbasics} (i), we see that $$\frac{s+1}{t+1} = \frac{(s+1)(st+1)}{(t+1)(st+1)}
= \frac{|\mathcal{P}|}{|\mathcal{L}|} = \frac{|G\colon G_P|}{|G\colon G_\ell|} =
\frac{|G_\ell|}{|G_P|},$$ as desired.
\end{proof}

\begin{lem}
\label{lem:Pbounds} Assuming $s \le t$, the following inequalities hold:
\begin{itemize}
 \item[(i)] $(t+1)^2 < |\mathcal{P}| < (t+1)^3$
 \item[(ii)] $s^2 (t+1) < |\mathcal{P}| < s(t+1)^2$
\end{itemize}
\end{lem}

\begin{proof}
To show (i), we use Lemma \ref{lem:GQbasics}(i) and (iii): $$(t+1)^2 < t^2 + (t^{\frac{3}{2}} +
t^{\frac{1}{2}}) + 1 = (t^{\frac{1}{2}} + 1)(t^{\frac{3}{2}} + 1) \le (s+1)(st+1) = |\mathcal{P}|
\le (t+1)(t^2 + 1) < (t+1)^3.$$

The proof of (ii) is similar: $$s^2(t+1) = s^2t + s^2 < s^2t + st + s + 1 = |\mathcal{P}| < st^2 +
st + st + s = s(t+1)^2.$$
\end{proof}

\subsection{The known examples}

Here we provide data on the known examples of antiflag-transitive generalized quadrangles which
enable us to quickly identify them. During the course of this paper, we will be considering an
almost simple group $G$ and a maximal subgroup $H$, that will serve as a point stabilizer in a
primitive action. For most of the known examples, the action of $G$ on the right cosets of $H$ has
permutation rank $3$ and the generalized quadrangle can be readily identified. We will elaborate
below.

For a group $G$ and a subgroup $H$ of $G$, the group $G$ acts on $\{Hg\mid g\in G\}$ by right
multiplication and the kernel of the action is the \emph{core} $\cap_{g\in G}H^g$ of $H$ in $G$. Let
$G$ be a group with $A$ and $B$ proper subgroups of $G$. The {\it coset geometry} $\Cos(G;A,B)$ has
point set $\{Ax\mid x\in G\}$ and line set $\{By\mid y\in G\}$ such that a point $P:=Ax$ and a line
$\ell:=By$ are incident if and only if $Ax\cap By\neq \varnothing$. In particular, $G$ is a
flag-transitive group of automorphisms of this geometry and we have the following converse:

\begin{lem}[{\cite[Lemma 1]{HMc61}}]\label{prop:geom-flag-trans}
Let $\mathcal{G}$ be a geometry of points and lines, and let $G\le \Aut(\mathcal{G})$. Then $G$ acts
transitively on the flags of $\mathcal{G}$ if and only if $\mathcal{G}\cong \Cos(G;A,B)$ where $A$
is the stabilizer of a point $P$ and $B$ is the stabilizer of a line $\ell$ incident with $P$.
\end{lem}

Hence, if $G$ acts flag-transitively on a generalized quadrangle $\mathcal{Q}$, then we can
recognize the isomorphism type of $\mathcal{Q}$ by knowing the triple $(A,B,A\cap B)$. In the case
that $G$ is a primitive permutation group of rank 3 on the points of $\mathcal{Q}$, we only need to
know the stabilizer of a point or line to establish the type of $\mathcal{Q}$, since the rank $3$
primitive actions of the classical groups were classified by Kantor and Liebler
\cite{KantorLiebler}. Moreover, the finite classical generalized quadrangles are precisely those
admitting a rank $3$ primitive group on their points (see also \cite[p. 252]{Kantor1991}). The
infinite families of finite generalized quadrangles associated with classical groups are known as
the \textit{classical generalized quadrangles.} In each case, the full \textit{collineation group}
-- i.e., automorphisms sending points to points and lines to lines, preserving incidence -- is a
classical group, and the point and line stabilizers are maximal parabolic subgroups of the
appropriate index (the point/line incidence structure corresponds to the incidence of totally
singular subspaces). The classical generalized quadrangles are briefly summarized in Table
\ref{tbl:rank3}. The notation $\soc(G)$ denotes the \emph{socle} of a group $G$; the product of the
minimal normal subgroups of $G$. We write $E_q^{a}$ for the elementary abelian group of order $q^a$
(where $q$ is a prime power), and we write $E_q^{a+b}$ for a special group of order $q^{a+b}$ with
centre of order $q^a$. If we do not need to specify the structure of a subgroup, we will often write
$[n]$ to denote an undetermined subgroup of order $n$.

\begin{table}[ht]
\caption{The classical generalized quadrangles given by certain rank $3$ classical
groups.}\label{tbl:rank3}
\begin{tabular}{l|c|c|l}
\toprule
$\mathcal{Q}$& Order & $\soc(G)$ & Point stabilizer in $\soc(G)$\\
\midrule
$\W(3,q)$, $q$ odd& $(q,q)$ & $\PSp_4(q)$ & $E_q^{1+2}: ( \GL_1(q)\circ \Sp_2(q))$\\
$\W(3,q)$, $q$ even& $(q,q)$ & $\PSp_4(q)$ & $E_q^{3}:\GL_2(q)$\\
$\Q(4,q)$, $q$ odd & $(q,q)$ & $\POmega_5(q)$ & $E_q^{3}:(  (\frac{(q-1)}{2}\times \Omega_3(q)).2)$\\
$\Q^-(5,q)$ & $(q,q^2)$ &$\POmega^-_6(q)$ & $E_q^4: (\frac{q-1}{|Z(\Omega^-_6(q))|}\times \Omega^-_4(q))$\\
$\herm(3,q^2)$ & $(q^2,q)$ &$\PSU_4(q)$ & $E_q^{1+4}:\left(\SU_2(q):\frac{q^2-1}{\gcd(q+1,4)}\right)$\\
$\herm(4,q^2)$ &$(q^2,q^3)$ &$\PSU_5(q)$ & $E_q^{1+6}:\left(\SU_3(q):\frac{q^2-1}{\gcd(q+1,5)}\right)$\\
$\herm(4,q^2)^D$ &$(q^3,q^2)$ &$\PSU_5(q)$ & $E_q^{4+4}:\GL_2(q^2)$\\
\bottomrule
\end{tabular}
\end{table}

The classical generalized quadrangles come in dual pairs: $\W(3,q)$ is isomorphic to the dual of
$\Q(4,q)$, $\Q^-(5,q)$ is isomorphic to the dual of $\herm(3,q^2)$, and $\herm(4,q^2)^D$ denotes the
dual of $\herm(4,q^2)$.  

We will also make use of the following classification of small generalized quadrangles.

\begin{lem}[{\cite[\S6]{fgq}}]\label{lemma:smallGQ}
Let $\mathcal{Q}$ be a finite generalized quadrangle of order $(s,t)$.
\begin{enumerate}[(i)]
\item If $(s,t)=(2,2)$, then $\mathcal{Q}$ is isomorphic to $\W(3,2)$.
\item If $(s,t)=(2,4)$, then $\mathcal{Q}$ is isomorphic to $\Q^-(5,2)$.
\item If $(s,t)=(3,3)$, then $\mathcal{Q}$ is isomorphic to $\W(3,3)$ or $\Q(4,3)$.
\item There is (up to isomorphism) a unique generalized quadrangle of order $(3,5)$.
\end{enumerate}
\end{lem}

Apart from a single exception, the only known finite generalized quadrangles that are
antiflag-transitive are the classical generalized quadrangles. The lone exception (along with its
dual) is the unique generalized quadrangle of order $(3,5)$, henceforth referred to as
$\gqthreefive$. In this case, the collineation group is isomorphic to $2^6{:}(3.A_6.2)$, the
stabilizer of a point is isomorphic to $3.A_6.2$, and the stabilizer of a line is isomorphic to
$(A_5 \times A_4).2$. It should be noted that $\gqthreefive$ is point-distance-transitive
\cite{BuekenhoutHvM1994}, but it is not line-distance-transitive; indeed, $\gqthreefive$ is the only
known example of a finite generalized quadrangle that is antiflag-transitive but not Moufang.

\section{Permutation group theory and graph symmetry}\label{section:groupsgraphs}

In this section we review the permutation group theory and graph symmetry results necessary for this
paper. For a more complete discussion of these notions, the reader is referred to \cite{localsarc}.
Let the group $G$ act on the set $\Omega,$ and let $\omega \in \Omega.$ We denote the
$G$-\textit{orbit} of $\omega$ by $\omega^G$, and we refer to the \textit{stabilizer} of $\omega$ in
$G$ by $G_{\omega}$. Given a set $\Sigma \subseteq \Omega$, the subgroup of $G$ that fixes $\Sigma$
setwise is denoted by $G_{\Sigma}$ and the subgroup of $G$ that fixes every element of $\Sigma$ is
denoted by $G_{(\Sigma)}$.

An \textit{automorphism} of a graph $\Gamma$ is a permutation of the vertices that preserves
adjacency and non-adjacency. The set of all automorphisms of a graph $\Gamma$ forms a group and is
denoted by $\Aut(\Gamma)$. Let $\alpha$ be a vertex of a graph $\Gamma$, and let $G \le
\Aut(\Gamma)$. We denote the vertices adjacent to $\alpha$ in $\Gamma$ by $\Gamma(\alpha)$, and the
permutation group induced by $G_{\alpha}$ on $\Gamma(\alpha)$ will be denoted by
$G_{\alpha}^{\Gamma(\alpha)}.$ For any two vertices $\alpha, \beta$ of $\Gamma$, we define the
distance function $d(\alpha, \beta)$ to be the length of the shortest path between $\alpha$ and
$\beta$ in $\Gamma$. Given a natural number $i$, we define

$$G_{\alpha}^{[i]} := \{g \in G \mid \beta^g = \beta \text{ for all } \beta \text{ satisfying }
d(\alpha, \beta) \le i \}.$$ Note that $G_{\alpha}^{[1]}$ is a normal subgroup of $G_{\alpha}$, and
$G_{\alpha}^{\Gamma(\alpha)} \cong G_{\alpha}/G_{\alpha}^{[1]}.$

An \textit{s-arc} of a graph $\Gamma$ is a sequence of vertices $(\alpha_0, \alpha_1, ...,
\alpha_s)$ such that $\alpha_{i}$ is adjacent to $\alpha_{i+1}$ and $\alpha_{i-1} \neq \alpha_{i+1}$
for all possible $i$. Note that vertices can be repeated as long as $\alpha_{i-1} \neq \alpha_{i+1}$
for all possible $i$. Given a subgroup $G \le \Aut(\Gamma)$, $\Gamma$ is
$(G,s)$\textit{-arc-transitive} if $\Gamma$ contains an $s$-arc and any $s$-arc in $\Gamma$ can be
mapped to any other $s$-arc in $\Gamma$ via an element of $G$. The graph is \textit{locally}
$(G,s)$\textit{-arc-transitive} if $\Gamma$ contains an $s$-arc and, for any vertex $\alpha$ of
$\Gamma$, any $s$-arc starting at $\alpha$ can be mapped to any other $s$-arc starting at $\alpha$
via an element of $G$. In the cases that such a group $G$ exists, the graph $\Gamma$ is said to be
$s$\textit{-arc-transitive} or \textit{locally} $s$\textit{-arc-transitive}, respectively. Note that
it is possible for a graph to be locally $(G,s)$-arc-transitive but for $G$ to be intransitive on
the set of vertices. (As an example, one could take the complete bipartite graph $K_{2,3}$ with $G$
the full automorphism group.) On the other hand, when $\Gamma$ is locally $(G,s)$-arc-transitive and
every vertex in $\Gamma$ is adjacent to at least two other vertices, $G$ is transitive on the edges
of $\Gamma$. By definition, a locally $(G,3)$-arc-transitive graph is locally
$(G,2)$-arc-transitive. It is easily seen that a graph $\Gamma$ is locally $(G,2)$-arc-transitive if
and only if $G_\alpha^{\Gamma(\alpha)}$ is a 2-transitive permutation group for all vertices
$\alpha$.

Let $\Gamma$ be a graph with a group of automorphisms $G$. If $G$ has a normal subgroup $N$ that
acts intransitively on the vertices of $\Gamma$, define the \textit{(normal) quotient graph}
$\Gamma_N$ to have vertex-set the $N$-orbits of vertices of $\Gamma$, and two $N$-orbits $\Sigma_1$
and $\Sigma_2$ are adjacent in $\Gamma_N$ if and only if there exist vertices $\alpha \in \Sigma_1$
and $\beta \in \Sigma_2$ such that $\alpha$ is adjacent to $\beta$ in $\Gamma.$ Giudici, Li, and
Praeger \cite{localsarc} showed that if $\Gamma$ is a locally $(G,s)$-arc-transitive graph, then
$\Gamma_N$ is a locally $(G/N,s)$-arc-transitive graph unless $\Gamma_N$ is a star.

A transitive group $G$ acting on a set $\Omega$ is called \textit{quasiprimitive} if every
nontrivial normal subgroup $N$ of $G$ is transitive on $\Omega$. Indeed, the nontrivial normal
subgroups of any primitive group are transitive, and hence primitive groups are quasiprimitive.
Locally $s$-arc-transitive graphs with a group of automorphisms acting quasiprimitively on at least
one orbit of vertices have been studied extensively; see \cite{loccompbit, localrees, localsarc,
localstar, localdifferent, localsporadic, localsuz}.

A generalized quadrangle $\calQ=(\calP,\calL)$ is called {\it locally $(G,2)$-transitive} if for
each $P\in\calP$ and each $\ell\in\calL$, the stabilizer $G_P$ is 2-transitive on the lines which
are incident with $P$ and $G_\ell$ is 2-transitive on the points which lie on $\ell$.
Observe that, for an antiflag $(P,\ell)$ in a generalized quadrangle, there is a unique 3-arc
between $P$ and $\ell$, since by the definition of a generalized quadrangle there is a unique point on $\ell$ collinear with $P$. We thus have the following conclusion.

\begin{lem}\label{anti-f-trans-local-3-a-trans}
Let $\Gamma$ be the incidence graph of a generalized quadrangle $\mathcal{G}$, and
$G\leqslant\Aut(\mathcal{G})$. Then $\mathcal{G}$ is $G$-antiflag-transitive if and only if $\Gamma$
is locally $(G,3)$-arc-transitive. In particular, an antiflag-transitive generalized quadrangle is
flag-transitive and locally $2$-transitive.
\end{lem}

Let $\Gamma$ be a connected $G$-edge-transitive graph. For vertices
$\alpha_0,\alpha_1,\dots,\alpha_l$, let
\[G_{\alpha_0\alpha_1\dots \alpha_l}^{[1]}=G_{\alpha_0}^{[1]}\cap G_{\alpha_1}^{[1]}\cap\cdots\cap G_{\alpha_l}^{[1]}.\]

The following simple lemma was first obtained in \cite{Li-Song}, showing that some information of $(G_\alpha,G_\beta,G_{\alpha\beta})$
can be obtained from the permutation groups  $G_\alpha^{\Gamma(\alpha)}$ and $G_\beta^{\Gamma(\beta)}$.
For completeness, we give a proof here.

\begin{lem}\label{comp-factor}
Let $\Gamma$ be a connected graph, and let $G\le\Aut(\Gamma)$ be transitive on the edge set.
Let $\{\alpha,\beta\}$ be an edge of $\Gamma$.
Then each composition factor of $G_\alpha$ is a composition factor of $G_\alpha^{\Gamma(\alpha)}$, $G_{\alpha\beta}^{\Gamma(\beta)}$, or $G_{\alpha\beta}^{\Gamma(\alpha)}$.
Moreover, if $|\Gamma(\alpha)|\geqslant|\Gamma(\beta)|$, then $|\Gamma(\alpha)|$ is not smaller than the smallest permutation degree
of any composition factor of $G_\alpha$.
\end{lem}

\begin{proof}
Since $G\le\Aut(\Gamma)$ and $\Gamma$ is connected, there exists a path
$\alpha_0=\alpha,\alpha_1,\dots,\alpha_l$ such that $G_{\alpha_0\alpha_1\dots \alpha_l}^{[1]}=1$.
Thus we have
$$1=G_{\alpha_0\alpha_1\dots \alpha_l}^{[1]}\lhd G_{\alpha_0\alpha_1\dots \alpha_{l-1}}^{[1]}
  \lhd\dots\lhd G_{\alpha_0\alpha_1}^{[1]}\lhd G_{\alpha_0}^{[1]}=G_\alpha^{[1]}\lhd G_\alpha.$$
Suppose that $T$ is a composition factor of $G_\alpha$.
Then there exist subgroups $M\lhd N\lhd\lhd G_\alpha$ such that $T\cong N/M$.
If  $T$ is a composition factor of $G_\alpha^{\Gamma(\alpha)}$, then we are done.

We thus assume next that $T$ is not a composition factor of $G_\alpha^{\Gamma(\alpha)}$.
Then $T$ is a composition factor of $G_\alpha^{[1]}$
since $G_\alpha^{\Gamma(\alpha)}\cong G_\alpha/G_\alpha^{[1]}$.
Let $i$ be the least integer such that
$N\le G_{\alpha_0\alpha_1\dots \alpha_i}^{[1]}$.
Then $N\,{\not\le}\,G_{\alpha_0\alpha_1\dots \alpha_{i+1}}^{[1]}$, and
\[NG_{\alpha_0\alpha_1\dots \alpha_{i+1}}^{[1]}/G_{\alpha_0\alpha_1\dots \alpha_{i+1}}^{[1]}
\lhd G_{\alpha_0\alpha_1\dots \alpha_i}^{[1]}/G_{\alpha_0\alpha_1\dots \alpha_{i+1}}^{[1]}.\]
So $T$ is a composition factor of
$G_{\alpha_0\alpha_1\dots \alpha_i}^{[1]}/G_{\alpha_0\alpha_1\dots \alpha_{i+1}}^{[1]}$.
Since
\[
G_{\alpha_0\alpha_1\dots \alpha_i}^{[1]}/G_{\alpha_0\alpha_1\dots \alpha_{i+1}}^{[1]}
\cong (G_{\alpha_0\alpha_1\dots \alpha_i}^{[1]})^{\Gamma(\alpha_{i+1})}
\lhd G_{\alpha_i\alpha_{i+1}}^{\Gamma(\alpha_{i+1})},
\]
$T$ is a composition factor of $G_{\alpha_i\alpha_{i+1}}^{\Gamma(\alpha_{i+1})}$. Further, as $G$ is
edge-transitive on $\Gamma$, we have $G_{\alpha_i\alpha_{i+1}}^{\Gamma(\alpha_{i+1})}$ is isomorphic
to $G_{\alpha\beta}^{\Gamma(\alpha)}$ or $G_{\alpha\beta}^{\Gamma(\beta)}$.

So $T$ is a composition factor of $G_{\alpha\beta}^{\Gamma(\beta)}$ or
$G_{\alpha\beta}^{\Gamma(\alpha)}$. It follows that $T$ has a faithful permutation representation on
a set of size $|\Gamma(\alpha)|$ or $|\Gamma(\beta)|$. Thus an orbit of $T$ is not bigger than
$|\Gamma(\alpha)|$ or $|\Gamma(\beta)|$.
\end{proof}

For locally $2$-arc-transitive graphs, the unsolvable composition factors of the stabilizers are
determined in the following lemma.

\begin{lem}\label{stab}\leavevmode
Let $\Gamma$ be locally $(G,2)$-arc-transitive, and let $\{\alpha,\beta\}$ be an edge. Then the
following statements hold.
\begin{itemize}

\item[(i)] Any composition factor of $G_\alpha$ that is not solvable is one of the following
  groups: $A_n$, $\PSL_n(q)$, $\PSU_3(q)$, $\PSp_n(q)$, $\POmega^\pm_n(2)$, $\Sz(q)$, $\Ree(q)$,
  $\G_2(q)$, $\mathrm{HS}$, $\Co_3$, $\McL$, or $M_k$ where $k\in\{11,12,22,23,24\}$.

\item[(ii)] If $\POmega_{n}^\pm(2)$ with $n=2m\geqslant8$ is a composition factor of
  $G_{\alpha\beta}$, then $\PSp_{n}(2)$ is a composition factor of $G_\alpha$ or $G_\beta$.

\item[(iii)] For $n\geqslant 3$, $\PSU_n(q)$ is not a composition factor of $G_{\alpha\beta}$
  unless $(n,q)=(3,3)$.
\end{itemize}
\end{lem}

\begin{proof}
Let $T$ be an unsolvable composition factor of the stabilizer $G_\alpha$. By
Lemma~\ref{comp-factor}, $T$ is a composition factor of $G_\alpha^{\Gamma(\alpha)}$,
$G_{\alpha\beta}^{\Gamma(\alpha)}$, or $G_{\alpha\beta}^{\Gamma(\beta)}$. Since $\Gamma$ is locally
$(G,2)$-arc-transitive, both $G_\alpha^{\Gamma(\alpha)}$ and $G_\beta^{\Gamma(\beta)}$ are
2-transitive. Inspecting the 2-transitive permutation groups and their stabilizers \cite[Table
2.1]{LiSeressSong}, we conclude that $S$ lies in the list in part~(i).

Now suppose that $S=\POmega_{n}^\pm(2)$ is a composition factor of $G_{\alpha\beta}$. Then $T$
appears as a composition factor of the stabilizer of the 2-transitive group
$G_\alpha^{\Gamma(\alpha)}$ or $G_\beta^{\Gamma(\beta)}$ with socle $\PSp_{n}(2)$. This proves
part~(ii).

The statement of part~(iii) follows from an inspection of the Classification of Finite
$2$-Transitive Groups \cite[Tables 7.3 and 7.4]{CameronPerm}.
\end{proof}

\section{A characterization of antiflag-transitive generalized quadrangles}

Throughout this section we will assume that $\mathcal{Q}=(\calP,\calL)$ is a thick finite
generalized quadrangle of order $(s,t)$ with incidence graph $\Gamma$. Abusing notation slightly, we
will also refer to the biparts of the bipartite graph $\Gamma$ as $\mathcal{P}$ and $\mathcal{L}$.
If $S$ is a set of points, then $S^{\perp\perp}$ is the set of all points at distance $0$ or $2$
from every point of $S$ (in the collinearity graph). Recall from Lemma
\ref{anti-f-trans-local-3-a-trans} that a generalized quadrangle $\mathcal{Q}$ having a group of
collineations $G$ that acts transitively on antiflags is equivalent to the incidence graph $\Gamma$
being locally $(G,3)$-arc-transitive. Hence we will use the conditions of $G$-antiflag-transitivity
and local $(G,3)$-arc transitivity interchangeably in this section.

\begin{thm}
\label{thm:3arcNtrans} If $\Gamma$ is locally $(G,3)$-arc-transitive for some group of collineations
$G$, then $G$ acts quasiprimitively on $\mathcal{P}$ or $\mathcal{L}$.
\end{thm}

\begin{proof}
Suppose that $G$ is a group of collineations of $\mathcal{Q}$ such that $\Gamma$ is locally
$(G,3)$-arc-transitive and that $G$ contains a normal subgroup $N$ that is intransitive on both
$\mathcal{P}$ and $\mathcal{L}$. By \cite[Lemma 5.1]{localsarc}, $N$ acts semiregularly on both
$\mathcal{P}$ and $\mathcal{L}$.

Let $\ell_1^N$,..., $\ell_k^N$ be the orbits of $N$ on $\mathcal{L}$, where
$\ell_1,\dots,\ell_k\in\calL$. Let $P \in \mathcal{P}$. Since $\Gamma$ is locally
$(G,3)$-arc-transitive, $G_P$ is $2$-transitive on $\Gamma(P)$, and so for each $i$, $P$ has at most
one neighbor in the orbit $\ell_i^N$. Without a loss of generality, we may assume that $P$ is
incident to $\ell \in \ell_1^N$. Let $1 \neq n \in N$. Since $P$ is incident with $\ell$, we have
that $P$ is not incident with $\ell^n$. Since $\Gamma$ is of girth $8$, there exists a unique
$3$-arc beginning at $P$ and terminating at $\ell^n$. On the other hand, since $|\ell_j^N| > 1$
where $1\leqslant j\leqslant k$, for some $m \in \ell_j^N$, there exists a unique $3$-arc beginning
at $P$ and terminating at $m$. Since $\Gamma$ is locally $(G,3)$-arc-transitive, there exists $g \in
G_P$ such that $m^g = \ell^n$. This implies that $P$ is incident with the line $m^{gn^{-1}g^{-1}}
\in \ell_j^N$. Hence $P$ is incident with exactly one line in each orbit, which means that $k = t +
1$ and $|N| = st + 1$. In particular, for any $P \in \mathcal{P}$, $P^N$ is an ovoid\footnote{An
\emph{ovoid} of a generalized quadrangle $\mathcal{Q}$ is a set $S$ of points such that every line
of $\mathcal{Q}$ is incident with exactly one element of $S$.}, and the point set $\calP$ is a union
of ovoids.

Fix a flag $\{P, \ell\}$ and let $(P, \ell, P_1, \ell_1)$ and $(P, \ell, P_2, \ell_2)$ be two
$3$-arcs beginning with the arc $(P, \ell)$. Neither $\ell_1$ nor $\ell_2$ can be incident with $P$
since $\Gamma$ has girth $8$. On the other hand, both $\ell_1$ and $\ell_2$ must be incident with a
point in $P^N$. They cannot be incident with the same point since the girth of $\Gamma$ is $8$.
Since there are exactly $st$ $3$-arcs beginning with the arc $(P, \ell)$ and $st$ points in $P^N
\backslash \{P\}$, for any $n \in N$, there exists a unique $4$-arc beginning with the arc $(P,
\ell)$ and ending at $P^n$. Note further that, if $(P,
\ell, P_0, \ell_0)$ is the unique $3$-arc such that $P^n$ is incident with $\ell_0$, we have that
$(P, \ell, P_0, \ell_0, P^n, \ell^n, ..., \ell_0^{n^{-1}})$ is a cycle of length $4k$, where $k$ is
the order of the element $n$. (Since $N$ acts semiregularly on both $\mathcal{P}$ and $\mathcal{L}$,
$P^{n^i} = P^{n^j}$, $\ell^{n^i} = \ell^{n^j}$, $P_0^{n^i} = P_0^{n^j}$, and $\ell_0^{n^i} =
\ell_0^{n^j}$ all imply that $n^i = n^j$.) This cycle is completely determined by the choice of
$3$-arc $(P, \ell, P_0, \ell_0)$, and, since $\Gamma$ is locally $(G,3)$-arc-transitive, every such
cycle must have the same length. However, this means that every nonidentity element of $N$ has the
same order. Hence $|N| = p^d$ for some prime $p$ and some integer $d$.

Now, fix a point $P$ and another point $P^n$ in the orbit $P^N$. By the arguments in the previous
paragraph, for each line $\ell_i$ incident with $P$, there is a unique point $P_i$ and line
$\ell_i'$ such that $(P, \ell_i, P_i, \ell_i', P^n)$ is a $4$-arc, and, since the girth of $\Gamma$
is 8, the $P_i$ are pairwise distinct. Hence the only points at distance $0$ or $2$ from both $P$ and $P^n$ 
are $\{P_1, P_2,..., P_{t+1} \}$ and hence $(P, P^n)$ is a so-called \emph{regular pair}
(see \cite[\S1]{fgq}). By the dual argument, we also obtain a regular pair of noncurrent
lines. By \cite[1.3.6(i)]{fgq}, if $\mathcal{Q}$ has a regular pair of noncollinear points and $s >
1$, then $t \le s.$ Hence $s = t$. By \cite[1.8.5]{fgq}, since the generalized quadrangle
$\mathcal{Q}$ of order $s$ has a regular pair of noncollinear points and the point set $\mathcal{P}$
can be partitioned into ovoids, $s$ must be odd. Since $st + 1 = s^2 + 1 = |N|=p^d$ and
$s$ is odd, we have that $p=2$ and $s^2 + 1 = 2^d$. On the other hand, if $d \ge 2$, then
we have that $s^2 \equiv 3 \pmod 4$, a contradiction. Hence $d = 1$ and $s = 1$, a final
contradiction. Therefore, if $\Gamma$ is locally $(G,3)$-arc-transitive, then every normal subgroup
$N$ of $G$ is transitive either on $\mathcal{P}$ or on $\mathcal{L}$.

Finally, \cite[Lemma 5.4]{localsarc} gives us immediately that $G$ acts quasiprimitively on at least
one of $\mathcal{P}$ or $\mathcal{L}$.
\end{proof}

\begin{remark}
We remark that the previous result is free of the Classification of Finite Simple Groups. The
following result does depend on the CFSG.
\end{remark}

\begin{thm}\label{thm:3arcLimprim}
Let $G$ be a group such that $\Gamma$ is locally $(G,3)$-arc-transitive. Assume that $G$ acts
quasiprimitively on $\mathcal{P}$ but does not act primitively on $\mathcal{L}$. Then $\mathcal{Q}$
is isomorphic to $\gqthreefive$.
\end{thm}

\begin{proof}
Let $\mathcal{B} = \{\Sigma_1, \Sigma_2, ..., \Sigma_m\}$ be a nontrivial block system of $G$ on
$\mathcal{L}$, and let $K = G_{(\mathcal{B})}$ be the kernel of $G$ acting on $\mathcal{B}$. Let $P
\in \mathcal{P}$. 
Since $G_P$ is primitive on $\Gamma(P)$, $P$ has at most one neighbor in each $\Sigma_i$.
Without a loss of generality, we may assume that $|\Gamma(P) \cap \Sigma_1| = 1$, and, since
$\mathcal{B}$ is a nontrivial block system, for each $i$ there exists $\ell_i' \in \Sigma_i$ such
that $(P, \ell_i')$ is an antiflag. Let $\ell \in \Gamma(P) \cap \Sigma_1$. Since $\Gamma$ is
locally $(G,3)$-arc-transitive, there exists $g_i \in G_P$ such that $(P, \ell_1')^{g_i} = (P,
\ell_i')$. This implies that $\ell^{g_i} \in \Sigma_i$ is incident with $P$, and hence $|\Gamma(P)
\cap \Sigma_i| = 1$ for each $i$ and $|\mathcal{B}| = t+1$. Moreover, $G_P$ is thus $2$-transitive
on $\mathcal{B}$ and $G^{\mathcal{B}}$ is a $2$-transitive permutation group. We let $\Gamma(P) =
\{\ell_1, ..., \ell_{t+1} \}$, where $\ell_i \in \Sigma_i$. Each $\Sigma_i$ necessarily must have
the same size, and so $|\Sigma_i| = |\mathcal{L}|/(t+1) = st+1$. Since $G_P$ is transitive on
$\mathcal{L}\backslash \Gamma(P)$, $(G_P)_{\Sigma_i}$ is transitive on $\Sigma_i \backslash
\{\ell_i\}$. Since $\Gamma(P) \cap \Sigma_i = \{\ell_i\}$, it follows that $(G_P)_{\Sigma_i} =
(G_{\Sigma_i})_P \le (G_{\Sigma_i})_{\ell_i}$. Thus $(G_{\Sigma_i})_{\ell_i}$ is transitive on
$\Sigma_i \backslash \{\ell_i \}$ and the block stabilizer $G_{\Sigma_i}$ is $2$-transitive on the
block $\Sigma_i$. Hence we have shown that the following hold:
\begin{itemize}
 \item[(i)] $|\mathcal{B}| = t+1$, and $G^{\mathcal{B}}$ is $2$-transitive;
 \item[(ii)] $|\Sigma_i| = st + 1$, and $G_{\Sigma_i}^{\Sigma_i}$ is $2$-transitive for $1 \le i
   \le t+1$;
 \item[(iii)] for $P \in \mathcal{P}$, the intersection $\Gamma(P) \cap \Sigma_i$ contains exactly
   one element for $1 \le i \le t+1$;
\end{itemize}

We now show that $G^{\mathcal{P}}$ is an affine group, i.e., $G$ has a unique elementary abelian
minimal normal subgroup that acts regularly on the point set $\mathcal{P}$. Suppose first that $K =
1$. Since $G^{\mathcal{B}} \cong G/K$, this implies that $G$ is a $2$-transitive permutation group.
If $G$ is an affine $2$-transitive group, then we are done. Otherwise, $G$ is an almost simple group
with a faithful $2$-transitive representation on $t+1$ elements, and the stabilizer $G_{\Sigma_1}$
is $2$-transitive on $\Sigma_1$ with degree $st+1$. Looking at the possibilities for $G$ (see
\cite[Table 7.4]{CameronPerm}), the only possibility with $s > 1$ is $\soc(G) \cong A_8$,
$G_{\Sigma} \cong A_7$, and $st + 1 = 15$. However, this implies that $s^2 = 4 < 7 = t$, a
contradiction.

We may thus assume that $K \neq 1$, and assume further that $G^{\mathcal{P}}$ is not affine. Let $M$
be a minimal normal subgroup of $G$ contained in $K$. Then $M$ is nonabelian. Since $M^{\Sigma_i}$
is a (minimal) normal subgroup of the $2$-transitive permutation group $G_{\Sigma_i}^{\Sigma_i}$, we
conclude that $G_{\Sigma_i}^{\Sigma_i}$ is almost simple and that $M^{\Sigma_i}$ is a simple group.
Suppose that $M_{(\Sigma_1)} \neq 1.$ Then $M_{(\Sigma_1)}$ has a nontrivial action on some
$\Sigma_i$ and $M_{(\Sigma_1)}^{\Sigma_i}$ is thus a nontrivial normal subgroup of $M^{\Sigma_i}$,
which implies that $M_{(\Sigma_1)}$ is transitive on $\Sigma_i$. However, for a given $\ell \in
\Sigma_1$, there are exactly $s+1 < st + 1$ lines in $\Sigma_i$ that meet $\ell$ in a point, and
$M_{(\Sigma_1)}$ cannot be transitive on $\Sigma_i$, a contradiction. Thus $M_{(\Sigma_1)} = 1$, and
hence for all $i$, $M^{\Sigma_i}$ is a simple group. That is to say, the quasiprimitive permutation
group $T=G^{\mathcal{P}}$ has a normal subgroup $M^{\mathcal{P}}$ which is nonabelian simple.
Therefore, either
\begin{itemize}
\item[(a)] $\C_{G^{\mathcal{P}}}(T)=1$, and $G^{\mathcal{P}}$ is almost simple, or

\item[(b)] $\C_{G^{\mathcal{P}}}(T)\cong T$, and $\soc(G)=T\times T$.
\end{itemize}

Suppose that $G$ is almost simple. Then $G/M$ is a subgroup of $\Out(M)$ and has a $2$-transitive
permutation representation representation of degree $t+1$, and $G_{\Sigma_1}$ is $2$-transitive on
$\Sigma_1$ of degree $st + 1 \le t^3 + 1$. By the Classification of Finite Simple Groups, this is
not possible.

The latter case thus occurs. Since $G_{\Sigma_1} \rhd K\geqslant T$ and $G_{\Sigma_1}$ is a
$2$-transitive permutation group of $\Sigma_1$ of degree $st+1$, and $G/K$ is a $2$-transitive
permutation group on $\mathcal{B}$ of degree $t+1$, we have that $G_{\Sigma_1}$ and $G/K$ have
isomorphic socle $M$. Inspecting the $2$-transitive permutation groups with nonabelian socles (see
\cite[Table 7.4]{CameronPerm}), this is not possible. Therefore, in any case, $G$ is an affine group
on $\mathcal{P}$, as desired.

Since $G$ is acts primitively of affine type on the points, and acts transitively on the lines of
$\mathcal{Q}$, we can immediately deduce from \cite[Corollary 1.5]{GQtranshyp} that $\mathcal{Q}$ is
isomorphic to one of $\gqthreefive$ or the generalized quadrangle of order $(15,17)$ arising from
the Lunelli--Sce hyperoval. However, the latter is not antiflag-transitive, and therefore
$\mathcal{Q}$ is $\gqthreefive$.
\end{proof}

\begin{lem}
\label{lem:3arcQuasi=Prim} Assume that $G$ is quasiprimitive on both $\mathcal{P}$ and
$\mathcal{L}$. Then $G$ acts primitively on both $\mathcal{P}$ and $\mathcal{L}$.
\end{lem}

\begin{proof}
Assume that $G$ does not act primitively on $\mathcal{L}$. By Theorem \ref{thm:3arcLimprim}, $G$
must an affine group on $\mathcal{P}$. Since $G$ acts quasiprimitively on both $\mathcal{P}$ and
$\mathcal{L}$ and $\Gamma$ is locally $(G,3)$-arc-transitive, by \cite[Theorem 1.2]{localsarc}, $G$
is affine on $\mathcal{L}$. However, affine groups that act quasiprimitively on a set act
primitively on that set \cite[2.I]{quasiprim1}, a contradiction. Therefore, $G$ must act primitively
on $\mathcal{L}$, and, by duality, $G$ must therefore act primitively on both $\mathcal{P}$ and
$\mathcal{L}$.
\end{proof}

\begin{thm}
\label{thm:characterization} If $G$ is a group of collineations that is transitive on antiflags of
$\mathcal{Q}$, and $\mathcal{Q}$ is not $\gqthreefive$ or its dual, then $G$ is an almost simple
group of Lie type that acts primitively on both points and lines.
\end{thm}

\begin{proof}
Suppose that $G$ is a group of collineations that is transitive on the antiflags of $\mathcal{Q}$.
By Theorem \ref{thm:3arcNtrans}, $G$ must act quasiprimitively on $\mathcal{P}$ or $\mathcal{L}$.
Since $\mathcal{Q}$ is not $\gqthreefive$ or its dual, by Theorem \ref{thm:3arcLimprim}, $G$ must
act quasiprimitively on both $\mathcal{P}$ and $\mathcal{L}$. By Lemma \ref{lem:3arcQuasi=Prim}, $G$
acts primitively on both $\mathcal{P}$ and $\mathcal{L}$. By \cite[Corollary 1.3]{primGQ}, $G$ is an
almost simple group of Lie type.
\end{proof}

The following results give us bounds on the size of a point stabilizer.

\begin{prop}
\label{prop:stabP} Let $G$ be a group that is transitive on the antiflags of a finite thick
generalized quadrangle $\mathcal{Q}$, assume $s \le t$, and let $P \in \mathcal{P}$. If $s < t$,
then $|G| < |G_P|^2$. If $s = t$, then $|G| < |G_P|^{19/9}$.
\end{prop}

\begin{proof}
Note first that since $G$ is antiflag-transitive, this means that $G_P$ is transitive on $3$-arcs
that begin at $P$ in the incidence graph, i.e., $|G_P| \ge st(t+1)$. Assuming that $s < t$, we have:
$$|G_P|^2 \ge |G_P|st(t+1) \ge |G_P|s(s+1)(t+1) > |G_P|(s+1)(st+1) = |G|.$$ On the other hand, if $s
= t$, we note first that since $s \ge 2$, $s^{1/3} > 1 + 1/s^2$. Hence we have: $$|G_P|^{19/9} \ge
|G_P| s^2(s+1) \left( s^2(s+1) \right)^{\frac{1}{9}} > |G_P| s^2(s+1) s^{\frac{1}{3}} >
|G_P|s^2(s+1)\left(1 + \frac{1}{s^2}\right) = |G|,$$ as desired.
\end{proof}

\begin{cor}
\label{cor:largesub} Let $G$ be an almost simple group of Lie type that acts transitively on the
antiflags of a finite generalized quadrangle $\mathcal{Q}$ with $s \le t$, and let $T = \soc(G)$. If
$P \in \mathcal{P}$, then $|T| < |T_P|^3$.
\end{cor}

\begin{proof}
Let $G = T.A$, where $A$ is the group of outer automorphisms of $T$. By the Classification of Finite
Simple Groups, we know that asymptotically $|A|$ is $O(C \log(|T|))$ for some constant $C$, and in
fact we have that $|A| < 2 \log_2(|T|)$ \cite{AutTBd}. Note first that when $|T| \ge 1006796$ that
$$|A|^{\frac{15}{4}} < \left(2 \log_2(|T|) \right)^{\frac{15}{4}} < |T|,$$ which implies that
$$|A|^{\frac{19}{4}} < |G| < |G_P|^{\frac{19}{9}}$$ by Proposition \ref{prop:stabP}, and hence
$|A|^\frac{9}{4} < |G_P|$ and $|A|^{\frac{5}{4}} < |T_P|$. Again using Proposition \ref{prop:stabP},
we have: $$|T| < |T_P|^{\frac{19}{9}}|A|^{\frac{10}{9}} < |T_P|^{\frac{19}{9}}|T_P|^{\frac{8}{9}} =
|T_P|^3,$$ as desired.

This implies that if $G$ acts transitively on the antiflags of a generalized quadrangle and $|T_P|^3
\le |T|,$ then $|T| \le 1006795$. By Proposition \ref{prop:stabP}, this means that $$|T_P|^3 \le |T|
< |T_P|^{2.5} |A|^{1.5}.$$ Using \GAP\ \cite{GAP4}, we find the only possibility for $(T, T_P)$ is
$(\PSU_3(4), C_{13}{:}C_3)$. However, this choice of $T_P$ does not have a $2$-transitive action,
and therefore $|T| < |T_P|^3$.
\end{proof}

Following Alavi and Burness \cite{largesubs}, a proper subgroup $H$ of a group $G$ is said to be
\emph{large} if the order of $H$ satisfies the bound $|H|^3 < |G|$. Corollary \ref{cor:largesub}
shows that the point stabilizer $T_P$ is a large subgroup of $T$, where $T$ is the socle of an
almost simple group of Lie type acting transitively on the antiflags of a generalized quadrangle. In
the forthcoming sections, we will utilize the work of \cite{largesubs} that determined all the large
maximal subgroups of finite simple groups.


\section{Geometric maximal subgroups of classical groups}

In view of Theorem \ref{thm:characterization}, in order to classify the finite thick
antiflag-transitive generalized quadrangles, we need to show that all antiflag-transitive
generalized quadrangles arising from almost simple groups of Lie type with large point stabilizer
are classical or dual classical. Corollary \ref{cor:largesub} gives an immediate restriction on the
possibilities for the point stabilizer, and throughout the following sections we will assume the
following:

\begin{hyp}
\label{hyp:3arcAlmostSimple} Let $G$ be an almost simple group of Lie type with $T := \soc(G)$
acting primitively on the points and primitively on the lines of a finite generalized quadrangle
$\mathcal{Q}$ of order $(s,t)$, $s \le t$, where, for any point $P$ and line $\ell$, $T_P$ acts
$2$-transitively on lines incident with $P$, $T_\ell$ acts $2$-transitively on points incident with
$\ell$, and $|T| < |T_P|^3$.
\end{hyp}

The following lemma will be important when considering the indices of parabolic subgroups of
classical groups.

\begin{lem}[{\cite[Lemma 9.4.1]{DRG}}]
\label{lem:isotropics} Let $k$ be a natural number, and let $e$ be $1,1,0,2,\frac{3}{2},
\frac{1}{2}$ in the respective cases of $\Sp_{2d}(q), \Omega_{2d+1}(q), \Omega_{2d}^+(q),
\Omega_{2d}^-(q), U_{2d+1}(q),$ and $U_{2d}(q)$ acting on the natural vector space $V$ equipped with
the appropriate quadratic or bilinear form. Then the number of totally singular $k$-spaces in $V$ is
${d \brack k} \prod_{i=1}^k (r^{d+e - i} + 1)$, where, in the first four cases $r = q$ and in the
last two cases $r = q^2$ and ${d \brack k}$ is the $q$-ary Gaussian binomial coefficient.
\end{lem}

In each case below, we will use the notation afforded by Aschbacher's classification of the
subgroups of classical groups. That is, a maximal subgroup of a classical group will fall into one
of nine classes $\mathcal{C}_i$ ($i=1,\ldots,9$) in the standard way (see \cite[2.1.2,
2.1.3]{BrayHoltRoney-Dougal}). The first eight classes are the \emph{geometric subgroups}, whereas
$\mathcal{C}_9$ contains the `nearly simple subgroups' that do not fall into the other classes.

\subsection{Linear groups}

\begin{prop}
\label{prop:linearmaxlgeom} Assume Hypothesis {\rm \ref{hyp:3arcAlmostSimple}}. If $T$ is isomorphic
to $\PSL_n(q)$ for some $n,q \ge 2$, then the point stabilizer $T_P$ cannot be a maximal geometric
subgroup.
\end{prop}

\begin{proof}
Recall that $s\leqslant t$ by our assumption. By Corollary \ref{cor:largesub}, the stabilizer of a
point $T_P$ satisfies $|T| < |T_P|^3$. The geometric subgroups satisfying this condition are listed
in \cite[Proposition 4.7]{largesubs}, and we proceed down the list.

\begin{description}[leftmargin=0pt, style=unboxed, itemsep=1ex]
\item[Case 1] We first consider subgroups of type $\mathcal{C}_1$. We consider first the parabolic
  subgroups $P_i$ of $T$, where $P_i$ is the stabilizer of an $i$-dimensional subspace. Without a
  loss of generality, we may assume that $i \le n - i$. Then the stabilizer $G_P$ satisfies
\[T_P = q^{i(n-i)}\colon {1\over \gcd(n,q-1)}(\GL_i(q)\circ \GL_{n-i}(q)),\]
and the order of the point set is
\[|\mathcal{P}|={n\brack i}={(q^n-1)\cdots(q^{n-i+1}-1)\over(q^i-1)\cdots(q-1)}.\]
Noting that $G$ has rank $3$ in its action on $\mathcal{P}$ and $T$ is $2$-transitive on $1$-dimensional subspaces, we have $i\ge2$, and then
in particular, $n\geqslant4$.

Since $q^{i(n-i)}$ is the unique minimal normal subgroup of $G_P$
and $G_P^{\Gamma(P)}$ is 2-transitive, it follows from the Classification of Finite (affine) $2$-Transitive
Groups that $G_P^{\Gamma(P)}$ is almost simple.
By Lemma~\ref{stab}, we conclude that $G_P^{\Gamma(P)}\rhd\PSL_{n-i}(q)$. Note that we may also conclude that $t+1=|\Gamma(P)|=(q^{n-i}-1)/(q-1)$, since the other possible values of $t+1$ are ruled out by Table \ref{tbl:Pi}.
By Lemma~\ref{lem:Pbounds},
 \begin{align*}
 q^{i(n-i-1)} &= \frac{q^{n-1} \cdots q^{n-i}}{q^i \cdots q^1} < \frac{(q^n-1)\cdots(q^{n-i+1}-1)}{(q^i-1)\cdots(q-1)}= |\mathcal{P}|\\
 &< (t+1)^3=\frac{(q^{n-i} - 1)^3}{(q-1)^3}< q^{3n - 3i},
\end{align*}
which implies that $i(n-i -1) < 3n - 3i$.  Since $i \le n - i$, we have $2i \le n < i(i-2)/(i-3)$ if $i > 3$, which implies that $i < 4$, a contradiction.  If $i = 3$, then we have
$$\frac{(q^n-1)(q^{n-1}-1)(q^{n-2}-1)}{(q^3-1)(q^2-1)(q-1)} < \frac{(q^{n-1}-1)^3}{(q-1)^3},$$ which implies that
$$q^6 < \frac{(q^n-1)(q^{n-1}-1)(q^{n-2}-1)}{(q^{n-3} - 1)^3} < \frac{(q^3-1)(q^2-1)(q-1)}{(q-1)^3} = (q+1)(q^2+q+1),$$
a contradiction for $q \ge 2$.  This leads to $i=2$;
in particular, $n\geqslant 2i=4$.
Thus $T_P \cong P_2=q^{2(n-2)}\colon {1\over\gcd(n,q-1)}(\GL_2(q)\circ\GL_{n-2}(q))$,
$T_P^{\Gamma(P)}=\PGL_{n-2}(q)$, and $|\mathcal{P}| ={(q^n-1)(q^{n-1}-1)\over(q^2-1)(q-1)}$.

\begin{center}
\begin{table}
\caption{Ruling out examples for Case 1 of the proof of Proposition \ref{prop:linearmaxlgeom}.}\label{tbl:Pi}
\begin{tabular}{l|l|c|c|c|p{4cm}l}
\toprule
$(n,q)$ & $(n-i,q)$ & $t+1$ & $\mathcal{P}$ & $|\mathcal{P}| < (t+1)^3$? & Integral solution to $(s+1)(st+1) = |\mathcal{P}|$?\\
\midrule
$(4,4)$ &  $(2,4)$ & $6$ &   $357$ & false & \\
$(4,5)$ & $(2,5)$ & $5$ &   $806$ & false & \\
$(4,7)$ &  $(2,7)$ &  $7$ &   $2850$ & false & \\
$(4,8)$ & $(2,8)$ & $28$ & $4745$ & true & false\\
$(4,9)$ &  $(2,9)$ & $6$ & $7462$ & false & \\
$(4,11)$ & $(2,11)$ & $11$ & $16226$ & false & \\
$(5,2)$ & $(3,2)$ & $8$ & $155$ & true & false\\
$(6,2)$ & $(3,2)$ & $8$ & $1395$ & false & \\
$(6,2)$ & $(4,2)$ & $8$ & $651$ & false & \\
$(7,2)$ & $(4,2)$ & $8$ & $11811$ & false & \\
$(8,2)$ & $(4,2)$ & $8$ & $200787$ & false & \\
\bottomrule
\end{tabular}
\end{table}
\end{center}

Note that in this case we have $t+1 = (q^{n-2} - 1)/(q-1)$ since exceptions are ruled out by Table \ref{tbl:Pi}.  Furthermore, the flag stabilizer is $G_{P,\ell}=[q^{2(n-2)+(n-3)}]\colon ({1\over(n,q-1)}(\GL_2(q)\circ(\GL_1(q)\times\GL_{n-3}(q)))$.
The line stabilizer $G_\ell$ is a maximal subgroup of $G$ which contains $G_{P,\ell}$.
It follows that $G_\ell$ is a parabolic subgroup of type $P_2$ or $P_3$.

If $G_\ell$ is parabolic of type $P_3$, then we have $$|\mathcal{L}| = {n \brack 3} = \frac{(q^n-1)(q^{n-1}-1)(q^{n-2}-1)}{(q^3-1)(q^2-1)(q-1)}.$$  Noting that $|\mathcal{L}| = (t+1)(st+1)$ and $t+1 = (q^{n-2} - 1)/(q-1)$, this implies that $$st + 1 = \frac{(q^n-1)(q^{n-1}-1)}{(q^3-1)(q^2-1)},$$ and since $|\mathcal{P}| = (s+1)(st+1)$, we have $s+1 = q^2 + q + 1$.  On the other hand, this means that $$\frac{(q^n-1)(q^{n-1}-1)}{(q^3-1)(q^2-1)} = st + 1 = (q^2 + q)\left(\frac{q^{n-2} - q}{q-1}\right) + 1,$$ which in turn implies that $$(q^2 + q)(q^3 - 1)(q+1)(q^{n-2}-q) + (q^3 - 1)(q^2-1) = (q^n - 1)(q^{n-1} - 1).$$  From here, we see that $$(q^2 + q)(q^3 - 1)(q+1)(-q) + (q^3 - 1)(q^2-1) \equiv 1 \pmod {q^{n-2}}.$$  Simpliflying, we see that $q^3(1+q-2q^3-q^4) \equiv 0 \pmod {q^{n-2}}$.  Since $G_{\ell}$ has type $P_3$, $n \ge 6$; however, this means that $q^4$ divides $q^3(1+q-2q^3-q^4)$ and $q$ divides $1+q-2q^3-q^4$, a contradiction.  Hence $G_\ell$ must be of type $P_2$.

If $G_\ell$ is parabolic of type $P_2$, then we have $|\mathcal{L}| = |\mathcal{P}|$, and so $s+1 = t+1 = (q^{n-2}-1)/(q-1)$. Note first that since there are $P_2$ subgroups, $n \ge 4$.
By Lemma \ref{lem:Pbounds} (ii), $s^2(t+1) < |\mathcal{P}|$,
which here means that
$$\left( \frac{q^{n-2}-q}{q-1}\right)^2\frac{q^{n-2}-1}{q-1} < \frac{(q^n-1)(q^{n-1}-1)}{(q^2-1)(q-1)}.$$
Simplifying, this means that
$$q^2(q^{n-3}-1)^2 < (q^n-1)\frac{q^{n-1}-1}{q^{n-2}-1} < (q^n-1)(q +1),$$
which in turn implies that
$$q^{n-5} - 1 < q^{n-5} - \frac{1}{q^2} < \frac{1}{q^4}\left( \frac{q^n-1}{q^{n-3}-1} \right)(q+1) < \frac{1}{q^4}(q^3+1)(q+1) = (1 + \frac{1}{q^3})(1+\frac{1}{q}) < 2.$$
If $n > 5$, this last equation only holds if $n=6$ and $q=2$;
however, going back to our original inequality of $s^2(t+1) < |\mathcal{P}|$, this implies that
$$2940 =  \left(\frac{2^4 - 2}{2-1}\right)^2 \frac{2^4-1}{2-1} < \frac{(2^6-1)(2^5-1)}{(2^2-1)(2-1)} = 651,$$
a contradiction.
We may now examine the cases of $n=4$ and $n=5$ individually by examining the exact value of $|\mathcal{P}|$.  On the one hand, we know that $|\mathcal{P}| ={(q^5-1)(q^{4}-1)\over(q^2-1)(q-1)}$;
on the other hand, $|\mathcal{P}| = (s+1)(st+1),$ and so
$$\frac{(q^5-1)(q^4-1)}{(q^2-1)(q-1)} = \left(\frac{q^3-1}{q-1}\right)\left(\left( \frac{q^3-q}{q-1}\right)^2 +1\right),$$
a contradiction since a primitive prime divisor of $q^3 -1$ does not divide the left hand side, and so $n \neq 5$.  Finally, if $n = 4$, $s+1 = t+1 = q+1$, and proceeding as in the $n=5$ case we find that $$(q^2 + 1)(q^2+q+1) = \frac{(q^4-1)(q^3-1)}{(q^2-1)(q-1)} = |\mathcal{P}| = (q+1)(q^2+1),$$ which implies that $q = 0$, a contradiction.  Therefore, $T_P$ cannot be a parabolic subgroup.

\item[Case 2] Suppose $T_P$ is a $\mathcal{C}_2$-subgroup of type $\GL_{n/t}(q) \Wr S_t$.  We will first rule out $t = 3$.  By \cite[Proposition 4.7]{largesubs}, we have that either $q \in \{5,8,9\}$ and $(n, q-1) = 1$ or $(n,q) = (3,11)$.  Note when $q = 9$ we have $\Out(T) \cong C_2 \times C_2$, and otherwise $\Out(T) \cong C_2$.  If $(n,q) = (3,11)$, then $T_P \cong C_{10}^2{:}S_3$ and $G_P = C_{10}^2{:}S_3.2$.  This implies that $t+1 \le 3$, but then $(t+1)^3 < |T{:}T_P|,$ a contradiction to Lemma \ref{lem:Pbounds}(i).  Otherwise, by \cite[Proposition 4.2.9]{KleidmanLiebeck},
$$T_P = [(q-1)^2].\PSL_{n/3}(q)^3.S_3.$$  If $n \ge 6$, then $G_P$ is not solvable, and so $\PSL_{n/3}(q)^3$ is in the kernel of the action of $T_P$ on $\Gamma(P)$.  However, this means that $t+1 \le (q-1)^2$ and thus, using \cite[Corollary 4.3]{largesubs}. $$(t+1)^3 < (q-1)^6 < \frac{1}{6}q^{\frac{2n^2}{3} - 1} < |T{:}T_P| = |\mathcal{P}|,$$ a contradiction to Lemma \ref{lem:Pbounds}(i).  Hence we may assume that $n=3$, in which case $T_P \cong (q-1)^2{:}S_3$.  This implies that $t+1 \le q-1$, but then $(t+1)^3 < q^3 < |T{:}T_P| =|\mathcal{P}|$, a contradiction to Lemma \ref{lem:Pbounds}(i), and hence $t \neq 3$.

Hence we have that $T_P$ is of type $\GL_{n/2}(q) \Wr S_2$.
Suppose that $G_P$ is not solvable.
Then $\SL_{n/2}(q) \circ \SL_{n/2}(q)$ must be in the kernel of
the action of $T_P$ on $\Gamma(P)$, and $G_P^{\Gamma(P)}$ is a factor group of
$(C_{q-1}\times C_{q-1}).[2f]$, where $q = p^f$ for some prime $p$ and integer $f$.
Since $G_P^{\Gamma(P)}$ is 2-transitive, it implies that $G_P^{\Gamma(P)}\cong C_{t+1}\colon C_t$
where $t+1$ divides $q-1$ and $t$ divides $2f$.
Thus the valency $|\Gamma(P)|=t+1<\min\{q-1,2f + 1\}$, which is smaller than the smallest permutation degree for
$\PSL_{n/2}(q)$ (namely $(q^{n/2}-1)/(q-1)$), a contradiction.
Thus $G_P$ is solvable, and either $n=2$, or $n=4$ and $q=2$ or $3$.

We first examine the case $n=2$.  In this case, $T_P \cong D_{q-1}$ if $q$ is odd and $T_P \cong D_{2(q-1)}$ if $q$ is even.  This implies (in either case) that $|\mathcal{P}| = q(q+1)/2$.  Furthermore, if $q = p^f$, then $G_P$ is a subgroup of $D_{2(q-1)}.C_f$, and $G_P$ contains all of $T_P$ and the full group of outer automorphisms in $G$.  On the other hand, since $G_P$ acts $2$-transitively on $\Gamma(P)$, by the Classification of the Finite (affine) $2$-Transitive Groups, we have $G_P^{\Gamma(P)} \cong C_{t+1}\colon C_t$, which implies that $t \le 2f$ (and $t+1 \le q-1$).  Hence by Lemma \ref{lem:Pbounds}(i), we have $$\frac{p^f(p^f+1)}{2} = |\mathcal{P}| < (2f + 1)^3.$$  This implies that the only possible values of $p$ and $f$ are: $p=2$ and $f \le 6$; $p = 3$ and $f \le 2$; $p=5$ and $f = 1$.  If $f=1$, then $t+1 \le 3$, which implies that $\mathcal{Q}$ is the unique generalized quadrangle of order $(2,2)$ (see Lemma \ref{lemma:smallGQ}); however, $G_P$ is a subgroup of $D_{2(q-1)}.C_f$ and the generalized quadrangle of order $(2,2)$ has point stabilizer isomorphic to $C_2 \times S_4$, ruling this case out.   If $f = 2$, then $p=2$ implies that $t+1 \le q-1 \le 3$, a contradiction as in the case $f=1$, and $p=3$ implies that $G_P$ is isomorphic to $D_{16}.2$, which has no $2$-transitive representations on more than $2$ elements, a contradiction.  Hence $p=2$ and $3 \le f \le 6$.  If $f=3$, then $G_P$ is a subgroup of $D_{14}.C_3$ containing $D_{14}$.  The only possibility here is $G_P^{\Gamma(P)} \cong C_7\colon C_6$, which implies that $t+1 = 7$.  However, this means that $$(t+1)^2 = 49 > 36 = |\mathcal{P}|,$$ a contradiction to Lemma \ref{lem:Pbounds}(i).  If $f=4$, then $G_P$ is a subgroup of $D_{30}.C_4$.  The only possibilities here are: $G_P^{\Gamma(P)} \cong C_3\colon C_2$, which is a contradiction as in the $f=1$ case; or $G_P^{\Gamma(P)} \cong C_5\colon C_4$, which implies that $t+1 = 5$ and $$(t+1)^3 = 125 < 136 = |\mathcal{P}|,$$ a contradiction to Lemma \ref{lem:Pbounds}(i).  If $f = 5$, then $G_P$ is a subgroup of $D_{62}.C_5$, which has no subgroups with $2$-transitive representations on more than $2$ elements, a contradiction.  Finally, if $f=6$, then $G_P$ is a subgroup of $D_{126}.C_6$.  The only possibility here is that $G_P^{\Gamma(P)} \cong C_7\colon C_6$, which implies that $t+1 = 7$, a contradiction since then $(t+1)^3 < |\mathcal{P}|$ as in the $f = 4$ case.  Therefore, $n \neq 2$.

Finally, we examine the cases of $n=4$ and $q=2$ or $q=3$.  If $q=2$, then by \cite[Table 3.5.H]{KleidmanLiebeck}, we have that a $\mathcal{C}_2$-subgroup of type $\GL_2(2) \Wr S_2$ is contained in the $\mathcal{C}_8$-subgroup $\Sp_4(2)$, a contradiction to maximality.  If $q = 3$, then $T_P \cong (\SL(2,3)\circ \SL(2,3)).2^2$, which implies that $|\mathcal{P}| = |T:T_P| = 10530.$  However, $G_P$ is then a subgroup of $(\SL(2,3)\circ \SL(2,3)).2^2$, which implies that $t+1 \le 4$ and $$(t+1)^3 \le 64 < 10530 = |\mathcal{P}|,$$ a contradiction.  Therefore, $T_P$ cannot be a maximal subgroup of type $\mathcal{C}_2$.

\item[Case 3]
Assume that $T_P$ is a $\mathcal{C}_3$-subgroup of type $\GL_{n/k}(q^k)$, where $k>1$.
Suppose first that $k<n$.
Then (by the Classification of Finite $2$-Transitive Groups \cite[Tables 7.3 and 7.4]{CameronPerm}) either
\begin{itemize}
\item[(a)] the valency $t+1=|\Gamma(P)|={q^n-1\over q^k-1}$, or

\item[(b)] $(n/k,q^k,t+1)=(2,9,6)$, or $(2,8,28)$, so that $(n,q)=(4,2)$, $(4,3)$, or $(6,2)$.
\end{itemize}

In the first case, $t={q^n-1\over q^k-1}=q^k{q^{n-k}-1\over q^k-1}$,
which implies that $st+1$ is coprime to $q$.
Therefore,  $s+1$ is divisible by the $q$-part of $(s+1)(st+1)=|\mathcal{P}|=|T:T_P|$, which is equal to ${q^{n(n-1)/2}/(q^k)^{{n\over k}({n\over k}-1)/2}}=q^{(k-1)n^2\over2k}$.
It implies that $s>t$, not possible.

In the exceptional case, $(n,q)$ is equal to $(4,2)$, $(4,3)$, or $(6,2)$.
 If $(n,q) = (4,3)$ and $t+1 \neq (q^n - 1)/(q^2 - 1)$, then $t+1 = 6$ and
 $$|\mathcal{P}| = |T\colon T_P| = 8424 > 216 = (t+1)^3,$$
 a contradiction to Lemma \ref{lem:Pbounds}(i).
If $(n,q)=(4,2)$ and $t+1 \neq (q^n - 1)/(q^2 - 1)$, then $t+1 = 6$ and
\[
|\mathcal{P}| = |T\colon T_P| = 56,
\]
which implies that
\[
(s+1)(5s + 1) = 56.
\]
This equation has no integral solution for $s$, a contradiction. Likewise, if $(n,q)=(6,2)$, then
$|T_P| = 10584$, $|T| = 20158709760$, and $t+1 = 28$, which implies that $$(s+1)(27s + 1) =
|\mathcal{P}| = |T{:}T_P| = 1904640.$$ This equation has no integral solution for $s$, a
contradiction.

 We thus have $k=n$. Note that by \cite[Proposition 4.7]{largesubs}, either $n = k = 2$ or $n = k =
 3$. Suppose that $n = k = 3$. By Proposition \ref{prop:stabP}, we have $q = 5,8,9,$ or $11$.
 However, in each of the cases $q = 5,8,9,11$, by \cite[Table 8.3]{BrayHoltRoney-Dougal} we have
 that $T_P = (q^2 + q + 1): 3$, which implies that $t+1 \le q^2 + q + 1$. This means that $$(t+1)^3
 \le (q^2 + q + 1)^3 < q^7 < |\PSL_3(q)|$$ by \cite[Corollary 4.3]{largesubs}, a contradiction to
 Lemma \ref{lem:Pbounds}(i).

We therefore have $n=2$. Then $T_P$ is isomorphic to $D_{q+1}$ if $q$ is odd and is isomorphic to
$D_{2(q+1)}$ if $q$ is even. So $G_P$ is a subgroup of $D_{2(q+1)}.C_f$, where $q=p^f$, and we
proceed as in the $\mathcal{C}_2$ case above. Immediately, this implies that $G_P^{\Gamma(P)} \cong
C_{t+1}\colon C_t$, and so $t \le 2f$ and $t+1 \le 2f+1$. Moreover, $|\mathcal{P}| = |T\colon T_P| =
p^f(p^f - 1)/2$, and so by Lemma \ref{lem:Pbounds}(i) we have that $$\frac{p^f(p^f-1)}{2} =
|\mathcal{P}| < (2f+1)^3.$$ This implies that the only possible values of $p$ and $f$ are: $p=2$ and
$f \le 6$; $p = 3$ and $f \le 2$; $p=5$ and $f = 1$; $p=7$ and $f=1$. If $f=1$, then $t+1 \le 3$,
which implies that $\mathcal{Q}$ is the unique generalized quadrangle of order $(2,2)$ (see Lemma
\ref{lemma:smallGQ}); however, $G_P$ is a subgroup of $D_{2(q+1)}.C_f$ and the generalized
quadrangle of order $(2,2)$ has point stabilizer isomorphic to $C_2 \times S_4$, ruling this case
out. If $f=2$, then $p=2$ implies that $G_P$ is a subgroup of $D_{10}.2$, which would imply that
$G_P^{\Gamma(P)} \cong C_5\colon C_4$ and $t = 4$. However, this means that $(s+1)(4s+1) =
|\mathcal{P}| = 2^2(2^2 - 1)/2 = 6,$ a contradiction to $s$ being an integer. If $f=2$, then $p = 3$
implies that $G_P$ is a subgroup of $D_{20}.2$, and so this again implies that $G_P^{\Gamma(P)}
\cong C_5\colon C_4$ and $t = 4$. However, this means that $(s+1)(4s+1) = |\mathcal{P}| = 36$, which
has no integral solutions for $s$, a contradiction. Hence $p=2$ and $2 \le f \le 6$. If $f = 3$ or
$f=4$, then $G_P$ is a subgroup of one of $D_{18}.C_3$ or $D_{34}.C_4$. The only possibility here is
that $G_P^{\Gamma(P)} \cong C_3\colon C_2$, but then $t+1 = 3$, which is a contradiction as above.
If $f = 5$, then $G_P$ is isomorphic to a subgroup of $D_{66}.C_5$, which means that
$G_P^{\Gamma(P)} \cong C_3\colon C_2$ (ruled out as above) or $G_P^{\Gamma(P)} \cong C_{11}\colon
C_{10}$. This means that $t = 10$ and that $(s+1)(10s+1) = |\mathcal{P}| = 496$, which has no
integral solutions for $s$, a contradiction. Finally, if $f=6$ and $p=2$, we have that $G_P$ is a
subgroup of $D_{130}.C_6$. Here, $|\mathcal{P}| = 2016$. The possibilities for $G_P^{\Gamma(P)}$ are
$C_5\colon C_4$, which is impossible since then $(t+1)^3 = 125 < 2016 = |\mathcal{P}|$, a
contradiction to Lemma \ref{lem:Pbounds}(i), or $C_{13}\colon C_{12}.$ However, this means that $t =
12$, and so $(s+1)(12s+1) = |\mathcal{P}| = 2016$, which has no integral solutions for $s$, a
contradiction. Therefore, $G_P$ cannot be a $\mathcal{C}_3$-subgroup.

\item[Case 4] Suppose now that $T_P$ is a $\mathcal{C}_5$ subgroup of type
  $\GL_n(q^{\frac{1}{k}})$, where $k>1$. Then either the valency $t+1={q^{n\over k}-1\over q^{1\over
  k}-1}$, or one of the exceptional cases in Table~\ref{PSL-C_5} occurs. For the former,
  $t=q^{1\over k}{q^{(n-1)\over k}-1\over q^{1\over k}-1}$. Thus $st+1$ is coprime to $q$. The
  number of points $(s+1)(st+1)=|\mathcal{P}|=|T:T_P|$ is divisible by
\[{q^{n(n-1)/2}\over (q^{1/k})^{n(n-1)/2}} = (q^{1-{1\over k}})^{n(n-1)/2},\]
and so is $s+1$.
Since $s\leqslant t$, we conclude that $k=n=2$, and $t=q^{1\over2}$ and $s=q^{1\over2}-1$.
In this case, $T=\PSL_2(q)$ and $T_P=\PSL_2(q^{1\over2})$.
Therefore, $q^{1\over2}(q-q^{1\over2}+1)=(s+1)(st+1)=|\mathcal{P}|=|T:T_P|=q^{1\over2}(q+1)$, a contradiction.

The exceptional cases are ruled out as shown in Table~\ref{PSL-C_5}.

Hence $T_P$ must be solvable. So we must have $n=2$ and $q = 4,9$.  If $q=4$, then $G_P \le S_4$.  We rule out $t+1 \le 3$ as above, and conclude that $t+1 = 4$.  However, in this case $$(s+1)(3s+1) = |\mathcal{P}| = 5,$$ a contradiction to $s$ being an integer.  Finally, if $q=9$, then $G_P \le (C_3^2\colon C_8)\colon C_2$ and $|\mathcal{P}| = 10$.  Again, $(s+1)(st+1) = 10$ forces $s+1 \le 2$, a contradiction.  Therefore, $T_P$ cannot be a $\mathcal{C}_5$ subgroup.

\begin{center}
\begin{table}[ht]
\caption{Ruling out examples for Case 4 of the proof of Proposition \ref{prop:linearmaxlgeom}.}\label{tbl:C5}\label{PSL-C_5}
\begin{tabular}{l|l|l|c|c|p{5cm}}
\toprule
$(n,q)$ & $(n,q^{\frac{1}{2}})$ &$t+1$ & $\mathcal{P}$ & $(t+1)^2 < |\mathcal{P}| < (t+1)^3$? & Positive integral solution to $(s+1)(st+1) = |\mathcal{P}|$?\\
\midrule
$(2,16)$ &  $(2,4)$ & $6$ &   $68$ & true & false\\
$(2,25)$ & $(2,5)$ & $5$ &   $65$ & true & false\\
$(2,49)$ &  $(2,7)$ &  $7$ &   $175$ & true & false\\
$(2,64)$ & $(2,8)$ & $28$ & $520$ & false & \\
$(2,81)$ &  $(2,9)$ & $6$ & $369$ & false & \\
$(2,121)$ & $(2,11)$ & $11$ & $671$ & true & false\\
$(3,4)$ & $(3,2)$ & $8$ & $120$ & true & false\\
$(4,4)$ & $(4,2)$ & $8$ & $48960$ & false & \\
\bottomrule
\end{tabular}
\end{table}
\end{center}

\item[Case 5]
Suppose that $T_P$ is a $\mathcal{C}_6$-subgroup.  The possibilities for $(T,T_P)$ are $(\PSL_4(5), 2^4.A_6)$, $(\PSL_3(4),3^2.Q_8)$, $(\PSL_2(23),S_4)$, $(\PSL_2(17), S_4)$, $(\PSL_2(13), A_4)$, $(\PSL_2(11), A_4)$, $(\PSL_2(7), S_4)$, and $(\PSL_2(5), A_4)$.  If $T \cong \PSL_4(5)$, then $|\mathcal{P}| = 1259375$ and $t+1 \le 6$, a contradiction to Lemma \ref{lem:Pbounds}(i).  If $T \cong \PSL_3(4)$, then $|\mathcal{P}| = 280$ and $t+1 = 9$ (since as above $t + 1 \le 3$).  However, $(s+1)(8s+1) = 280$ has no integral solutions, a contradiction.  If $T \cong \PSL_2(q)$ for $q = 13, 17, 23$, then $t+1 \le 4$ and $(t+1)^3 = 64 < |\mathcal{P}|$, a contradiction to Lemma \ref{lem:Pbounds}(i).  If $T \cong \PSL_2(11)$, then $t + 1 = 4$ and $|\mathcal{P}| = 55$.  However, $(s+1)(3s+1) = 55$ has no integral solutions, a contradiction.  Finally, if $T \cong \PSL_2(q)$ for $q = 5,7$, then $|\mathcal{P}| = 5, 7,$ respectively, a contradiction to $(s+1)(st+1) = |\mathcal{P}|$.  Hence $T_P$ cannot be a $\mathcal{C}_6$-subgroup.

\item[Case 6] Suppose that $T_P$ is a $\calC_8$-subgroup.  Assume first that $T_P$ is not solvable.  By Lemma \ref{stab}(i), we have that $T_P \cong \PSU_3(q^{1/2})$ or $T_P \cong \Sp_n(2)$.  Suppose first that $T_P \cong \PSU_3(q^{1/2})$.  We note that $T_P \cong \PSU_3(q^{1/2})$ cannot be in the kernel of the action of $G_P$ on $\Gamma(P)$, since this would then imply that $$(t+1) \le q-1 < q^{\frac{3}{2}}+1,$$ a contradiction to Lemma \ref{stab}(iv).  Hence we have that $t+1 = q^{3/2} + 1$.  However, this implies that $$(s+1)(sq^{\frac{3}{2}}+1) = |\mathcal{P}| = \frac{|T|}{|T_P|} = \frac{|\PSL_3(q)|}{|\PSU_3(q^{1/2})|} = q^{\frac{3}{2}}(q+1)(q^{\frac{3}{2}} - 1).$$  This implies that $q^{\frac{3}{2}}$ divides $s+1$, and, since $s+1 \le t+1$, this means that $s+1 = q^{\frac{3}{2}}$.  Plugging in this value of $s$, we find that there are no integral solutions, a contradiction.

Now suppose that $T_P \cong \Sp_n(2)$.  First, if $n = 4$, then $T_P \cong S_6$, $G_P \le S_6 \times C_2$, and $$(s+1)(st+1) = |\mathcal{P}| = |T\colon T_P| = 28.$$  Since $s+1 > 2$, we must have that $s+1 = 4$ and $st+1 = 7$.  Thus $s=3$, $3t+1 = 7$, and $t=2$, a contradiction to $s\le t$ (and to the fact that there is no generalized quadrangle of order $(2,3)$ or $(3,2)$ \cite[\S1.2.2,\S1.2.3]{fgq}).  Suppose that $n \ge 6$.  We note that $T_P \cong \Sp_n(2)$ cannot be in the kernel of the action of $G_P$ on $\Gamma(P)$, since this would then imply that $t+1 = 1,$ a contradiction to Lemma \ref{stab}(iv).  Hence we have that $t+ 1 = 2^{n-1} \pm 2^{n/2 - 1}$.  By  \cite[Lemma 4.2, Corollary 4.3]{largesubs}, must have that $$\frac{2^{n^2 - 2}}{\left(\frac{45}{64} \cdot 2^{\frac{n(n+1)}{2}}\right)} < |T\colon T_P| = |\mathcal{P}|< (t+1)^3 \le (2^{n-1} + 2^{\frac{n}{2} - 1})^3,$$ which implies that $n = 6$.  When $n = 6$, we have that $|\mathcal{P}| = 13888$ and $t+1$ is either $28$ or $36$.  This means that either $(s+1)(27s+1) = 13888$ or $(s+1)(35+1) = 13888$, which is a contradiction since neither of these equations has an integral root.  Hence $T_P$ must be solvable.

If $T_P$ is solvable, then either $n=2$ and $T_P$ is of type $\rmO^{\pm}_2(q) \cong D_{2(q\mp 1)}$ or $n=3$, $q = 2$, and $T_P \cong S_4$.  However, $T_P$ cannot be a dihedral maximal subgroup of $T = \PSL_2(q)$ by the arguments used above in the $\mathcal{C}_2$ and $\mathcal{C}_3$ cases.  If $T_P \cong S_4$, then $(s+1)(st+1) = |\mathcal{P}| = 7$, which has no solutions with $s,t$ positive integers.  Hence $T_P$ cannot be a $\mathcal{C}_8$-subgroup.
\end{description}

This exhausts all possibilities for $T_P$, and therefore $T_P$ cannot be a maximal geometric subgroup of $T = \PSL_n(q)$, as desired.

\end{proof}

\subsection{Unitary groups}

For unitary groups, the following proposition shows that $\calQ$ is a classical generalized quadrangle $\herm(4,q^2)$ or  $\Q^-(5,q)$.

\begin{prop}
\label{prop:unitarymaxlgeom}

Assume Hypothesis \ref{hyp:3arcAlmostSimple}.
Assume that $T=\soc(G)=\PSU_n(q)$ for any $n \ge 3, q \ge 2$, and the point stabilizer $T_P$ is a maximal geometric subgroup.
Then either
\begin{itemize}
\item[(i)] $n=4$, $T_P=P_1=q.q^4\colon {1\over\gcd(4,q+1)}(\GL_1(q^2)\circ\GU_2(q))$,
$T_\ell=P_2=q^4\colon {1\over\gcd(4,q+1)}\GL_2(q^2)$, $(s,t)=(q,q^2)$, and $\calQ=\Q^-(5,q)$, or

\item[(ii)] $n=5$, $T_P=P_1=q.q^6\colon {1\over\gcd(5,q+1)}(\GL_1(q^2)\circ\GU_3(q))$,
$T_\ell=P_2=q^4.q^4\colon {1\over\gcd(5,q+1)}(\GL_2(q^2)\circ\GU_1(q))$, $(s,t)=(q^2,q^3)$, and $\calQ=\herm(4,q^2)$.
\end{itemize}
\end{prop}

\begin{proof}
By Corollary \ref{cor:largesub}, the stabilizer of a point $T_P$ satisfies $|T| < |T_P|^3$.  The geometric subgroups satisfying this condition are listed in \cite[Proposition 4.17]{largesubs}, and we proceed down the list.

\begin{description}[leftmargin=0pt, style=unboxed, itemsep=1ex]
\item[Case 1] We suppose first that $T_P$ is a maximal parabolic subgroup $P_k$, namely,
\[T_P=P_k=q^{k^2}.q^{2k(n-2k)}\colon {1\over\gcd(n,q+1)}(\GL_k(q^2)\circ\GU_{n-2k}(q)),\]
where $2\leqslant 2k\leqslant n$.
By Lemma~\ref{stab}, $n-2k\leqslant 3$.

\vskip0.1in
\noindent{\bf Subcase 1.1.} Assume first that $n-2k=3$.
Suppose that $T_P^{\Gamma(P)}\rhd\PSU_3(q)$.
Then $|\Gamma(P)|=q^3+1$.
Suppose further that $k\geqslant2$.
Then $\PSL_k(q^2)$ is a composition factor of $G_{P,\ell}^{\Gamma(\ell)}$, and
by Lemma~\ref{stab}, one of the following appears:
\begin{itemize}
\item[(a)] $T_\ell^{\Gamma(\ell)}\rhd p^m\colon \SL_k(q^2)$ with $p^m=q^{2k}$, or $\PSL_{k+1}(q^2)$, or

\item[(b)] $(k,q^2)=(2,4)$, and $T_\ell^{\Gamma(\ell)}\cong A_6$, with $s+1=6$, or

\item[(c)] $(k,q^2)=(2,9)$, and $T_\ell^{\Gamma(\ell)}\cong A_6$ or $2^4\colon A_6$, with $s+1=6$ or $16$, respectively.
\end{itemize}

For case~(a), the valency $s+1=|\Gamma(\ell)|$ equals $q^{2k}$ or ${(q^2)^{k+1}-1\over q^2-1}$,
which is bigger than $q^3=t+1$, a contradiction.
For case~(b), $n=2k+3=7$, and $q=2$.
Thus $t=2^3=8$, and $s=5$.
Therefore, $6.41=(s+1)(st+1)=|T:T_P|$, which is not possible.
For the third case, $n=7$, and $q=3$.
Hence $t=3^3=27$, and $s=5$ or $15$, which do not satisfy the equality $(s+1)(st+1)=|T:T_P|$.

We therefore have $k=1$, and $n=2k+3=5$.
In this case, the point stabilizer
\[T_P=q.q^6\colon {1\over \gcd(5,q+1)}(\GL_1(q^2)\circ\GU_3(q)),\]
and  thus $|\mathcal{P}| = (q^2 + 1)(q^5 + 1)$.
Since $t = q^3$ and $|\mathcal{P}|=(s+1)(st+1)$, we conclude that $s+1=q^2+1$, and $s = q^2$,
and $|\mathcal{L}|=(t+1)(st+1)=(q^3+1)(q^5+1)$.
It implies that the line stabilizer
\[T_\ell=P_2=q^4.q^4.{1\over\gcd(5,q+1)}(\GL_2(q^2)\circ\GU_1(q)).\]
This gives rise to the only possibility that $\mathcal{Q}$ is the classical generalized quadrangle $\herm(4,q^2)$,
of order $(q^2,q^3)$.

Now assume that $\soc(G_P^{\Gamma(P)})\not=\PSU_3(q)$.
Then $\PSU_3(q)$ is a composition factor of $G_{P,\ell}^{\Gamma(\ell)}$.
By Lemma~\ref{stab}, we have $q=3$, and $G_\ell^{\Gamma(\ell)}=2^6\colon \PSU_3(3)$.
Then either
\begin{itemize}
\item[(i)] $k=1$, $G_P^{\Gamma(P)}$ is solvable, $T=\PSU_5(3)$, and $T_P=3.3^6\colon (\GL_1(3^2)\circ\GU_3(3))$,
or

\item[(ii)] $k\geqslant2$, and $G_P^{\Gamma(P)}\rhd \PSL_k(3^2)$.
Then $t+1=|\Gamma(P)|={9^k-1\over 9-1}$, and $s+1=2^6$.
\end{itemize}

To rule out case~(i), we note that by \cite[Tables 8.20, 8.21]{BrayHoltRoney-Dougal}, there is no such subgroup $G_\ell$.  To rule out case~(ii), since $t \le s^2$, we have ${9^k-9\over 9-1} \le 63^2$, which implies that $k \le 4$.  This means that $n = 7,9, 11$, and by \cite[Tables 8.37, 8.38, 8.56, 8.57, 8.72, 8.73]{BrayHoltRoney-Dougal}, no such subgroup $G_\ell$ exists in any of these cases, either.  Hence $n - 2k \neq 3$.

\vskip0.1in
\noindent{\bf Subcase 1.2.}  Assume now that $n-2k=2$.  By assumption, $T_P = q^{k^2}.q^{4k}\colon \frac{1}{\gcd(n,q+1)}(\GL_k(q^2) \circ \GL_2(q))$.
Suppose that $G_P^{\Gamma(P)}$ is solvable. If $k \ge 2$, then both $\PSL_2(q)$ and $\PSL_k(q^2)$ are composition factors of $G_{P,\ell}^{\Gamma(\ell)}$, which is impossible by the Classification of Finite $2$-Transitive Groups \cite[Tables 7.3, 7.4]{CameronPerm}.  Hence $k = 1$, which implies that $n = 4$, $T_P = q.q^{4}\colon \frac{1}{\gcd(4,q+1)}(\GL_1(q^2) \circ \GL_2(q))$, and $|\mathcal{P}| = (q^2+1)(q^3+1)$.  Let $q = p^f$, where $p$ is a prime and $f$ is a positive integer.  Note further that apart from the exceptional cases when $q = 4,5,7,11$, since $\PSL_2(q) \le G_{P,\ell}^{\Gamma(\ell)}$, $t+1 \ge s+1 \ge q^2$.  However, this implies that $$(q^2 + 1)(q^3 + 1) = |\mathcal{P}| = (s+1)(st+1) \ge q^2((q^2-1)^2 + 1),$$ which is false for $q > 2$.  When $q=2$, we have $|\mathcal{P}| = 45 = (s+1)(st+1)$, and we know that $s, t \ge 2^2 - 1 = 3$.  However, this implies that $st + 1 \ge 10$ is a divisor of $45$, i.e., $st + 1 \ge 15$ and $s+1 \le 3 \le s$, a contradiction.  Hence we need only examine the cases $q = 4,5,7,11$, which is done in Table \ref{PSU-1.2}.

\begin{center}
\begin{table}[ht]
\caption{Ruling out examples for Subcase 1.2 of the proof of Proposition \ref{prop:unitarymaxlgeom}.}\label{PSU-1.2}
\begin{tabular}{l|l|l|p{5cm}|l}
\toprule
$q$ & $|\mathcal{P}|$ &$s+1$ & Positive integral solution to $(s+1)(st+1) = |\mathcal{P}|$? & $t \le s^2$ ?\\
\midrule
$4$ &  $1105$ & $11$ & false & \\
$5$ & $3276$ &   $11$ & false & \\
$7$ &  $17200$ &  $15$ & false & \\
$1$ & $162504$ & $12$ & true & false\\
\bottomrule
\end{tabular}
\label{tbl:PSU1.2}
\end{table}
\end{center}

Next we assume that $G_P^{\Gamma(P)}$ is unsolvable.
Then either $G_P^{\Gamma(P)}\rhd\PSU_2(q)\cong\PSL_2(q)$, or
$G_P^{\Gamma(P)}\rhd\PSL_k(q^2)$.

For the former, if $k\geqslant 2$, then $\PSL_k(q^2)$ is a composition factor of $G_{P,\ell}^{\Gamma(\ell)}$;
it follows that $s>t$, a contradiction.
Thus $k=1$, and $n=2k+2=4$.
In this case,
\[T_P=q.q^4\colon {1\over\gcd(4,q+1)}(\GL_1(q^2)\circ\GU_2(q)),\]
and hence $(s+1)(st+1)=|\mathcal{P}| = (q^2+1)(q^3 + 1)$.
If $t+1=|\Gamma(P)|\leqslant q+1$, it implies that $s>t$, which is a contradiction.
Thus one of the exceptional cases occurs: $q=4$ or $q=8$.  If $q = 4$, then $(s+1)(5s+1) = |\mathcal{P}| = 1105$, which has no integral solutions for $s$, and, if $q=8$, then $(s+1)(27s+1) = 33345$, which also has no integral solutions for $s$; both are contradictions.

Assume now that $G_P^{\Gamma(P)} \rhd \PSL_k(q^2)$, $\PSL_k(q^2)$ not solvable.  Then $t+1 \le \frac{q^{2k}-1}{q^2-1} \le q^{2k-1}$, and, using Lemmas \ref{lem:Pbounds}(i) and \ref{lem:isotropics}, we have:
\[
q^{6k-3} \ge (t+1)^3 \ge |\mathcal{P}| = \frac{q^{2(k+1)}-1}{q^2 - 1}\prod_{j=1}^k(q^{2j+1}+1) \ge q^{2k}\cdot q^{(k+1)^2 - 1} = q^{k^2 + 4k}.
\]

Hence $6k - 3 \ge k^2 + 4k$, which implies that $-2 \ge (k-1)^2$, a contradiction. Therefore $n - 2k
\neq 2$.

\vskip0.1in \noindent{\bf Subcase 1.3.} Next, suppose that $n-2k=1$. Then
\[T_P=q^{k^2}.q^{2k}\colon {1\over\gcd(2k+1,q+1)}(\GL_k(q^2)\circ \GU_1(q)).\] 
Note also that $|\mathcal{P}| = \prod_{j=1}^k (q^{2j+1}+1)$ by Lemma \ref{lem:isotropics}.  Suppose first that $G_P^{\Gamma(P)}$ is solvable.  Let $q = p^f$, and assume $k \ge 2$.  We know by Lemma \ref{lem:Pbounds}(i) that $(t+1)^3 \ge |\mathcal{P}| > q^{(k+1)^2 - 1}$, and hence $$t+1 > q^{\frac{k(k+2)}{3}} = p^{\frac{k(k+2)f}{3}}.$$  On the other hand, since $G_P^{\Gamma(P)}$ is solvable, $G_P^{\Gamma(P)} \cong E_{t+1}\colon C_t,$ where $E_{t+1}$ is an elementary abelian group of order $t+1$.  By the structure of $T_P$ (and hence $G_P$), we know that $t \le \frac{1}{\gcd(2k+1,q+1)}(q^2 -1)\cdot 2(2k+1,q+1)f = 2f(p^{2f}-1).$  Hence $$2f p^{2f} > 2f(p^{2f}-1) + 1 \ge t+1 > p^{\frac{k(k+2)f}{3}},$$ and $p^{\frac{((k+1)^2 -3)f}{3}} < 2f.$  If $k \ge 3$, this means that $2^{\frac{13f}{3}} \le p^{\frac{((k+1)^2 -3)f}{3}} < 2f$, which no real value of $f$ satisfies, a contradiction.  Hence $k=2$, which means that $p^{\frac{2f}{3}} < 2f$.  The only possible solutions are $p=2$ and $f<5$.  This means that $|\mathcal{P}| = (2^{3f} + 1)(2^{5f} + 1)$ and $G_P = 2^{4f}.2^{4f}\colon GL_2(2^{2f}).f$, which means that $t \le f(2^{2f} - 1)$ and $t+1 \le f\cdot 2^{2f}$.  Since $t+1 > 2^{\frac{8f}{3}}$ as above, we have $2^\frac{8f}{3} < f \cdot 2^{2f}$, which implies that $2^\frac{2f}{3} < f$, an equation that has no real solutions.  Hence $k = 1$.  However, this means that $n = 3$, and $G$ would be $2$-transitive on $\mathcal{P}$, a contradiction.  Thus $G_P^{\Gamma(P)}$ cannot be solvable and $k \ge 2$.

Since $k \ge 2$ and $G_P^{\Gamma(P)}$ is not solvable, we have that $\PSL_k(q^2)$ is a composition factor of $G_P^{\Gamma(P)}$ or $G_{P,\ell}^{\Gamma(\ell)}$.  Noting that $T_P = q^{k^2}.q^{2k}\colon \frac{1}{\gcd(2k+1,q+1)}\GL_k(q^2)$, by the Classification of Finite $2$-Transitive Groups the most that $t+1$ can be in this case is $q^{2k}$, and by Lemmas \ref{lem:Pbounds}(i) and \ref{lem:isotropics}, we have

$$q^{6k} \ge (t+1)^3 > |\mathcal{P}| =  \prod_{j=1}^k (q^{2j+1}+1) > q^{(k+1)^2 - 1}.$$  This means that $6k > k^2 + 2k$, and so $k = 2,3$.

Assume first that $k = 2$.  In this case, $|\mathcal{P}| = (q^5 + 1)(q^3 + 1)$.  By the structure of $T_P$ and the Classification of Finite $2$-Transitive Groups, there are exactly three possibilities: $t+1 = q^4$, $t+1 = q^2 + 1$, or $t+1 = 6$ and $q = 3$.  If $t+1 = q^4$, then $$(q^5+1)(q^3 + 1) = |\mathcal{P}| = (s+1)(s(q^4-1) + 1) = s^2(q^4-1) + sq^4 + 1.$$  Since $s^2 \ge t$, $s^2 \ge q^4 - 1$; in fact, $s^2 \ge q^4$ since $s^2$ is a positive integer.  Thus $$q^8 + q^5+ q^3 = |\mathcal{P}| - 1 \ge q^4(q^4 - 1) + q^6 + 1,$$ which implies that $q^2 + 1 \ge q^3 - q$, which is false for $q \ge 2$, a contradiction.  If $t+1 = q^2 + 1$, then $$(q^3 + 1)(q^5 + 1) = |\mathcal{P}| = (s+1)(sq^2 + 1),$$ which implies that $s = q^3 > t$, a contradiction.  Finally, if $t+1 = 6$ and $q = 3$, we have $$(3^3 + 1)(3^5 + 1) = |\mathcal{P}| = (s+1)(5s+1),$$ which has no integral roots, a contradiction.

Assume now that $k = 3$.  This means that $T_P = q^9.q^6\colon \frac{1}{\gcd(7,q+1)}\GL_3(q^2)$.  Note that $t+1 = q^6$ or $t+1 = q^4 + q^2 +1$, since $T_P$ is $2$-transitive on $\Gamma(P)$.  If $t+1 = q^6$, then by Lemma \ref{lem:ratio} $s+1 = q^6|T_\ell|/|T_P|$.  Since $s+1$ is a positive integer, this implies that $q^{15}$ divides $|T_\ell|.$ Looking at the possibilities for $T_\ell$ \cite[Tables 8.37, 8.38]{BrayHoltRoney-Dougal}, and noting that $|T_\ell| \le |T_P|$ since $$|T\colon T_P| = |\mathcal{P}| =(s+1)(st+1) \le (t+1)(st+1) = |\mathcal{L}| = |T\colon T_\ell|,$$ we see that $T_\ell \cong T_P$ or $T_\ell \cong \frac{1}{\gcd(q+1,7)}\GU_6(q)$.  Since $\PSU_6(q)$ cannot be the composition factor of a $2$-transitive group, we get that $T_P \cong T_\ell$, which implies that $s = t = q^6 -1$ by Lemma \ref{lem:ratio}.  However, this implies that $$(q^7 + 1)(q^5 + 1)(q^3 + 1) = |\mathcal{P}| = q^6((q^6-1)^2 + 1),$$ a contradiction since the left hand side is coprime to $q$.  Hence $t+1 = q^4 +q^2 + 1$.  We apply Lemma \ref{lem:ratio} as above to conclude that $q^{21}$ divides $T_\ell$.  Looking at the possibilities for $T_\ell$ with $|T_\ell| \le |T_P|$ \cite[Tables 8.37, 8.38]{BrayHoltRoney-Dougal}, we see that $T_\ell \cong T_P$, which implies that $s+1 = t+1 = q^4 + q^2 + 1$, and so $$(q^7+1)(q^5+1)(q^3+1) = |\mathcal{P}| = (q^4+q^2 + 1)((q^4+q^2)^2 + 1),$$ which has no real roots for $q \ge 2$, a contradiction.  Hence $n - 2k \neq 1$.

\vskip0.1in
\noindent{\bf Subcase 1.4.}  Finally, assume that $n-2k=0$, namely, $n=2k$.
Then $k\geqslant 2$, $T_P=q^{k^2}\colon {1\over\gcd(2k,q+1)}\GL_{k}(q^2)$, and the number of points
\[|\mathcal{P}|=(q^{2k-1}+1)(q^{2k-3}+1)\cdots(q+1) > q^{(2k-1)+(2k-3)+\cdots+1}=q^{k^2}.\]

Suppose that $k\geqslant3$.
Since $G_P^{\Gamma(P)}$ is 2-transitive, either $G_P^{\Gamma(P)}$ is affine of degree $t+1=q^{2k}$,
or almost simple with socle $\PSL_k(q^2)$ and degree $t+1={q^{2k}-1\over q^2-1}$.
By Lemma \ref{lem:Pbounds}(i), $q^{k^2}<|\mathcal{P}| < (t+1)^3\leqslant (q^{2k})^3$,
and so $k=3,4$ or $5$.  Moreover, by Lemma \ref{lem:ratio}, we have that $q^{k^2 + k(k-1)/2}$ divides $(t+1)|T_\ell|$.  By the Classification of Finite $2$-Transitive Groups and the structure of $T_P$, we see that either $t+1 = q^{2k}$ or $t+1 = (q^{2k}-1)/(q-1)$.  If $t+1 = (q^{2k} - 1)/(q-1)$, then $q^{k^2 + k(k-1)/2}$ divides $|T_\ell|$, and since $|T_\ell| \le |T_P|$ and $G_\ell^{\Gamma(\ell)}$ is a $2$-transitive group (i.e., it has no projective special unitary group of degree $4$ or more as a composition factor), by \cite[Tables 8.26, 8.27, 8.46, 8.47, 8.62, 8.63]{BrayHoltRoney-Dougal}, we must have $T_\ell \cong T_P$, which by Lemma \ref{lem:ratio} implies that $s=t = (q^{2k} - q^2)/(q^2 - 1)$.  This means that 

$$\frac{q^{2k}-1}{q^2-1}\left(\left(\frac{q^{2k}-q^2}{q^2-1}\right)^2 + 1\right) = |\mathcal{P}|=(q^{2k-1}+1)(q^{2k-3}+1)\cdots(q+1).$$  This equation has no real roots for $k = 3,4,5$, so we must have $t+1 = q^{2k}$.  When $k = 3$, we have $$(q^5+1)(q^3+1)(q+1) = |\mathcal{P}| = (s+1)(s(q^6-1) + 1) \ge (q^3 + 1)(q^9 - q^3 + 1),$$ since $s^2 \ge t$ and $s^2$ is a positive integer.  However, this is false for $q \ge 2$, a contradiction.

Now assume that $k=4$, which means that $n = 8$.  By Lemma \ref{lem:ratio}, we have that $q^{22}$ divides $(t+1)|T_\ell| = q^8 \cdot |T_\ell|$, and so $q^{14}$ divides $|T_\ell|$.  By \cite[Tables 8.46, 8.47]{BrayHoltRoney-Dougal} and noting that $|T_\ell| \le |T_P|$ and by Lemma \ref{stab} (i) that $\PSU_n(q)$ is not a composition factor of either $G_P$ or $G_\ell$ for $n \ge 4$, we are left with only two possibilities: $T_\ell \cong T_P$, or $T_\ell \cong \frac{1}{\gcd(q+1,8)}\Sp_8(q).[\gcd(q+1,4)]$ and $q=2$ (hence $T_\ell \cong \Sp_8(2)$).  If $T_\ell \cong T_P$, then $s = t$, and so $s+1 = q^8$ divides $|\mathcal{P}|$, a contradiction since $|\mathcal{P}|$ is coprime to $q$.  If $T_\ell \cong \Sp_8(2)$, then $s+1 = 2^7 + 2^3 = 136$ or $s+1 = 2^7 - 2^3 = 120$ by the Classification of Finite $2$-Transitive Groups.  In particular, $s+1$ is even, while $|\mathcal{P}|$ is odd, a contradiction.

Assume next that $k=5$, which means that $n=10$.  We proceed as in the $k=4$ case.  By Lemma \ref{lem:ratio}, we have that $q^{35}$ divides $(t+1)|T_\ell| = q^{10} \cdot |T_\ell|$, and so $q^{25}$ divides $|T_\ell|$.  By \cite[Tables 8.62, 8.63]{BrayHoltRoney-Dougal} and noting that $|T_\ell| \le |T_P|$ and by Lemma \ref{stab} (i) that $\PSU_n(q)$ is not a composition factor of either $G_P$ or $G_\ell$ for $n \ge 4$, we are left with only two possibilities: $T_\ell \cong T_P$, or $T_\ell \cong \frac{1}{\gcd(q+1,10)}(\gcd(q+1,5) \times \Sp_{10}(q))$ and $q=2$ (hence $T_\ell \cong \Sp_{10}(2)$).  If $T_\ell \cong T_P$, then $s = t$, and so $s+1 = q^{10}$ divides $|\mathcal{P}|$, a contradiction since $|\mathcal{P}|$ is coprime to $q$.  If $T_\ell \cong \Sp_{10}(2)$, then $s+1 = 2^9 + 2^4 = 528$ or $s+1 = 2^9 - 2^4 = 496$ by the Classification of Finite $2$-Transitive Groups.  In particular, $s+1$ is even, while $|\mathcal{P}|$ is odd, a contradiction.

We therefore conclude that $k=2$ and $n=2k=4$.
Hence $T_P=q^4.{1\over\gcd(4,q+1)}\GL_2(q^2)$, and $(s+1)(st+1)=|T:T_P|=(q+1)(q^3+1)$.
Suppose that  $t+1\not=q^2+1$.  Then either $t+1 = 6$ and $q=2$ or $q=3$, or $t+1 = q^4$.  Assume first that $t+1 = q^4$.  Then by Lemma \ref{lem:Pbounds} (ii) $$q^4 + q^3 + q +1 = (q+1)(q^3 + 1) = |\mathcal{P}| > s^2\cdot q^4,$$ which implies that $s^2 < 1 + 1/q + 1/q^3 + 1/q^4 < 2$, a contradiction.  If $t+1 = 6$ and $q=2$ or $q=3$, then $$(s+1)(5s+1) = |\mathcal{P}| = (q+1)(q^3 + 1),$$ which has no positive integral solutions in either case.  Thus $t+1=q^2+1$.
The equation $$(s+1)(sq^2+1) = |\mathcal{P}| = (q+1)(q^3+1)$$ implies that $s=q$, and $T_\ell = P_1=q.q^4\colon {1\over\gcd(4,q+1)}(\GL_1(q^2)\circ\GU_2(q))$ by Lemma \ref{lem:ratio} and \cite[Tables 8.10, 8.11]{BrayHoltRoney-Dougal}.  Since $T_P \cong P_2$ and $T_\ell \cong P_1$, we may identify the points and lines with totally singular subspaces, and the values of $s$ and $t$ imply that $\mathcal{Q} = \Q^-(5,q)$.

\item[Case 2]
Suppose now that $T_P$ is a $\mathcal{C}_1$-subgroup of type $\GU_m(q) \times \GU_{n-m}(q)$, 
the stabilizer of a non-singular $m$-space, where $m < n/2$.
By Lemma~\ref{stab}, we have $m,n-m\leqslant3$, and it implies that $m=1$ or 2.
Thus $n\leqslant5$, and $n=3,4$ or 5.

Suppose that $n=3$.
Then $T=\PSU_3(q)$, and $T_P={1\over\gcd(3,q+1)}\GU_2(q)$.
Hence $t+1=|\Gamma(P)|=q+1$ unless $q= 4,5,7,8,9$. If $t+1 \le q+1$, then $(s+1)(st+1)=|\mathcal{P}|=|T:T_P|=q^2(q^3-1)$.
It implies that $s>t$, a contradiction.  This leaves only the cases $q=4$ and $t+1 = 6$ or $q=8$ and $t+1 = 28$.  In each case, we see that $(s+1)(st+1) = q^2(q^3-1)$ has no integral solutions for $s$, a contradiction.

For the case where $n=4$, the stabilizer $T_P={1\over\gcd(4,q+1)}\GU_3(q)$.
Thus $t+1=|\Gamma(P)|=q^3+1$, and $(s+1)(st+1)=|\mathcal{P}|=|T:T_P|=q^3(q^2+1)(q-1)$.
It implies that $q^3$ divides $s+1$ and $st+1$ divides $(q^2+1)(q-1)$, which is not possible.

We therefore have $n=5$, and $T_P={1\over\gcd(5,q+1)}(\SU_3(q)\times\SU_2(q))\colon (q+1)$ (note that $T_P \neq \frac{1}{\gcd(5,q+1)}\GU_4(q)$ by Lemma \ref{stab} (i)).
Then $T_P^{\Gamma(P)}=\PSU_3(q)$, and $t+1=|\Gamma(P)|=q^3+1$, and 
$(s+1)(st+1)=|\mathcal{P}|=|T:T_P|=q^6(q^5+1)(q^2+1)/(q+1)$.
It follows that $q^6$ divides $s+1$, which is a contradiction since $s\leqslant t=q^3$.

\item[Case 3]  Let $T_P$ be a $\calC_2$-subgroup.

Suppose first that $T_P$ is a $\mathcal{C}_2$-subgroup of type $\GU_{n/k}(q) \Wr S_k$.  By Lemma \ref{stab} (i), $n/k \le 3$.  By \cite[Proposition 4.17]{largesubs}, if $k \ge 4$, then $4 \le n = k \le 11$ and either $q = 2$ or $(n,q) \in \{ (6,3), (5,3), (4,3), (4,4), (4,5) \}$.  First assume that $q=2$.  By \cite[Proposition 4.2.9]{KleidmanLiebeck}, $T_P \cong \frac{1}{(3,n)}3^{n-1}\colon S_n$, and in these respective cases we have $|\mathcal{P}| = |T\colon T_P| = 40,$ $1408,$ $157868,$ $61997056,$ $84315996160,$ $410113005322240,$ 
  $7160244982522052608,$ and $455256187165096674328576$.  However, in each case, by the Classification of Finite $2$-Transitive Groups, we know that the normal subgroup $\frac{1}{\gcd(3,n)}3^{n-1}$ is in the kernel of the action, which implies that $T_P^{\Gamma(P)}$ is a section of $S_n$.  This implies that the value of $t+1$ is: $4$ when $n = 4$, which is ruled out since $\mathcal{Q}$ is not the unique generalized quadrangle of order $(3,3)$ (see Lemma \ref{lemma:smallGQ}); and too small in all other cases, since $(t+1)^3 < |\mathcal{P}|$ for all possible values of $n$, a contradiction to Lemma \ref{lem:Pbounds}(i).  

Next, assume that $k=3$.  By \cite[Proposition 4.17]{largesubs}, we have $q = 2,3,4, 5, 7, 9 , 13$ or $16$.  Since $n/3 \le 3$, $n = 3,6,$ or $9$.  By \cite[Proposition 4.2.9]{KleidmanLiebeck}, $$T_P \cong \left[\frac{(q+1)^2\gcd(q+1,n/3)}{\gcd(q+1,n)}\right].\PSU_{n/3}(q)^3.[\gcd(q+1,n/3)^2].S_3,$$ which means that $$G_P \cong \left[\frac{(q+1)^2\gcd(q+1,n/3)}{\gcd(q+1,n)}\right].\PSU_{n/3}(q)^3.[\gcd(q+1,n/3)^2].S_3.[2\gcd(q+1,n)f],$$ where $f=1$ unless $q=9$ or $q=16$, in which cases $f=2$ and $f=4$, respectively.  These cases are ruled out by calculation in each case by examining the possible values of $t+1$ and noting that $(s+1)(st+1) = |\mathcal{P}| = |T\colon T_P|$.

If $k = 2$, then we first note that $n = k = 2$ has already been ruled out in the proof of Proposition \ref{prop:linearmaxlgeom}.  Since $n/k \le 3$, we have that $n = 4$ or $n = 6$, and $T_P \cong \left[\frac{(q+1)\gcd(q+1,n/2)}{\gcd(q+1,n)}\right].\PSU_{n/2}(q)^2.[\gcd(q+1,n/2)].S_2$.  Suppose first that $n = 6$.  Then, unless $q=2$, $G_P$ is not solvable.  However, $\PSU_3(q)^2$ must be in the kernel of the action by the Classification of Finite $2$-Transitive Groups, and since $T_P$ is transitive on $\Gamma(P)$, $t+1 \le \frac{(q+1)\gcd(q+1,3)}{\gcd(q+1,6)} \cdot \gcd(q+1,3) \cdot 2 \le 6(q+1)$.  However, by Lemma \ref{stab}(iv), $t+1 \ge q^3 + 1$, a contradiction.  If $q=2$, then $T_P \cong 3^{1+4}\colon (Q_8 \times Q_8)\colon S_3$, $G_P \le 3^{1+4}\colon (Q_8 \times Q_8)\colon S_3.S_3$ and $|\mathcal{P}| = 98560$.  Since $(t+1)^3 > |\mathcal{P}|$ by Lemma \ref{lem:Pbounds}(i), $t+1 \ge 47$; however, $G_P$ has no $2$-transitive action on $47$ or more elements, a contradiction.  Thus $n=4$, and unless $q=2$ or $q=3$, $G_P$ is not solvable.  This means that $\PSL_2(q)^2$ must be in the kernel of the action.  Since $T_P$ is transitive on $\Gamma(P)$, this means that $t + 1 \le 4(q+1)$. Hence $$\frac{q^4(q^2 - q + 1)(q^2 + 1)}{2} = |T\colon T_P| = |\mathcal{P}| = (s+1)(st+1) \le (4q+4)((4q+3)^2 + 1),$$ which implies that $q = 2,3$.  These final cases are ruled out by calculation by examining the possible values of $t+1$ and noting that $|\mathcal{P}| = (s+1)(st+1)$.

Assume now that $T_P$ is a $\mathcal{C}_2$-subgroup of type $\GL_{n/2}(q^2).2$.
Suppose that $t+1\not={q^n-1\over q^2-1}$.
Then by the Classification of Finite $2$-Transitive Groups, $(n,q)=(4,2)$ or $(4,3)$ and $t+1 = 6$.  However, these cases are ruled out by examining $(s+1)(5s+1) = |\mathcal{P}| = |T\colon T_P|$.
Thus $t+1 = {q^n - 1\over q^2 - 1}$, and so $t=q^2{q^{n-2}-1\over q^2-1}$.  
Then $(s+1)(st+1)=|\mathcal{P}|=|T:T_P|$ is divisible by $q^{n^2/4}$.
It implies that $s+1$ is divisible by $q^{n^2/4}$ since $\gcd(st+1,q)=1$.
Thus $s>t$, a contradiction.

\item[Case 4] 
Assume that $T_P$ is a $\mathcal{C}_3$-subgroup, an extension field subgroup.
By \cite[Proposition~4.17]{largesubs} and Lemma~\ref{stab}, we have
$T_P={1\over\gcd(n,q+1)}\GU_{n/3}(q).3$ with $n/3=3$ and $q=27$.
Thus $t+1=|\Gamma(P)|=27^3+1$, and $(s+1)(st+1)=|\mathcal{P}|=|T|/|T_P|$ forces $s>t$, which is a contradiction.

\item[Case 5] 
Suppose that $T_P$ is a $\mathcal{C}_5$-subgroup.
There are several candidates to be considered.
If $T_P$ is a $\mathcal{C}_5$ subgroup of type $\GU_n(q_0)$, where $q = q_0^3$, then $n = 3$ and $t = q_0^3 = q$.  
Now $(s+1)(st+1)=|\mathcal{P}|=|T:T_P|$ is divisible by ${q^3/ q_0^3}=q^2$.
It follows that $s+1$ is divisible by $q^2$, which contradicts $s\leqslant t$.

For other candidates, $T_P$ is a symplectic group or an orthogonal group of dimension $n$.
By Lemma~\ref{stab}, noticing that $\POmega_3(q)\cong\PSL_2(q)$, $\POmega_4^-(q)\cong\PSL_2(q^2)$, and 
$\POmega_6^+(q)\cong\PSL_4(q)$, we conclude that 
\[\mbox{$T_p=\PSp_n(2)$, $\POmega_3(q)$, $\POmega_4^-(q).[d]$ with $q$ odd, or $\POmega_6^+(q)$,}\]
where $d=2$ or 4.

For the first case where $T_P=\PSp_n(2)$, we have $t+1=|\Gamma(P)|=2^{2m-1}\pm 2^{m-1}$, where $2m=n$.
The number of points has the form
\[\begin{array}{lll}
(s+1)(st+1)=|\mathcal{P}|=|T:T_P|&=&2^{m(m-1)}(2^{2m-1}+1)(2^{2m-3}+1)\cdots(2+1)\\
 &>&2^{m(m-1)+m^2}=2^{2m^2-m}.\\
\end{array}\]
Since $|\mathcal{P}|<(t+1)^3$, it implies that $m\leqslant3$.
If $m=3$, then a more precise calculation shows that $|\mathcal{P}|>(t+1)^3$, which is not possible.
For $m=2$, we have $T=\PSU_4(2)$, and $T_P=\PSp_4(2)\cong S_6$. This implies that $|\mathcal{P}| = |T\colon T_P| = 36$ and $t+1 = 6$ or $t+1 = 10$. Neither $(s+1)(5s+1) = 36$ nor $(s+1)(9s+1) = 36$ have integer solutions for $s$, a contradiction.

For the second case where $T=\PSU_3(q)$ and $T_P=\POmega_3(q)\cong\PSL_2(q)$, 
we have $t+1=|\mathcal{P}|=q+1$, and $(s+1)(st+1)=|\mathcal{P}|=|T:T_P|$ is divisible by $q^2$.
Then $t=q$, and it implies that $s+1$ is divisible by $q^2$, not possible.

Suppose that $T=\PSU_4(q)$ and $T_P=\POmega_4^-(q).[d]\cong\PSL_2(q^2).[d]$,
where $q$ is odd and $d=(4,q+1)$.
Then $t+1=|\Gamma(P)|=q^2+1$, and $(s+1)(st+1)=|\mathcal{P}|$ is divisible by $q^4$.
It implies that $s+1$ is divisible by $q^4$, which is not possible.

Assume that $T=\PSU_6(q)$ and $T_P=\POmega_6^+(q).2\cong\PSL_4(q).2$.
Then $t+1=|\Gamma(P)|=(q^2+1)(q+1)$, and $(s+1)(st+1)=|\mathcal{P}|={1\over2}q^9(q^5+1)(q^3+1)^2>(t+1)^3$, a contradiction.

\item[Case 6] 
Finally, suppose that $T_P$ is a $\calC_6$-subgroup.
Then 
\[\mbox{$(T,T_P)=(\PSU_4(7),2^4.\Sp_4(2))$, $(\PSU_4(3),2^4.A_6)$, or $(\PSU_3(5),3^2\colon \Q_8)$,}\]
by  \cite[Proposition 4.17(vii)]{largesubs}.
A simple calculation shows that there is no proper integral solution to $(s+1)(st+1) = |T|/|T_P|$,
which is a contradiction.
\end{description}

Therefore, the result holds for $T_P$ any maximal geometric subgroup, as desired. 

\end{proof}

\subsection{Symplectic groups}

Symplectic groups give rise to two families of generalized quadrangles and a sporadic example.

\begin{prop}
\label{prop:symplecticmaxlgeom}
Assume Hypothesis \ref{hyp:3arcAlmostSimple}.
Assume that $T=\soc(G)=\PSp_{2n}(q)$ for any $n \ge 2, q \ge 2$ and
the point stabilizer $T_P$ is a maximal geometric subgroup.
Then one of the following holds:
\begin{itemize}
\item[(i)] $\mathcal{Q}$ is the classical generalized quadrangle $\W(3,q)$, or 

\item[(ii)] $\mathcal{Q}$ is the classical generalized quadrangle $\Q(4,q)$, or

\item[(iii)] $\mathcal{Q}$ is the classical generalized quadrangle $\Q^-(5,2)$, $T=\PSp_4(3)\cong\PSU_4(2)$, $T_P=2^4.A_5$, $T_\ell= 2.(A_4 \times A_4).2$. 
\end{itemize}
\end{prop}

\begin{proof}
By Corollary \ref{cor:largesub}, the stabilizer of a point $T_P$ satisfies $|T| < |T_P|^3$.  The geometric subgroups satisfying this condition are listed in \cite[Proposition 4.22]{largesubs}, and we proceed down the list. 

\begin{description}[leftmargin=0pt, style=unboxed, itemsep=1ex]

\item[Case 1] 
We first consider the case when $T_P$ is a parabolic subgroup of type $P_k$, where $0 < k \le k + m = n$.  
Then 
$$T_P=P_k= [q^{2nk - \frac{3k^2 - k}{2}}]\colon (\GL_k(q) \circ \Sp_{2n-2k}(q)),$$
and by Lemma \ref{lem:isotropics},
$$|\mathcal{P}| = {n \brack k}_q \prod_{i=1}^k (q^{n+1 - i} + 1).$$

In particular, $(s+1)(st+1)=|\mathcal{P}|=|T:T_P|$ is coprime to $q$.  Note that 
\begin{align*}
|\mathcal{P}| &= {n \brack k}_q \prod_{i=1}^k (q^{n+1 - i} + 1)\\
&> \frac{q^{n-1} \cdot q^{n-2} \cdot \dots \cdot q^{n-k}}{q^k \cdot q^{k-1} \cdot \dots \cdot q^1} \cdot q^{n} \cdot \dots \cdot q^{n - (k-1)}\\
&= q^{kn - k(k+1)/2 - k(k+1)/2 + kn - k(k-1)/2}\\
&=q^{2nk - k(3k+1)/2}.
\end{align*}
Moreover, using the Classification of Finite $2$-Transitive Groups, we see that one of the following must hold:

\begin{itemize}
\item[(a)] $T_P^{\Gamma(P)}$ is solvable;

\item[(b)]  $T_P^{\Gamma(P)}\rhd q^{2m}\colon \Sp_{2m}(q)$, and $t + 1 = q^{2m}$;

\item[(c)] $T_P^{\Gamma(P)} \rhd q^k\colon \GL_k(q)$, and $t+1 = q^k$; 

\item[(d)] $\soc(T_P^{\Gamma(P)}) \cong \PSL_k(q)$, and $t+1 = (q^k - 1)/(q-1)$; 

\item[(e)] $\soc(T_P^{\Gamma(P)}) \cong \Sp_{2m}(2)$, and $t+1 = 2^{2m-1} \pm 2^{m-1}$.  
\end{itemize}

We first examine case~(a) and assume that $T_P^{\Gamma(P)}$ is solvable.  Suppose further that $k, 2n-2k \ge 2$.  By Lemma \ref{comp-factor}, both $\PSL_2(q)$ and $\PSp_{2n-2k}(q)$ are composition factors of $G_{P,\ell}^{\Gamma(\ell)}$, which is impossible by the Classification of Finite $2$-Transitive Groups.  Thus we have either that $k=1$ or $n=k$.  If $n=k$, then $T_P = q^{n(n+1)/2}\colon \GL_n(q)$ and by Lemma \ref{lem:Pbounds} (ii) we have 
$$(t+1)^3 > |\mathcal{P}| = \prod_{i=1}^n(q^i+1) > q^{n(n+1)/2},$$ 
and so $t+1 > q^{n(n+1)/6}$.  On the other hand, since $n \ge 2$, $\PSL_n(q)$ lies in $T_P^{[1]}$, and looking at the structure of $T_P$ (and hence $G_P$), we see that $t < df(p^f - 1)\gcd(p^f-1,n)$, where $q = p^f$ and $d = 2$ unless $q$ is even and $n \ge 3$.  Thus $p^{fn(n+1)/6} < 2nf(p^f - 1) < 2nfp^f.$  If $n > 2$, then we must be in one of the following cases, which are each ruled out via inspection: $n=4$, $p=2$, $f=1$; $n=3$ and either $p=2$ and $f \le 4$, $p=3$ and $f \le 2$, or $p=5$ and $f=1$.  Hence $n=2$, $|\mathcal{P}| = (q+1)(q^2+1)$, and $T_P = q^3\colon \frac{1}{\gcd(2,q-1)}\GL_2(q)$.  However, this still means that $s \ge q+1$ since $\PSL_2(q)$ is a section of $G_{P,\ell}^{\Gamma(\ell)}$ (unless $q = 5$, $q=7$, or $q=9$, which are ruled out by inspection), this means by Lemma \ref{lem:Pbounds}(ii) that $$(q+1)(q^2+1) = |\mathcal{P}| > (q+2)^3,$$ a contradiction.  

Hence we have $k=1$ and $|\mathcal{P}| = (q^{2n}-1)/(q-1)$.  Assume first that $q>2$.  If $n-1 \ge 2$, then $\PSp_{2n-2}(q)$ is a composition factor of $G_{P,\ell}^{\Gamma(\ell)}$ and $s \ge (q^{2n-2}-1)/(q-1)$.  However, by Lemma \ref{lem:Pbounds}(ii), 
$$\frac{q^{2n}-1}{q-1} = |\mathcal{P}| > s^2(t+1) > \left(\frac{q^{2n-2}-1}{q-1} \right)^3,$$ 
which implies that $$q^{2n+2} > (q-1)^2(q^{2n}-1) > (q^{2n-2}-1)^3 > q^{3(2n-3)},$$ which only holds when $n \le 2$.  If $n = 2$, then $|\mathcal{P}| = (q^4-1)/(q-1)$ and $\PSL_2(q)$ is a composition factor of $G_{P,\ell}^{\Gamma(\ell)}$, which implies that $s \ge q+1$ (unless $q= 5,7,9$, which are ruled out by inspection).  However, this means that $$q^3 + q^2 + q + 1 = |\mathcal{P}| = (s+1)(st+1) \ge (q+2)((q+1)^2 + 1),$$ a contradiction.  Finally, we assume that $q = 2$.  If $n - 1 \ge 3$, then proceeding as above we have $$2^{2n} > 2^{2n} - 1 = |\mathcal{P}| > s^2(t+1) > 2^{3n-6}(2^{n-1} - 1)^3 \ge 2^{6n-12},$$ and so $n < 3$, a contradiction.  If $n=3$, then $(s+1)(st+1) = |\mathcal{P}| = 63$, which has no solutions in integers satisfying $s \le t \le s^2$.  Finally, if $n=2$, then $T = \Sp_4(2) \cong S_6$, and $|\mathcal{P}| = 15$, which implies that $s=t=2$, and $\mathcal{Q} = \W(3,2)$, the unique generalized quadrangle of order $(2,2)$ (see Lemma \ref{lemma:smallGQ}).  

For case~(b), $t+1 = q^{2m}$, and $(t+1)^3 = q^{6(n-k)}$.
Since $|\mathcal{P}|<(t+1)^3$ by Lemma \ref{lem:Pbounds}(i),  this implies that $2nk - k(3k+1)/2 < 6(n-k)$.  If $k-2 > 0$, then $$n < \frac{3k^2 - 11k}{6(k-2)} < \frac{3k^2 - 11k + 10}{6(k-2)} = \frac{3k-5}{6} < k,$$ a contradiction.  Hence $k \le 2$.  However, if $k \le 2$, then $s^2(t+1) < |\mathcal{P}|$ forces $s^2 < t$, a contradiction.

Assume that case~(c) occurs.
Then $t+1 = q^k$, and $(t+1)^3 = q^{3k}$ and so $2nk - k(3k+1)/2 < 3k$, which implies that $$ k \le n < \frac{3k+7}{4},$$ 
which holds only if $k \le n < 7$.  
However, as above, in this case we have $(q-1).(\PGL_5(q)\times \PSp_{2n-10}(q)) \le T_{P, \ell}$, 
where here $\ell$ is a line incident with $P$.  
Unless $k = 2$ and $q = 3$ and $T_{\ell} \cong 2.(\PSp_2(3) \times \PSp_{2n-2}(3))$, we have that $T_{\ell}$ is a subgroup of a group isomorphic to $T_P$ and $q$ divides $s+1$, a contradiction since $|\mathcal{P}| = (s+1)(st+1)$ is coprime to $q$.  
If $k=2$ and $q=3$, then the full subgroup $[3^{2n-5}] \le T_P^{\Gamma(P)}$, which has socle isomorphic to $3^4$.  This means that $n \le 4$, and the remaining options are ruled out via inspection of the specific groups.

In case~(d), $t+1 = (q^k - 1)/(q - 1) < q^k$, and we know immediately as in cases above that $k \le n < 7$.  
In fact, using the slightly better bound of $(t+1)^3 < q^{3k-1}$, we obtain that $n < (3k^2 + 7k - 2)/(4k)$, 
which implies that $2 \le k = n \le 6$.  
If $k = n \ge 5$, then $s(t+1)^2 > |\mathcal{P}|$, a contradiction to Lemma \ref{lem:Pbounds}(ii).  
If $n = k =2$, then $t+1 = q+1$ and $|\mathcal{P}| = (q+1)(q^2+1)$, which forces $s = q$.  Examining the possible maximal subgroups of $\PSp_4(q)$ \cite{BrayHoltRoney-Dougal}, we see that the only possibility is the classical generalized quadrangle $\W(3,q)$ or its dual, the classical generalized quadrangle $\Q(4,q)$.  If $k = n =3$, then $t+1 = q^2 + q+ 1$ and $|\mathcal{P}| = (q+1)(q^2+1)(q^3+1)$.  By Lemma \ref{lem:Pbounds}(ii), $s(t+1)^2 > |\mathcal{P}|$, which implies that $s > q^2 - q$.  Moreover, $(s+1)(st+1) = |\mathcal{P}|$ implies that $q$ divides $s$.  Since $s \le t = q^2 + q$, this implies that either $s = q^2$ or $s = t = q^2 + q$.  If $s = q^2$, then $(q^2+1)(q^4 + q^3 + 1) = (s+1)(st+1) = |\mathcal{P}| =  (q^2+1)(q^4 +q^3 +q + 1)$, a contradiction.  If $s = t$, then $(q+1)(q^2+1)(q^3 + 1) = |\mathcal{P}| = (t+1)(t^2 + 1) \equiv 1 \pmod {q+1}$, a contradiction.  Finally, if $k=n= 4$, then $t+1 = (q^4 - 1)/(q-1)$ and $|\mathcal{P}| = (q+1)(q^2+1)(q^3+1)(q^4+1)$.  By Lemma \ref{lem:Pbounds}(ii), $s(t+1)^2 > |\mathcal{P}|$, which implies that $t \ge s > q^4 - q^3 + q - 1$.  This only holds when $q = 2$, and examination of the specific group $\Sp_8(2)$ rules it out.

Finally, for case~(e),  $q=2$ and $t+1 = 2^{2(n-k) - 1} \pm 2^{n-k-1}$, then $(t+1)^3 < 2^{6(n-k)}$, and $2nk - k - k(k-1) - k(k+1)/2 < 6(n-k)$ implies that $k = 1,2$.  If $k = 1$, then $2^{2n} > 2^{2n}-1 = |\mathcal{P}| > (t+1)^2 > 2^{4n-7}$, which implies that $n < 3.5$.  However, this implies that $m = n-1 = 2$, and we must have $m \ge 3$ for this $2$-transitive action, a contradiction.  Thus $k=2$, $t+1 = 2^{2n-5} \pm 2^{n-3}$, and $|\mathcal{P}| = \frac{1}{3}(2^{2n} - 1)(2^{2n-2} - 1)$.  We know that $ts^2 + (t+1)s + 1 = |\mathcal{P}|$ and that $s$ is an integer; hence the discriminant $(t+1)^2 -4t(1 - |\mathcal{P}|)$ must be a perfect square.  When $n \ge 5$, the discriminant divided by $2^{13}$ is $1 \pm 2^{n-3} + 307 \cdot 2^{2 n - 9} \pm 77 \cdot 2^{3n-10} - 1181 \cdot 2^{4n-13} \pm 2^{5 n - 6} + 4^{3n-4}$, so it is not a square for $n \ge 5$.  For $n = 2,3,4$, none of the values of the discriminant are perfect squares, a contradiction.  Hence $T_P$ is a maximal parabolic subgroup only when $\mathcal{Q}$ is classical. 

\item[Case 2]
Next, have that $T_P$ is a $\mathcal{C}_1$-subgroup isomorphic to $\gcd(2,q-1).(\PSp_k(q) \times \PSp_{2n-k}(q))$, where $k$ is even.  Since $T_P$ has a $2$-transitive action, this implies that either $q=2$, $k=2$, or $2n - k = 2$.  Note that $k \le n$ and $2n - k = 2$ imply that $k=2$, so whenever $q \neq 2$ we may assume that $k=2$.  Except for the specific exception $n=2$ and $q=3$, by Lemma \ref{stab} (iv) we have that $t+1 \ge (q^{2n-2} - 1)/(q-1)$.  On the other hand, except for the specific exceptions when $q = 4,5,8,9,11$, we also have $t+1 = q+1$, which implies that $2n-2 \le 2$, i.e., $n = k = 2$.  If $n = k = 2$ and $t = q$, then $|T| = |T_P|(s+1)(st+1)$ implies that $q^2$ divides $s+1$, a contradiction to $s \le t$, and the exceptions above are ruled out by inspection.  If $q=2$, then 
 $T_P \cong \Sp_k(2) \times \Sp_{2n-k}(2)$, leaving the following possibilities for $t+1$: $2^{k-1} \pm 2^{k/2 - 1}$ and $2^{2n -k -1} \pm 2^{n - k/2 - 1}$.  In any case, $t+1 < 2^{2n-k}$, and, by \cite[Corollary 4.3]{largesubs}, $|\mathcal{P}| = |T\colon T_P| > 2^{2nk-k^2 -1}$.  Thus Lemma \ref{lem:Pbounds} (ii) implies that $2nk - k^2 - 1 < 6n - 3k$.  If $k > 3$, then this implies that $$k \le n < \frac{k^2 - 3k +1}{2k-6} = \frac{k}{2} + \frac{1}{2k-6} < \frac{k}{2} + 1,$$ a contradiction.  Since $k \le 3$ is even, we have $k=2$.  We note, however, in this situation that $T_{\ell} > T_{P,\ell} \ge \Sp_2(2) \times \POmega^{\epsilon}_{2n-2}(2)$.  This implies that $T_{\ell} \cong T_P$ and $s=t$; however, then $s^2(t+1) > |\mathcal{P}|,$ a contradiction to Lemma \ref{lem:GQbasics}.

\item[Case 3]
Suppose that $T_P$ is a $\mathcal{C}_2$-subgroup of type $\Sp_{2n/k}(q) \Wr S_k$, where $2n/k$ is even and $k \ge 2$.  In particular, this means that $\Sp_{2n/k}(q)$ is not solvable.  Note that $T_P$ has $k$ composition factors isomorphic to $\PSp_{2n/k}(q)$, and, since $\Sp_{2n/k}(q)^k$ is a normal subgroup of $T_P$, either all are composition factors of $G_P^{\Gamma(P)}$ or all are composition factors of $G_{P,\ell}^{\Gamma(\ell)}$.  In either case, there is no such $2$-transitive group, a contradiction.

We next assume that $T_P$ is a $\mathcal{C}_2$-subgroup isomorphic to $(q-1)/2.\PGL_n(q).2$, where $q$ is odd.  In this case, after ruling out the few sporadic possibilities for $(n,q)$ via inspection, we have that $t+1 = (q^n - 1)/(q-1) < q^n$.  By \cite[Lemma 4.2, Corollary 4.3]{largesubs}, we have that $|\mathcal{P}| > q^{n^2 + n - 3}$.  This implies that $n^2 + n - 3 < 3n$, a contradiction for $n \ge 3$.  If $n = 2$, then $t+1 = q+1$; however, this is a contradiction to $(t+1)^3 > |\mathcal{P}| = |T\colon T_P|$.

\item[Case 4] 
Suppose that $T_P$ is a $\calC_3$-subgroup.
By \cite[Proposition 4.22]{largesubs} and Lemma~\ref{stab}, we have
\[\mbox{$(T,T_P)=(\PSp_4(q),\PSp_2(q^2))$, $(\PSp_6(q),\PSp_2(q^3))$, or $(\PSU_6(q),\PSU_3(q))$.}\]

In the case where $T_P=\PSp_2(q^2) \cong \PSL_2(q^2)$, the number of points $(s+1)(st+1)=|\mathcal{P}|=|T:T_P|=q^2(q^2-1)$,
and either  the valency $t+1=|\Gamma(P)|=q^2+1$, or $q=2$  or 3, and $t+1=6$.
For the former,  $st+1=sq^2+1$ is coprime to $q$, and hence $s+1$ is divisible by $q^2$.
Since $s\leqslant t$, we have $s=q^2$ or $q^2-1$.
It implies that $(s+1)(st+1)<q^2(q^2-1)=|\mathcal{P}|$, which is a contradiction.
Therefore, $q=2$ or 3,  $t=5$, and $T=\PSp_4(2)$ or $\PSp_4(3)$. If $q=2$, then $$|\mathcal{P}| = |T\colon T_P| = 12 < 25 = (t+1)^2,$$ a contradiction to Lemma \ref{lem:Pbounds} (ii).  If $q=3$, then $$(s+1)(5s+1) = |\mathcal{P}| = 72,$$ which has no integral solutions. 

For the second case, $(s+1)(st+1)=|T:T_P|$ is divisible by $q^{9-3}=q^6$, and either
$t+1=q^3+1$, or $q=2$, and $t+1=28$. 
For the former, $\gcd(st+1,q)=1$, and so $s+1$ is divisible by $q^6$, which contradicts the assumption $s\leqslant t$.
Thus $q=2$, and $T=\PSp_6(2)$. This means that $$(s+1)(27s+1) = |\mathcal{P}| = 2880,$$ which has no integral solutions.

In the third case where $T_P=\PSU_3(q)$, we have $t+1=q^3+1$, and $t=q^3$.
Now $(s+1)(st+1)=|T:T_P|$ is divisible by $q^{15-3}=q^{12}$.
It implies that $s+1$ is divisible by $q^{12}$, which is a contradiction since $s\leqslant t$.

\item[Case 5] 
Let $T_P$ be a $\calC_5$-subgroup.
By Lemma~\ref{stab}, the only possibility is $T_P=\Sp_{2m}(2)$, with $q=2^f$.
The valency is $t+1=|\Gamma(P)|=2^{2m-1}\pm2^{m-1}$, and the number of points
$(s+1)(st+1)=|\mathcal{P}|=|T:T_P|$ is divisible by $q^{m^2}/2^{m^2}=(2^{f-1})^{m^2}$.
Since $|\mathcal{P}|<(t+1)^3$, we conclude that $f=2$ and $m=2$, namely, $T=\PSp_4(4)$ and $T_p=\PSp_4(2)$.
In this case, $t+1=6$ or 10, and $|\mathcal{P}|=|T:T_P|=1360>10^3$ by the \textsc{Atlas} \cite{atlas}, which is a contradiction.

\item[Case 6] 
Let $T_P$ be a $\calC_6$-subgroup.
By \cite[Proposition 4.22]{largesubs} and noticing that $\Omega_4^-(2)\cong\SL_2(4)\cong A_5$, 
since $T_P$ has a 2-transitive permutation representation,
\[\mbox{$(T,T_P)=(\PSp_4(7), 2^4. S_5)$, $(\PSp_4(5),2^4.A_5)$, or $(\PSp_4(3), 2^4.A_5)$.}\]
In each of these three cases, the valency $t+1=|\Gamma(\alpha)|$ equals 5, 6 or 16.

Suppose that $T=\PSp_4(7)$.
Then $(s+1)(st+1)=|\mathcal{P}|=|T:T_P|$ is larger than $16^3$, and so $|\mathcal{P}|>(t+1)^3$, which is contradiction to Lemma \ref{lem:Pbounds}.

In the case where $T=\PSp_4(5)$, we have $(s+1)(st+1)=|\mathcal{P}|=|T:T_P|=3\cdot5^3\cdot13$.
Since $|\mathcal{P}|<(t+1)^3$, it implies that $t+1=16$, and $t=15$.
Thus $st+1$ is coprime to 15, and so $s+1$ is divisible by $3.5^3$, which contradicts the assumption $s\leqslant t$.

We therefore have $T=\PSp_4(3) \cong \PSU_4(2)$.  Furthermore, $(s+1)(st+1)=|\mathcal{P}|=|T:T_P|=27$, which implies that $s+1=3$ and $st+1=9$, so $s=2$ and $t=4$.  By Proposition \ref{prop:unitarymaxlgeom}, we conclude that $\mathcal{Q}$ is $\Q^-(5,2)$, $T_P=2^4.A_5$, and $T_\ell= 2.(A_4 \times A_4).2$.

\item[Case 7] 
Let $T_P$ be a $\mathcal{C}_8$-subgroup. Hence $q$ is even, and, since $G_P^{\Gamma(P)}$ is $2$-transitive, $n = 2$ and either $T_P = \Omega_4^+(q) \cong \SL_2(q) \circ \SL_2(q)$ or $T_P = \Omega_4^-(q) \cong \PSL_2(q^2)$ or $n=3$ and $T_P = \POmega_6^+(q) \cong \PSL_4(q)$.  If $T_P = \Omega_4^+(q)$, then unless $q = 4,8$, we have $t+1 = q+1$.  On the other hand, $\SL_2(q)$ must be a composition factor of $G_{P,\ell}^{\Gamma(\ell)}$, which implies that $s+1 \ge q^2 > t+1$, a contradiction.  The cases $q=4,8$ are ruled out by inspection.  If $T_P = \Omega_4^-(q)$, then unless $q=2$ (which is ruled out by inspection), $t+1 = q^2 + 1$.  On the other hand, $|T| = |T_P|(s+1)(st+1)$ implies that $q^2$ divides $s+1 \le t+1 =  q^2 + 1$, which implies that $s+1 = q^2$.  However, $|\Sp_4(q)| \neq q^4(q^4-1)(q^4-q^2 +1)$, a contradiction.  Finally, if $T_P = \POmega_6^+(q)$, then $t+1 = (q^4-1)/(q-1)$ unless $q=2$ (which is ruled out by inspection).  On the other hand, $|T| = |T_P|(s+1)(st+1)$ implies that $q^3$ divides $s+1 \le t+1 = q^3 + q^2 + q + 1$, which implies that $s+1 = q^3$, which is a contradiction to $|\Sp_6(q)| = |T_P|(s+1)(st+1)$ as above.
\end{description}

Therefore, if $T_P$ is a maximal geometric subgroup, then $\mathcal{Q}$ must be classical, as desired.
\end{proof}

\subsection{Orthogonal groups}

We first recall some isomorphism relations between some small dimensional groups (see \cite[Prop. 2.9.1]{KleidmanLiebeck}:
\begin{itemize}
\item[] For $q$ odd, $\Omega_3(q)\cong\PSL_2(q)$, 

\item[] $\rmO_2^{\pm}(q)\cong\D_{2(q\mp1)}$, 

\item[] $\Omega_4^+(q)\cong\SL_2(q)\circ\SL_2(q)$, 

\item[]$\Omega_4^-(q)\cong\PSL_2(q^2)$,

\item[] $\POmega_6^+(q)\cong\PSL_4(q)$.

\end{itemize}

\begin{prop}
\label{prop:orthogonalmaxlgeom}
Assume Hypothesis \ref{hyp:3arcAlmostSimple}.
Assume that $T=\soc(G)$ is a simple orthogonal group of dimension at least $7$.
Then the point stabilizer $T_P$ cannot be a maximal geometric subgroup.
\end{prop}

\begin{proof}
Let $\calQ$ be of order $(s,t)$.
Without loss of generality, assume that $s\leqslant t$.
Let $T=\POmega^{\epsilon}_{n}(q)$, where $n \geqslant 7$.
By Corollary \ref{cor:largesub}, the stabilizer of a point $T_P$ satisfies $|T| < |T_P|^3$.  
The geometric subgroups satisfying this condition are listed in \cite[Proposition 4.23]{largesubs}, and we proceed down the list.

\begin{description}[leftmargin=0pt, style=unboxed, itemsep=1ex]

\item[Case 1]  Assume first that $T_P$ is a stabilizer of a non-singular subspace of dimension $2k$, where $k \le n/2$.
Since $G_P$ has a 2-transitive representation, we have
\begin{itemize}
\item[(a)] if $n=2m+1$, then either 
\begin{itemize}
\item[] $n=7$, and $T_P=\rmO_4^-(q)\times\rmO_3(q)$, or $\rmO_4^+(q)\times\rmO_3(q)$, or  $\rmO_6(q)$, or 

\item[] $n=9$, and $T_P=\rmO_6^+(q)\times\rmO_3(q)$;
\end{itemize}

\item[(b)] if $n=2m$ and $T=\POmega_{2m}^-(q)$, then either
\begin{itemize}
\item[] $n=8$, and $T_P=\rmO_4^+(q)\times\rmO_4^-(q)$, or $\rmO_6^+(q)\times\rmO_2^-(q)$, or

\item[] $n=10$, and $T_P=\rmO_6^+(q)\times\rmO_4^-(q)$.  
\end{itemize}

\item[(c)] if $n=2m$ and $T=\POmega_{2m}^+(q)$, then either 

\begin{itemize}
\item[] $n=8$, and $T_P=\rmO_4^-(q)\times\rmO_4^-(q)$, or

\item[] $n=12$, and $T_P=\rmO_6^+(q)\times\rmO_6^+(q)$.
\end{itemize}

\end{itemize}

\begin{center}
\begin{table}[ht]
\caption{Ruling out examples for Case 1 of Proposition \ref{prop:orthogonalmaxlgeom}.}\label{O1}
\begin{tabular}{l|l|l|l|l|l}
\toprule
$n$ & $T_P$ & $t+1$ & $|T|_q$ & $|T_P|_q$ & $(s+1)_q$ \\
\midrule
$7$ & $\rmO_4^-(q)\times\rmO_3(q)$ & $q^2 + 1$ & $q^9$ & $q^3$ & $q^6$ \\
$7$ & $\rmO_4^+(q)\times\rmO_3(q)$ & $q + 1$ & $q^9$ & $q^3$ & $q^6$ \\
$7$ & $\rmO_6^+(q)$ & $q^3 +q^2 +q + 1$ & $q^9$ & $q^6$ & $q^3$ \\
$8$ & $\rmO_4^-(q)\times\rmO_4^+(q)$ & $q^2 + 1$ & $q^{12}$ & $q^{4}$ & $q^8$ \\
$8$ & $\rmO_6^+(q)\times\rmO_2^-(q)$ & $q^3+q^2+ q + 1$ & $q^{12}$ & $q^6$ & $q^6$ \\
$8$ & $\rmO_4^-(q)\times\rmO_4^-(q)$ & $q^2 + 1$ & $q^{12}$ & $q^4$ & $q^8$ \\
$9$ & $\rmO_6^+(q)\times\rmO_3(q)$ & $q^3 + q^2 + q + 1$ & $q^{16}$ & $q^7$ & $q^9$ \\
$10$ & $\rmO_6^+(q)\times\rmO_4^-(q)$ & $q^3+ q^2 +q + 1$ & $q^{20}$ & $q^8$ & $q^{12}$ \\ %
$12$ & $\rmO_6^+(q)\times\rmO_6^+(q)$ & $q^3+q^2 +q+ 1$ & $q^{30}$ & $q^{12}$ & $q^{18}$ \\
\bottomrule
\end{tabular}
\label{tbl:O1}
\end{table}
\end{center}

We proceed uniformly in each case.  Except for the finite number of cases when the action of $T_P^{\Gamma(P)}$ is not the natural action of $\PSL_n(q)$ on $1$-dimensional projective subspaces (each of which is ruled out by inspection), we have $t+1$ as in Table \ref{tbl:O1}.  However, this means that $q$ divides $t$, and since $|T| = |T_P|(s+1)(st+1)$, we have $(s+1)_q = |T|_q/|T_P|_q$ as in Table \ref{tbl:O1}.  In each case except $n = 7$ and $T_P =\rmO_6^+(q)$, this implies that $s > t$, a contradiction.  When $n = 7$ and $T_P = \rmO_6^+(q)$, since $s \le t$, we have $s + 1 = q^3$.  However, this implies that $$|\POmega_7(q)| = |\Omega_6(q)|\cdot q^3((q^3-1)(q^3+q^2+q) + 1),$$ a contradiction. 

\item[Case 2] Let $T_P$ be the stabilizer $P_k$ of a totally singular $k$-space.  Consider first the case $n = 2d+1$.  Noticing the isomorphisms listed at the beginning of this subsection, the possible candidates are as follows:
\begin{itemize}
\item[] $k = d$ and $T_P=q^{d(d-1)/2}.q^d.\GL_d(q)$, or 

\item[] $k = d-1$ and $T_P=q^{(d-1)(d-2)/2}.q^{3(d-1)}.(\GL_{d-1}(q)\times\rmO_3(q))$.
\end{itemize}

By Lemma \ref{lem:isotropics}, we have:
$$|\mathcal{P}| = {d \brack k}_q \prod_{i=1}^k (q^{d+1 - i} + 1) = \frac{(q^{2d}-1) \cdots (q^{2(d-k+1)}-1)}{(q^k-1) \cdots (q-1)} > \frac{q^{2d-1} \cdots q^{2d -2k +1}}{q^k \cdots q} = q^{2dk - \frac{k(3k-1)}{2}}.$$  On the other hand, examining the structure of $T_P$ we note that $t+1 \le q^k$.  Hence by Lemma \ref{lem:Pbounds}(ii) we have
$$2dk - \frac{k(3k-1)}{2} < 3k,$$ which implies that $$4d < 3k +5.$$  Since $k \le d$, this implies that $d < 5$ and hence $d = 3, 4$.  Substituting these values back into $4d < 3k + 5$ also yields $k = d$, and hence $T_P = q^{d(d-1)/2}.q^d.\GL_d(q)$ for $3 \le k = d \le 4$.

Suppose first that $k = d = 3$.  Note that $t+1 = q^3$ or $q^2+q+1$ and $$(s+1)(st+1) = |\mathcal{P}| = (q^3+1)(q^2+1)(q+1).$$  By Lemma \ref{lem:ratio}, since $s+1$ is coprime to $q$, it follows that 
$q^6$ must divide $|T_\ell|$, which implies that $T_\ell$ is a $\mathcal{C}_1$-subgroup.  Note that $T_\ell \neq \Omega_6^{\pm}(q)$ since we have a contradiction to Lemma \ref{stab} in the ``$-$'' case and $s>t$ in the ``$+$'' case.  This implies that $T_\ell$ is also a parabolic subgroup.  Since $s \le t$, $|T_P| \ge |T_\ell|$, and so $T_\ell = P_2$ or $P_3$.  If $T_\ell \cong T_P$, then Lemma \ref{lem:ratio} implies that $s +1 = t+1 = q^2 + q + 1$ since $s+1$ is coprime to $q$.  On the other hand, this implies that $$(q^3+1)(q^2+1)(q+1) = |\mathcal{P}| = (q^2+q+1)((q^2+q)^2 + 1),$$ which has no real roots for $q \ge 3$.  Hence $T_\ell = P_2$.  This implies that $|\mathcal{L}|= (q^6 - 1)(q^2+1)/(q-1)$ is coprime to $q$ by Lemma \ref{lem:isotropics}, and so $t+1 = q^2 + q + 1$.  Hence $$(q^2 + q + 1)(s(q^2 + q) + 1) = |\mathcal{L}| = (q^6 - 1)(q^2+1)/(q-1),$$ which implies that $s = (q^4 + q^2 + q)/(q+ 1) > q^2 + q = t$, a final contradiction.

Suppose now that $k = d = 4$.  Thus $t+1 = q^4$ or $(q^4-1)/(q-1)$ and $$|\mathcal{P}| = (q^4+1)(q^3+1)(q^2+1)(q+1)$$ by Lemma \ref{lem:isotropics}.  If $t+1 = (q^4 - 1)/(q-1)$, then $(t+1)^3 < |\mathcal{P}|$, a contradiction to Lemma \ref{lem:Pbounds} (ii).  Hence $t+1 = q^4$. In this case, Lemma \ref{lem:GQbasics}(ii) implies that $(s+ q^4 - 1)$ divides $s(q^4-1)q^4(s+1)$.  Now, $s(q^4-1)q^4(s+1) - (s+q^4 - 1)(sq^4(q^4-1) + q^4(q^4-1)) = -2q^4(q^4-1)$, so $s + q^4 - 1$ must divide $2q^4(q^4 - 1)$.  Examining $(s+1)(st + 1) = |\mathcal{P}|$, we see that $q$ divides $s^2$, and so $q$ is coprime to $s + q^4 - 1$, and $s+ q^4 - 1$ must divide $2(q^4-1)$.  Thus for some positive integer $m$, $m(s + q^4 - 1) = 2(q^4 - 1)$ and $ms = (2 -m )(q^4 - 1)$.  This implies that $m = 1$ and $s = t$, but then $s^2(t+1) > |\mathcal{P}|$, a contradiction to Lemma \ref{lem:Pbounds}(ii).  

Next we consider the case $n=2d$ and $T=\POmega_{2d}^-(q)$.  Upon consideration of the isomorphisms listed at the beginning of this subsection, the possible candidates are as follows:

\begin{itemize}
\item[] $k = d-1$ and $T_P=q^{(d-1)(d-2)/2}.q^{2(d-1)}.(\GL_{d-1}(q)\times\rmO_2^-(q))$, or

\item[] $k = d-2$ and $T_P=q^{(d-2)(d-3)/2}.q^{4(d-2)}.(\GL_{d-2}(q)\times\rmO_4^-(q))$.
\end{itemize}

We start as in the $n = 2d+1$ case and find by Lemma \ref{lem:isotropics} that
$$|\mathcal{P}| > q^{2dk - \frac{(3k-2)(k+1)}{2}}.$$  Noting again that $t+1 \le q^k$, by Lemma \ref{lem:Pbounds} (ii) we have $4d < 3k + 7$.  Since $k \le d-1$, this means that $4d < 3(d-1) + 7$ and $d \le 3$, a contradiction to $n \ge 7$.

Finally, we consider the case $n = 2d$ and $T=\POmega_{2d}^+(q)$.  Noticing the isomorphisms listed at the beginning of this subsection, the possible candidates are as follows:

\begin{itemize}
\item[] $k=d$ and $T_P=q^{d(d-1)/2}.\frac{1}{q-1}\GL_{d}(q)$, or

\item[] $k = d-2$ and $T_P=q^{(d-2)(d-3)/2}.q^{4(d-2)}.(\GL_{d-2}(q)\times\rmO_4^+(q))$, or

\item[] $k =d -3$ and $T_P=q^{(d-3)(d-4)/2}.q^{6(d-3)}.(\GL_{d-3}(q)\times\rmO_6^+(q))$.
\end{itemize}

We start again as in the $n = 2d+1$ case and find by Lemma \ref{lem:isotropics} that
$$|\mathcal{P}| > q^{2dk - \frac{3k(k+1)}{2}}.$$  Noting again that $t+1 \le q^k$, by Lemma \ref{lem:Pbounds} (ii) we have $4d < 3k + 9$.  If $k \le d-2$, then $4d < 3d + 3$, a contradiction to $d \ge 4$.  Hence we have $k = d$ and $4 \le k = d \le 8$.  In each case, $t+1 = q^d$ or $t+1 = (q^d-1)/(q-1)$ and $$|\mathcal{P}| = 2(q^d+1)(q^{d-1}+1) \cdots (q+1).$$  We rule out $k = d = 7$ and $k = d = 8$ by noting that $$(t+1)^3 \le q^{3d} < |\mathcal{P}|,$$ a contradiction to Lemma \ref{lem:Pbounds} (ii).  

Suppose that $k = d = 4$.  If $t+1 = q^4$, then $(t+1)^2 > |\mathcal{P}|$, a contradiction to Lemma \ref{lem:Pbounds} (ii).  Hence $t+1 = q^3 + q^2 + q + 1$, and by Lemma \ref{lem:ratio}, $|T_\ell|_q = q^{12}$ and $T_\ell$ is a parabolic subgroup.  Since $s \le t$, $|T_P| \ge |T_\ell|$, and $k \neq 3$, we have $T_P \cong T_\ell$ or $T_\ell = P_2$.  If $T_\ell = P_2$, since $s^2 \ge t$, we see that $s +1 = q^2$, which is a contradiction since $|\mathcal{P}| = (s+1)(st+1)$ is coprime to $q$.  If $T_P \cong T_\ell$, we see that $s +1 = t+ 1$ by Lemma \ref{lem:ratio}, a contradiction to $|\mathcal{P}| = (s+1)(st+1)$ since $|\mathcal{P}| \equiv 2 \pmod q$ but $(s+1)(st+1) \equiv 1 \pmod q$.  

Suppose that $k = d = 5$.  If $t+1 = q^5$, then since $|\mathcal{P}|$ is coprime to $q$ by Lemma \ref{lem:ratio}, we see that $|T_\ell|_q = q^{15}$.  However, this and Lemma \ref{stab} rule out all possible maximal subgroups (see \cite[Tables 8.66, 8.67]{BrayHoltRoney-Dougal}).  Hence $t+1 = (q^5-1)/(q-1)$.  We reach a contradiction by precisely the same reasons as the case $k = d = 4$.

Finally, suppose that $k = d = 6$.  First, if $t+1 = (q^6 - 1)/(q-1)$, then $(t+1)^3 < |\mathcal{P}|$, a contradiction to Lemma \ref{lem:Pbounds} (ii).  Second, if $t+1 = q^6$, then since $|\mathcal{P}|$ is coprime to $q$ by Lemma \ref{lem:ratio}, $|T_\ell|_q = q^{24}$.  However, this and Lemma \ref{stab} rule out all possible maximal subgroups (see \cite[Tables 8.82, 8.83]{BrayHoltRoney-Dougal}).

\item[Case 3] Suppose that $T_P$ is a $\calC_2$-subgroup.
Then 
\begin{itemize}
\item[(a)] $T=\POmega_7(q)$, and $T_P=\rmO_3(q)\Wr S_2$;

\item[(b)] $T=\POmega_9(q)$, and $T_P=\rmO_4^\pm(q)\Wr S_2$;

\item[(c)] $T=\POmega_{13}(q)$, and $T_P=\rmO_6^+(q)\Wr S_2$;

\item[(d)] $T=\POmega_8^-(q)$, and $T_P=\rmO_4^\pm(q)\Wr S_2$;

\item[(d)] $T=\POmega_8^+(q)$, and $T_P=\rmO_4^\pm(q)\Wr S_2$;

\item[(d)] $T=\POmega_{12}^\pm(q)$, and $T_P=\rmO_6^+(q)\Wr S_2$.

\end{itemize}

In particular, in each case  $\rmO^{\epsilon}_{k}(q)$ is not solvable.  Note that $T_P$ has two composition factors isomorphic to $\POmega^{\epsilon}_{k}(q)$, and, since $\rmO^{\epsilon}_{k}(q) \times \rmO^{\epsilon}_{k}(q)$ is a normal subgroup of $T_P$, either all are composition factors of $G_P^{\Gamma(P)}$ or all are composition factors of $G_{P,\ell}^{\Gamma(\ell)}$.  In either case, there is no such $2$-transitive group, a contradiction.

\item[Case 4]  Let $T_P$ is a $\calC_3$-subgroup.
By \cite[Proposition 4.23]{largesubs}, we have the candidates
\begin{itemize}
\item[(a)] $T=\POmega_8^+(q)$, and $T_P=\rmO_4^+(q^2)$;

\item[(b)] $T=\POmega_8^-(q)$, and $T_P=\rmO_4^-(q^2)$;

\item[(c)] $T=\POmega_{12}^+(q)$, and $T_P=\rmO_6^+(q^2)$.

\end{itemize}

\begin{center}
\begin{table}[ht]
\caption{Ruling out examples for Case 4 of Proposition \ref{prop:orthogonalmaxlgeom}.}\label{O2}
\begin{tabular}{l|l|l|l|l|l}
\toprule
$n$ & $T_P$ & $t+1$ & $|T|_q$ & $|T_P|_q$ & $(s+1)_q$ \\
\midrule
$8$ & $\rmO_4^+(q^2)$ & $q^2 + 1$ & $q^{12}$ & $q^4$ & $q^{8}$ \\
$8$ & $\rmO_4^-(q^2)$ & $q^4+1$ & $q^{12}$ & $q^4$ & $q^8$ \\
$12$ & $\rmO_6^+(q^2)$ & $(q^8-1)/(q^2-1)$ & $q^{30}$ & $q^{12}$ & $q^{18}$ \\
\bottomrule
\end{tabular}
\label{tbl:O2}
\end{table}
\end{center}

We proceed uniformly in each case.  We have $t+1$ as in Table \ref{tbl:O2} except possibly for the cases $q=2,3$ when $n=8$ and $T_P = \rmO_4^+(q^2)$, and these are ruled out by inspection.  However, this means that $q$ divides $t$, and since $|T| = |T_P|(s+1)(st+1)$, we have $(s+1)_q = |T|_q/|T_P|_q$ as in Table \ref{tbl:O2}.  In each case, this implies that $s > t$, a contradiction.  

\item[Case 5]  If $T_P$ is a $\calC_4$-subgroup, then $n = 8$ and $T_P=\Sp_4(2)\times\Sp_2(2)$ or $n = 12$ and $\Sp_6(2)\times\Sp_2(2)$.  In each case, we have $t+1 = 2^{n/2 -1} \pm 2^{n/4 - 1}$ and $(s+1)(st+1) = |\mathcal{P}| = |\Sp_n(2)\colon \Sp_{n/2}\times\Sp_2(2)|$ implies $s > t$, a contradiction.

\item[Case 6]  The subgroup $T_P$ is not a $\calC_5$-subgroup since $G_P$ has a 2-transitive permutation representation.

\item[Case 7]  For the case that we have a $\calC_6$-subgroup, we have $T=\POmega_8^+(3)$, and $T_P=2^6.\Omega_6^+(2)$.  However, $(t+1)^3 \le 2^{18} < |T\colon T_P| = |\mathcal{P}|$, a contradiction to Lemma \ref{lem:Pbounds} (ii). 

\item[Case 8]  Finally, for $\calC_7$-subgroups, we have $T=\POmega_8^+(q)$, and $T_P=\Sp_2(q) \Wr  S_3$ with $q\leqslant 2^7$ even.  However, this implies that either $G_P^{\Gamma(P)}$ or $G_{P,\ell}^{\Gamma(\ell)}$ has three composition factors isomorphic to $\PSp_2(q)$, a contradiction.
\end{description}
\end{proof}

\section{$\mathcal{C}_9$-subgroups of classical groups}

\begin{prop}
\label{prop:notC9}
Assume Hypothesis \ref{hyp:3arcAlmostSimple}.
Then the point stabilizer $T_P$ cannot be a $C_9$-subgroup.
\end{prop}

\begin{proof}
From \cite[Proposition 4.28]{largesubs}, we list the possibilities in Table \ref{tbl:classical_not_novelty} and apply
the simple inequality from Lemma \ref{lem:Pbounds}(i). We then look at the left-over cases, and we can solve for $s$, since $|T:T_P|=(s+1)(st+1)$. In most cases, the discriminant $\Delta := (t+1)^2 - 4 t (1-|T:T_P|)$ is not a square, and the remaining case, $(T,t+1)=(\PSL_2(19),6)$, yields no solution for $s$. Therefore, there are no examples in this case.

\begin{center}
\begin{table}[ht]
\caption{Possibilities for $(T,T_P)$ arising from \cite[Proposition 4.28]{largesubs}.}\label{tbl:classical_not_novelty}
\begin{tabular}{l|c|c|c}
\toprule
$T$ & $T_P$ &  $t+1$ & $(t+1)^2<|T:T_P|<(t+1)^3$? \\
\midrule
$\POmega_{14}^+(2)$ & $A_{16}$  & 16 & false \\
$\PSp_{12}(2)$ & $S_{14}$  & 14 & false \\
$\POmega_{11}^-(2)$ & $A_{13}$  & 13 & false\\
$\POmega_{10}^-(2)$ & $A_{12}$  & 12 & false\\
$\PSp_8(2)$ & $S_{10}$   & 10 & false\\
$\POmega_8^+(2)$ & $A_9$  & 9 & false\\
$\POmega_7(3)$ & $\PSp_6(2)$  & 28, 36 & true\\
$\POmega_7(3)$ & $S_9$  & 9 & false\\
$\PSU_6(2)$ & $M_{22}$  & 22 & false\\
$\PSp_6(2)$ & $\PSU_3(3).2$  &  28 & false\\
$\PSL_4(2)$ & $A_7$  & 7 & false \\
$\PSU_4(3)$ & $A_7$  & 7 & false \\
$\PSU_4(3)$ & $\PSL_3(4)$  & 21 & false \\
$\PSp_4(2)'$ & $A_5$  & 5, 6 & false\\
$\PSL_3(4)$ & $A_6$  & 6 & true\\
$\PSL_3(4)$ & $A_6$  & 10 & false\\
$\PSU_3(5)$ & $A_7$  & 7 & true\\
$\PSU_3(5)$ & $M_{10}$  & 10 & true \\
$\PSU_3(3)$ & $\PSL_2(7)$  & 7, 8 & false\\
$\PSL_2(9)$ & $A_5$  & 5, 6 & false\\
$\PSL_2(11)$ & $A_5$  & 5, 6 & false \\
$\PSL_2(19)$ & $A_5$  & 5, 6 & true\\
$\PSp_4(q)'$ & $\,^2B_2(q)$, $q=2^f$, $f$ odd & $q^2+1$ & true\\
\bottomrule
\end{tabular}
\end{table}
\end{center}

\begin{center}
\begin{table}[ht]
\caption{The remaining possibilities for $(T,T_P)$ from Table \ref{tbl:classical_not_novelty}, after consideration of the discriminant.}
\begin{tabular}{l|c|c|c}
\toprule
$T$ & $T_P$ & $t+1$ & $\Delta$ \\
\midrule
$\Omega_7(3)$ & $\Sp_6(2)$  & 28, 36 & 341848\\
$\PSL_3(4)$ & $A_6$  & 6 & 443416\\
$\PSU_3(5)$ & $A_7$ &  7 & 1136\\
$\PSU_3(5)$ & $M_{10}$ & 10 & 6364 \\
$\PSL_2(19)$ & $A_5$ &  5 & 921\\
$\PSL_2(19)$ & $A_5$ &  6 & 1156\\
$\PSp_4(q)'$ & $\,^2B_2(q)$, $q=2^f$, $f$ odd  & $q^2+1$ & $(q-1) (q+1)^2 (4q^4+q-1)$\\
\bottomrule
\end{tabular}
\end{table}
\end{center}
\end{proof}

\section{Novelty maximal subgroups of classical groups}

Let $S$ be a finite nonabelian simple group and let $A$ be an almost simple group with socle $S$.
If $M$ is a maximal subgroup of $A$, then $M\cap S$ is not necessarily a maximal subgroup of $S$.
In the case that $M\cap S$ is \textbf{not} maximal in $S$, we say that $M$ is a \emph{novelty}.
This section deals with the case that $G_P$
is a novelty subgroup of $G$.

First we look at the case that $G_P$ is a $\mathcal{C}_1$-subgroup. Let $n$ be the dimension of the natural module
for $G$.

\begin{prop}\label{propn:NoveltyC1L}
Suppose $T \cong \PSL_n(q)$. Then
 $G_P$ cannot be a novelty $\mathcal{C}_1$-subgroup of $G$. 
\end{prop}

\begin{proof}
Suppose that $T \cong \PSL_n(q)$, where $n\ge 3$.
We have three cases:
\begin{enumerate}[(i)]
\item \emph{$T_P$ is a $P_k$ type subgroup}.
We may assume without loss of generality that $k\le n/2$ (since we can apply the standard 
duality/graph automorphism to the case $k>n/2$). We know that $t+1 \le (q^{k} - 1)/(q-1)$.  By Lemma \ref{lem:Pbounds}(i) we must have 
\[
\frac{(q^{n/2}-1)^3}{(q-1)^3}\ge \frac{(q^{k} - 1)^3}{(q-1)^3} \ge (t+1)^3 > |\mathcal{P}| \ge {n \brack k}_q,
\] 
and this immediately leads to a contradiction.

\item \emph{$T_P$ is a $\GL_m(q) \oplus \GL_{n-m}(q)$ type subgroup}. Here $m<n/2$. We know that
  $n-m > m$, so $t+1 \le (q^{n-m} - 1)/(q-1)$. By Lemma \ref{lem:Pbounds}(i) we must have
\[
(q^{n-m} - 1)^3/(q-1)^3 \ge (t+1)^3 > |\mathcal{P}| > {n \brack m}_q,
\] and this immediately implies that $m = 1,2$.  Suppose that $m = 1$.  
Then we can be sure about the value of $t+1$; it is $(q^{n-1} - 1)/(q-1)$. Now $|\mathcal{P}| =
q^{n-1}(q^n - 1)/(q-1)$, and by Lemma \ref{lem:Pbounds}(ii), $s^2(t+1) < |\mathcal{P}|$ implies that
 \[
 s^2< q^{n-1}\frac{q^n-1}{q^{n-1}-1}
 \]
 and hence $s \le q^{n/2} + 1$. On the other hand, $(s+1)(st+1) = |\mathcal{P}|$, and $st+1$ is
 coprime to $q$, which implies that $q^{n-1}$ divides $s+1$, a contradiction. Now suppose $m=2$.
 Necessarily $n\ge 5$ and $t+1=(q^{n-2} - 1)/(q-1)$. Then
\[
|\mathcal{P}| = q^{2(n-2)}\frac{(q^n - 1)(q^{n-1}-1)}{(q-1)(q^2-1)}\le q^{2(n-2)} \cdot (q^{n-2})^2,
\]
and by Lemma \ref{lem:Pbounds}(ii), $s^2(t+1) < |\mathcal{P}|$ implies that
\[
s^2 \le q^{2(n-2)}\frac{(q^n - 1)(q^{n-1}-1)}{(q^2-1)(q^{n-2}-1)}\le q^{3n-4}.
\]
On the other hand, $(s+1)(st+1) = |\mathcal{P}|$, and $st+1$ is coprime to $q$, which implies that
$q^{2(n-2)}$ divides $s+1$. Note that $s\ne 2$, since otherwise we would have $q^{2(n-2)}=3$. For
$s\ge 3$, we have $(s+1)^2\le 2s^2$ and hence
\[
q^{4(n-2)}\le (s+1)^2 \le 2 s^2 \le 2 q^{3n-4}.
\]
The only solution is $q=2$ and $n= 5$. Then $|\mathcal{P}|=9920$, but there do not exist $s$ and $t$
with $\sqrt{s}\le t\le s^2$ such that $(s+1)(st+1)=9920$. Therefore, $T_P$ cannot be a $\GL_m(q)
\oplus \GL_{n-m}(q)$ type subgroup.

\item \emph{$T_P$ is a $P_{m,n-m}$ type subgroup}. We now assume that $T_P$ is a $P_{m,n-m}$ type
  subgroup in $\PSL_n(q)$. Since $T_P$ must have a $2$-transitive action on $t+1$ points, it follows
  that $t+1$ is either $(q^m - 1)/(q-1)$ or $(q^{n-2m} - 1)/(q-1)$. If $t+1 = (q^m - 1)/(q-1)$, then
  again using Lemma \ref{lem:Pbounds}(i), we see that ${n \brack m}_q < |\mathcal{P}| < (q^m -
  1)^3/(q-1)^3 < q^{3n/2}$. This implies that $m = 1$, a contradiction to $t \ge 2$. If $t+1 = (q^{n
  - 2m} - 1)/(q-1)$, then a similar argument immediately implies that $m = 1,2$. When $m = 2$,
  $|\mathcal{P}| = (q^n-1)(q^{n-1} - 1)(q^{n-2} - 1)(q^{n-3} - 1)/((q^2-1)(q^2-1)(q-1)(q-1))$ and
  $t+1 = (q^{n-4} - 1)/(q-1)$. However, $(t+1)^3 < |\mathcal{P}|$ in this case, a contradiction.
  Hence we assume that $m = 1$ and $t+1 = (q^{n-2} - 1)/(q-1)$. We note that $|\mathcal{P}| =
  (q^n-1)(q^{n-1}-1)/(q-1)^2$, and by Lemma \ref{lem:Pbounds}(ii) we find that $s <
  \sqrt{2}q^{n/2}$. We also note that $T_P \cong [q^{2n-3}]:[a_{1,1,n-1}^+/\gcd(q-1,n)].(\PSL_1(q)^2
  \times \PSL_{n-2}(q)).[b_{1,1,n-2}^+]$ (see \cite[Proposition 4.1.22]{KleidmanLiebeck} for
  details). Note that $t+1 = (q^{n-2} - 1)/(q-1)$ implies that the kernel of the action of $T_P$ on
  its neighbors contains the full subgroup $[q^{2n-3}]:\PSL_{n-3}(q)$, which is a subgroup of
  $T_{\ell}$. Note that this implies that $T_{\ell} > [q^{2n-3}]:\PSL_{n-3}(q)$, and we deduce that
  $T_{\ell}$ is a $\mathcal{C}_1$-subgroup of $T$ with $m \le 3$. However, in all cases this forces
  $s+1$ to be larger than $\sqrt{2}q^{n/2}$, a contradiction. Thus $T_P$ cannot be a $P_{m,n-m}$
  type subgroup in $\PSL_n(q)$.
\end{enumerate}

\end{proof}

\begin{prop}\label{propn:NoveltyC1Oplus}
Suppose $T \cong \POmega^+_n(q)$. Then $T_P$ cannot be a novelty $\mathcal{C}_1$-subgroup of
$T$, unless $T_P$ is a $P_{n/2}$ type subgroup and $n\le 12$.
\end{prop}

\begin{proof}
We have the following cases:
\begin{enumerate}[(i)]
\item \emph{$T_P$ is a $P_k$ type subgroup, $k\le n/2$}. Since $T_P$ must have a $2$-transitive
  action on $t+1$ points, it follows that $t+1$ is $(q^k - 1)/(q-1)$. We also know from Lemma
  \ref{lem:isotropics} that $|\mathcal{P}|$ is
\[
{n/2 \brack k}_q\prod_{i=1}^k(q^{n/2-i}+1).
\]
Using Lemma \ref{lem:Pbounds}(i), we see that
\[
{n/2 \brack k}_q\prod_{i=1}^k(q^{n/2-i}+1) < (q^k - 1)^3/(q-1)^3
\]
which implies\footnote{$|\mathcal{P}| \ge {n/2 \brack n/2} \prod_{i=1}^{n/2}(q^{n/2-i} + 1) >
q^{3n/2} \ge (t+1)^3$ when $n \ge 13$ } that $k=n/2$ and $n\le 12$.

\item \emph{$T_P$ is a $\GO_k^{\epsilon'}(q)\perp \GO_{n-k}^{\epsilon'}(q)$ type subgroup, $k\le
  n/2$}. In order for $T_P$ to have a $2$-transitive action, we must have $\epsilon'=-$ and either
  $k=4$ or $n-k=4$. Hence $t+1=q^2+1$. Now $|T_P|$ is at least
\[
|\Omega_4^-(q)|^2 = \gcd(2,q-1)^2 q^4(q^4-1)^2
\]
and $|T_P|< (t+1)^3$ (by Lemma \ref{lem:Pbounds}(i)). However, $q^4(q^4-1)^2\le (q^2+1)^3$ if and
only if $q^2\le 1$; a contradiction. Therefore, this case does not arise.

\item \emph{$q$ is even and $T_P$ is a $\Sp_{n-2}(q)$ type subgroup}. In order for $T_P$ to have a
  $2$-transitive action, we must have $q=2$, $n-2\ge 6$ and $t+1=2^{n-3}\pm 2^{n/2-3}$. Now $|T_P|$
  is
\[
|\Sp_{n-2}(2)| = 2^{(n-2)^2/4}\prod_{i=1}^{(n-2)/2}(2^{2i}-1)
\]
which in turn is at least $2^{9}(2^2-1)(2^4-1)(2^6-1)=1451520$ (by taking $n-2=6$). The bound
$|T_P|\le (t+1)^3$ (Lemma \ref{lem:Pbounds}(i)) immediately implies that $n\le 10$. The bound
$(t+1)^2\le |T_P|$ (Lemma \ref{lem:Pbounds}(i)) implies that $n\ge 10$, and hence $n=10$. Therefore,
$|T:T_P|=|\mathrm{P\Omega}^+_{10}(2):\Sp_8(2)|=496$ and so we have
\[
(s+1)(st+1) = 496
\]
where $t=124$ or $t=132$. We have two quadratic equations in $s$, both of which have no integer
solutions. So this case does not arise.

\end{enumerate}

\end{proof}

\begin{prop}\label{prop:dimatmost12notC1}
For $n\le 12$, $G_P$ cannot be a novelty $\mathcal{C}_1$-subgroup of $G$.
\end{prop}

\begin{proof}
For a classical simple group $T$, we will write $\bar{T}$ for its preimage in the associated matrix
group. Likewise, for a subgroup $H$ of $T$, the notation $\bar{H}$ will denote the preimage of $H$
in $\bar{T}$. By looking through the tables in Bray, Holt, Roney-Dougal \cite{BrayHoltRoney-Dougal},
and excluding the examples that do not satisfy the conclusions of Propositions
\ref{propn:NoveltyC1L} and \ref{propn:NoveltyC1Oplus}, we see that $(\bar{T}, \bar{T}_P)$ can only
be the following: (i) $(\Omega_{8}^+(q),\frac{1}{\gcd(q-1,2)}\times \G_2(q))$ or (ii)
$(\Omega_3(9),D_{10})$.

For the case $(T,T_P,t+1)=(\Omega_3(9) , D_{10}, 5)$, we have $(s+1)(4s+1)=36$ which has no integer
solutions for $s$. In the case that $T_P= \frac{1}{\gcd(q-1,2)}\times \G_2(q)$, since $T_P$ must have a
$2$-transitive action, we must have $q=2$. Then $t+1=28$ and $|T:T_P|= 14400$. So
$14400=(s+1)(27s+1)$ which has no integer solutions in $s$.
\end{proof}

\begin{lem}
Let $\bar{G}$ be a nearly simple subgroup of $\GL_n(q)$, and let $G$ be the image of $\bar{G}$ upon
factoring out by the scalars $Z$ of $\GL_n(q)$. Let $T$ be the socle of $G$, let $U$ be a subgroup
of $T$, and let $\bar{T}$ and $\bar{U}$ be the full preimages of $T$ and $U$. If $|T|<|U|^2\cdot
|\Out(T)|$ then
\[
|\bar{T}|< |\bar{U}|^2 \cdot |\Out(T)| \cdot   (q-1)
\]
\end{lem}

\begin{proof}
Suppose $|T|<|U|^2\cdot |\Out(T)|$. Then
\begin{align*}
|\bar{T}|&=|T| |\bar{T}\cap Z|<|U|^2\cdot |\Out(T)||\bar{T}\cap Z|\\
&=|\bar{U}:\bar{U}\cap Z|^2 \cdot |\Out(T)||\bar{T}\cap Z|\le |\bar{U}|^2 \cdot |\Out(T)| \cdot   (q-1)
\end{align*}
\end{proof}

\begin{prop}\label{prop:leftover2}
For $n\le 12$, $G_P$ cannot be a novelty subgroup of $G$. 
\end{prop}

\begin{proof}
By Proposition \ref{prop:dimatmost12notC1}, $G_P$ is not a $\mathcal{C}_1$-subgroup.
In Table \ref{table:novelties1}, we summarise the novelties outlined in the tables in Bray, Holt, Roney-Dougal \cite{BrayHoltRoney-Dougal}, apart from $\mathcal{C}_1$-subgroups. 
Let $\delta=| \bar{T}\cap Z(\GL_n(q))|$.
We first apply the restriction that $|\bar{T}|< |\bar{T}_P|^2 |\Out(T)|  (q-1)$:

\begin{center}
\begin{longtable}{p{1.6cm}|l|p{6cm}|c}
\caption{Possibilities for $(\bar{T},\bar{T}_P)$ upon consideration of the bound
$|\bar{T}|< |\bar{T}_P|^2 |\Out(T)|  (q-1)$ in the proof of Proposition \ref{prop:leftover2}.}\label{table:novelties1}\\
\toprule
$\bar{T}$ & $\bar{T}_P$ & Conditions & Bound? \\ 
\midrule
$\Omega_{12}^-(q)$ & $\PSL_3(3)$ & $q=p\equiv \pm 2,\pm 5,\pm 6\pmod{13}$,  $p\equiv \pm 5 \pmod{13}$& false\\ %
$\Omega_{12}^+(q)$ & $2\times \PSL_3(3)$ & $q=p\equiv \pm 1,\pm 3,\pm 4\pmod{13}$, $q\ne 3$ & false\\ %
 & $2^.M_{12}$ & $q=p\equiv \pm 5,\pm 7,\pm 11\pmod{24}$ & false\\ %
 & $\Omega^+_2(5)^6.2^{10}.S_6$ & $q=5$ & false\\ %
 & $\Omega^-_2(3)^6.2^{10}.S_6$ & $q=3$ & false\\ %
 & $\Omega^+_4(q)\times \SO_3(q)$ & $q\ge 5$, $q$ odd & false\\ %
$\SU_{12}(q)$& $\delta\circ 6^.A_6$ & $q=p\equiv 11,14\pmod{15}$ & false\\ %
$\SL_{12}(q)$&  $\delta\circ 6^.A_6$ & $q=p\equiv 1,4\pmod{15}$ & false\\ %
 & $12\circ 6^.A_6$ & $q=p^2$, $p\equiv 2,3\pmod{5}$, \hspace{2cm }$p\ne 2,3$ & false\\ %
$\SL_{11}(2)$&  $\PSL_2(23)$ & -- & false\\ %
$\Omega_{10}^-(q)$ & $2\times A_6$ & $q=p\equiv 7\pmod{11}$ & false\\ %
 & $2\times A_6.2_1$ & $q=p\equiv 7\pmod{11}$, $q\ne 7$ & false\\ %
 & $\delta\times \PSL_2(11)$ & $q=p\equiv 2,6,7,8,10 \pmod{11}$, $q\ne 2,7$, $d=1,2$ & false\\ %
 & $2^.\PSL_3(4)$ & $q=7$ & false\\ %
 & $M_{12}$ & $q=2$ & false\\ %
  & $\Omega^-_2(3)^5.2^8.S_5$ & $q=3$ & false\\ %
$\Omega_{10}^+(q)$ &  $2\times A_6$ & $q=p\equiv 5\pmod{12}$ & false\\ %
 & $2\times A_6.2_1$ & $q=p\equiv 1\pmod{12}$ & false\\ %
 & $d\times \PSL_2(11)$ & $q=p\equiv 1,3,4,5,9\pmod{11}$, $q\ne 3$, $d=1,2$ & false\\ %
 & $\Omega^+_2(5)^5.2^{8}.S_5$ & $q=5$ & false\\ %
 & $\SL_5(q).\frac{q-1}{\delta}$ & -- & true\\ %
$\SU_{10}(q)$& $\delta\circ 2^.\PSL_3(4)$ & $q=p\equiv 5,13,17\pmod{28}$ & false\\ %
 &  $\delta\circ 2^.\PSL_3(4).2_2$ & $q=p\equiv 3,19,27\pmod{28}$, \hspace{1cm} $q\ne 3$ & false\\ %
$\SL_{10}(q)$&  $\delta\circ 2^.\PSL_3(4)$ & $q=p\equiv 11,15,23\pmod{28}$ & false\\ %
 & $\delta\circ 2^.\PSL_3(4).2_2$ & $q=p\equiv 1,9,25\pmod{28}$ & false\\ %
$\SU_{9}(2)$& $3\times \PSL_2(19)$ & -- & false\\ %
$\Omega_{8}^+(q)$ &  $2^7:A_8$ & $q=p\equiv \pm 3 \pmod{8}$ & true: $q=3$\\ %
&  $2^4.2^6.\PSL_3(2)$ & $q=p$ odd & false \\ %
&  $\Omega_2^+(5)^4.4^3.S_4$ & $q=5$ & false\\ %
&  $\Omega_2^-(3)^4.4^3.S_4$ & $q=3$ & false\\ %
& $(D_{\tfrac{2(q^2+1)}{\delta}})^2.\delta.S_2$ & -- & false\\ %
$\Omega_{6}^-(q)$ & $\frac{\gcd(q+1,4)}{2}\times \PSL_2(7)$ & $q=p\equiv 3,5,6\pmod{7}$, $q\ne 3$ & false\\ %
  & $\frac{1}{4}\GO^-_2(3)^3.S_3$ & $q=3$ & true\\ %
$\Omega_{6}^+(q)$ & $\frac{\gcd(q-1,4)}{2}\times \PSL_2(7)$ & $q=p\equiv 1,2,4\pmod{7}$, $q\ne 2$ & false\\ %
  & $\frac{1}{4}\GO^+_2(5)^3.S_3$ & $q=5$ & false\\ %
  & $\frac{1}{\delta}\GL_3(q)$ & -- & true\\%
$\Sp_6(q)$ & $2^.A_5$ & $q=p\equiv \pm 11,\pm 19 \pmod{40}$ & false\\%
 & $2^.A_7$ & $q=9$ & false\\%
 & $2\times U_3(3)$ & $q=p\equiv \pm 19, \pm 29 \pmod{40}$ & false\\%
$\SU_{6}(q)$& $2\times 3^.A_6$ & $q=p\equiv 11,17 \pmod{24}$ & false \\%
 & $2\times 3^.A_6.2_3$ & $q=p\equiv 5,23 \pmod{24}$ & false\\%
 & $6^.A_6$ & $q=p\equiv 17,23,41,47 \pmod{48}$ & false\\%
  & $6^.\PSL_3(4)$ & $q=p\equiv 11,17 \pmod{24}$ & false\\%
  & $3^5.S_6$ & $q=2$ & true \\
    & $\SU_2(2^3).3.2$ & $q=2$ & false \\ %
$\SL_{6}(q)$& $2\times 3^.A_6$ & $q=p\equiv 7,13 \pmod{24}$ & false \\ %
 & $2\times 3^.A_6.2_3$ & $q=p\equiv 1,19 \pmod{24}$ & false\\ %
 & $6^.A_6$ & $q=p\equiv 1,7,25,31 \pmod{48}$ & false\\ %
 & $6^.A_6$ & $q=p^2\equiv 5,11,13,19 \pmod{24}$ & false\\ %
  & $6^.\PSL_3(4)$ & $q=p\equiv 7,13 \pmod{24}$ & false\\ %
$\Omega_5(7)$ & $\PSL_2(7)$ & --  & false\\ %
$\SL_5(3)$ & $\PSL_2(11)$ & -- & false \\ %
$\Omega^-_4(3)$ & $D_{10}$ & -- & true\\
$\Sp_4(2)$ & $ 5:4$ & -- & true \\
 & $ D_8\times 2 $ & -- & false\\ %
$\Sp_4(q)$, $q>2$ even & $[q^4]:C_{q-1}^2$ & -- & true\\%
& $C_{q-1}^2:D_8$ & $q\ne 4$ & false \\ %
& $C_{q+1}^2:D_8$ & -- & false\\%
& $C_{q^2+1}:4$ & -- & false \\%
$\Sp_4(7)$ & $\SL_2(7)$ & -- & false \\%
$\SU_4(q)$ &  $\gcd(q+1,4)\circ 2^.\PSL_2(7)$ & $q=p\equiv 3,5,6 \pmod{7}$, $q\ne 3$ & false\\%
 &  $4^3.S_4$ & $q=3$ & true\\ %
 &  $\SL_2(9).2$ & $q=3$ & true\\ %
$\SL_4(q)$ &  $\gcd(q-1,4)\circ 2^.\PSL_2(7)$ & $q=p\equiv 1,2,4 \pmod{7}$, $q\ne 2$ & false\\%
 &  $4^3.S_4$ & $q=5$ & false\\%
 &  $\SL_2(3)^2:2.2$ & $q=3$ & true\\%
$\Omega_3(q)$ & $A_4$ & $q=p\equiv \pm 11, \pm 19 \pmod{40}$ & true: $q=11,19$\\ %
$\SU_3(5)$ & $3\times \PSL_2(7)$ & -- & true\\%
 & $6^2:S_3$ & -- & true\\%
 & $21:3$ & -- & false\\%
 & $3_+^{1+2}:Q_8.3$ & -- & true\\%
$\SL_3(4)$ & $21:3$ & -- & true\\
 \bottomrule
\end{longtable}
\end{center}
Hence the cases left-over are those given by Table \ref{tbl:leftover1}.
\begin{center}
\begin{table}[ht]
\caption{Left-over cases in the proof of Proposition \ref{prop:leftover2}.}\label{tbl:leftover1}
\begin{tabular}{l|l||l|l}
\toprule
$T$ & $\bar{T}_P$ &$T$ & $\bar{T}_P$ \\
\midrule
$\Omega_{10}^+(q)$ &  $\SL_5(q).\frac{q-1}{\gcd(q-1,2)}$ &$\SU_4(3)$ &   $4^3.S_4$ \\
$\Omega_{6}^+(q)$ & $\frac{1}{\gcd(q-1,2)}\GL_3(q)$ &  &   $\SL_2(9).2$ \\
$\Omega_{8}^+(3)$ &  $2^7:A_8$ &  $\SL_4(3)$ &   $\SL_2(3)^2:2.2$ \\
$\Omega_{6}^-(3)$ & $\frac{1}{4}\GO^-_2(3)^3.S_3$ & $\Omega_3(11)$ & $A_4$ \\
$\SU_{6}(2)$&  $3^5.S_6$ & $\Omega_3(19)$ & $A_4$ \\
$\Omega^-_4(3)$ & $D_{10}$ & $\SU_3(5)$ & $3\times \PSL_2(7)$ \\
$\Sp_4(q)$ & $[q^4]:C_{q-1}^2$, $q>2$ even &  & $6^2:S_3$ \\
$\Sp_4(2)$ & $ 5:4$ &  & $3_+^{1+2}:Q_8.3$\\
&& $\SL_3(4)$ & $21:3$\\%
\bottomrule
\end{tabular}
\end{table}
\end{center}
We now apply the restriction $(t+1)^2<|T:T_P|<(t+1)^3$, together with the fact that $G_P$ induces a
$2$-transitive subgroup of degree $t+1$. Note, $t+1>2$, so the group of order $2$ is ruled out as a
possible $2$-transitive permutation group. The cases $\Omega^-_4(3)$ and $\SL_3(4)$ are ruled out as
the groups $D_{10}$ and $21:3$ do not act $2$-transitively on a set of size greater than $2$.

\begin{center}
\begin{table}[ht]\caption{Left-over cases in the proof of Proposition \ref{prop:leftover2},
after consideration of the bound $(t+1)^2<|T:T_P|<(t+1)^3$.}\label{tbl:leftover2}
\begin{tabular}{l|p{2.7cm}|c|c|c}
\toprule
$\bar{T}$ & $\bar{T}_P$ & $t+1$ &$|\bar{T}:\bar{T}_P|$ &  Bound?\\ 
\midrule
$\Omega_{10}^+(q)$ &  $\SL_5(q).\frac{q-1}{\gcd(q-1,2)}$ & $(q^5-1)/(q-1)$ & $q^{10} (q+1) (q^2+1)(q^3+1) (q^4+1)$ & false\\
$\Omega_{6}^+(q)$ & $\frac{1}{\gcd(q-1,2)}\GL_3(q)$ & $(q^3-1)/(q-1)$  & $q^3 (q+1) (q^2+1)$  & false\\
$\Omega_{8}^+(3)$ &  $2^7:A_8$ & $8$ & $3838185$ & false\\
$\Omega_{6}^-(3)$ & $\frac{1}{4}\GO^-_2(3)^3.S_3$ & $3$, $4$ & $8505$ &  false\\
$\SU_{6}(2)$&  $3^5.S_6$ & $6$ & $157696$& false\\
$\Sp_4(2)$ & $ 5:4$  & $5$ & $36$ & true\\
$\Sp_4(2^f)$ & $[2^{4f}]:C_{2^f-1}^2$ & $\le 2^{4f}$ & $ (2^f+1)^2 (2^{2f}+1) $& false \\
$\SU_4(3)$ &   $4^3.S_4$ & $3$, $4$ & $8505$ & false \\
 &   $\SL_2(9).2$ & $10$ & $36288$ & false \\
 $\SL_4(3)$ &   $\SL_2(3)^2:2.2$ & $3$ & $5265$ & false \\
$\Omega_3(11)$ & $A_4$ & $4$ & $55$ &  true \\
 $\Omega_3(19)$ & $A_4$ & $4$ &$285$ & false\\
$\SU_3(5)$ & $3\times \PSL_2(7)$ & $7$, $8$ & $750$ & false \\
 & $3_+^{1+2}:Q_8.3$ & $9$ & $1750$ & false \\
 & $6^2:S_3$ & $3$, $4$ & $1750$ & false  \\
\bottomrule
\end{tabular}
\end{table}
\end{center}

So the only cases (from Table \ref{tbl:leftover2}) we are left with are $(T,T_P)=(\Omega_3(11), A_4)
$ where the degree is $55$ and $t+1=4$, and $(T,T_P)=(\Sp_4(2), 5:4) $ where the degree is $36$ and
$t+1=5$. For both of these cases, we solve for $s$, since $(s+1)(st+1)$ is equal to the degree $d$.
The discriminant of this quadratic in $s$ is
\[
\Delta_{t,d} := (t-1)^2  + 4 d t 
\]
\begin{itemize}
\item $(T,T_P)=(\Omega_3(11), A_4) $ where $d=55$ and $t=3$. Here, $\Delta_{t,d}= 664$ which is not a square.
So this case does not arise.
\item  $(T,T_P)=(\Sp_4(2), 5:4) $ where $d=36$ and $t=4$. Here, $\Delta_{t,d}= 585$ which is not a square.
So this case does not arise either.
\end{itemize}
\end{proof}

\begin{prop}
\label{prop:noveltydimatleast13}
For $n\ge 13$,  $G_P$ cannot be a novelty subgroup of $G$. 
\end{prop}

\begin{proof}
We consider first the subgroups listed in \cite[Table 3.5.I]{KleidmanLiebeck}. In the first two
cases, we have $\Sp_2(q) \Wr S_r < \POmega_{2^r}^+(q)$, where $q \ge 4$ and even. Assuming that $q =
2^k$ and using \cite[Lemma 4.3]{largesubs} and Corollary \ref{cor:largesub}, we see that
$$2^{k2^{r-1}(2^r-1) - 3} < |\POmega_{2^r}(2^k)| < |\Sp_2(q) \Wr S_r|^3 < 2^{r^2 + 3kr},$$ which
implies that $k < 3(r^2+1)/(2^{r-1}(2^r - 1) - 9r) \le 17/28$, a contradiction to $4 \le q = 2^k$.
The final case in \cite[Table 3.5.I]{KleidmanLiebeck}, where $\Sp_4(q) \Wr S_r <
\POmega_{4^r}^+(q)$, is ruled out in precisely the same manner.

We now consider the subgroups listed in \cite[Table 3.5.H]{KleidmanLiebeck}. We rule out nearly all
cases in a similar manner: we bound $t+1$ from above by examining the possible $2$-transitive
actions of our putative $T_P$ and show that in fact $(t+1)^3 \le |T:T_P| = |\mathcal{P}|$, a
contradiction to Lemma \ref{lem:Pbounds}. For instance, we consider the case when $T_P$ is a
subgroup of type $\GL_1(2) \Wr S_n$. This means that $t+1 \le n$, and so by Lemma \ref{lem:Pbounds}
and \cite[Lemma 4.3]{largesubs} we have:
\[
n^3 \ge (t+1)^3 > |\mathcal{P}| = |T:T_P| > \frac{2^{n^2-2}}{n!}.
\]
This implies that $2^{n^2 -2} < n^{n+3}$, which does not hold for $n \ge 13.$

We will now consider the possible $T_P$ listed in \cite[Table 3.5.H]{KleidmanLiebeck} such that the
group has a $2$-transitive action on more than $n$ points. First, we suppose that $T \cong
\PSL_n(q)$ and $T_P$ is a $\GL_m(q) \oplus \GL_{n-m}(q)$ type subgroup. We know that $n-m \ge m$, so
$t+1 \le (q^{n-m} - 1)/(q-1)$. By Lemma \ref{lem:Pbounds}(i) we must have $(q^{n-m} - 1)^3/(q-1)^3
\ge (t+1)^3 > |\mathcal{P}| > {n \brack m}_q$, and this immediately implies that $m = 1,2$. Suppose
that $m = 1$. Then $|\mathcal{P}| = q^{n-1}(q^n - 1)/(q-1)$, and by Lemma \ref{lem:Pbounds}(ii),
$s^2(t+1) < |\mathcal{P}|$ implies that $s \le q^{n/2} + 1$. On the other hand, $(s+1)(st+1) =
|\mathcal{P}|$, and $st+1$ is coprime to $q$, which implies that $q^{n-1}$ divides $s+1$, a
contradiction. A similar argument also rules out $m=2$, and so $T_P$ cannot be a $\GL_m(q) \oplus
\GL_{n-m}(q)$ type subgroup.

We now assume that $T_P$ is a $P_{m,n-m}$ type subgroup in $\PSL_n(q)$. Since $T_P$ must have a
$2$-transitive action on $t+1$ points, it follows that $t+1$ is either $(q^m - 1)/(q-1)$ or
$(q^{n-2m} - 1)/(q-1)$. If $t+1 = (q^m - 1)/(q-1)$, then again using Lemma \ref{lem:Pbounds}(i), we
see that ${n \brack m}_q < |\mathcal{P}| < (q^m - 1)^3/(q-1)^3 < q^{3n/2}$. This implies that $m =
1$, a contradiction to $t \ge 2$. If $t+1 = (q^{n - 2m} - 1)/(q-1)$, then a similar argument
immediately implies that $m = 1,2$. When $m = 2$, $|\mathcal{P}| = (q^n-1)(q^{n-1} - 1)(q^{n-2} -
1)(q^{n-3} - 1)/((q^2-1)(q^2-1)(q-1)(q-1))$ and $t+1 = (q^{n-4} - 1)/(q-1)$. However, $(t+1)^3 <
|\mathcal{P}|$ in this case, a contradiction. Hence we assume that $m = 1$ and $t+1 = (q^{n-2} -
1)/(q-1)$. We note that $|\mathcal{P}| = (q^n-1)(q^{n-1}-1)/(q-1)^2$, and by Lemma
\ref{lem:Pbounds}(ii) we find that $s < \sqrt{2}q^{n/2}$. We also note that $T_P \cong
[q^{2n-3}]:[a_{1,1,n-1}^+/\gcd(q-1,n)].(\PSL_1(q)^2 \times \PSL_{n-2}(q)).[b_{1,1,n-2}^+]$ (see
\cite[Proposition 4.1.22]{KleidmanLiebeck} for details). Note that $t+1 = (q^{n-2} - 1)/(q-1)$
implies that the kernel of the action of $T_P$ on its neighbors contains the full subgroup
$[q^{2n-3}]:\PSL_{n-3}(q)$, which is a subgroup of $T_{\ell}$. Note that this implies that $T_{\ell}
> [q^{2n-3}]:\PSL_{n-3}(q)$, and we deduce that $T_{\ell}$ is a $\mathcal{C}_1$-subgroup of $T$ with
$m \le 3$. However, in all cases this forces $s+1$ to be larger than $\sqrt{2}q^{n/2}$, a
contradiction. Thus $T_P$ cannot be a $P_{m,n-m}$ type subgroup in $\PSL_n(q)$.

We now assume that $T_P$ is a $P_{n/2 - 1}$ subgroup in $\POmega_n^+(q)$. This forces $t+1$ to be
$(q^{n/2-1} - 1)/(q-1)$ or $q^{n/2 -1}$. However, by Lemma \ref{lem:isotropics}, we see that
$|\mathcal{P}| = {n/2 \brack {n/2-1}}_q \prod_{i=1}^{n/2 -1}(q^{n/2 - i} + 1) > q^{3n/2 - 3} \ge
(t+1)^3$, a contradiction.

Now assume that $T_P$ is a $\GL_{n/2}(q).2$ type subgroup of $\POmega_n^+(q)$, where $n/2$ is odd.
However, this implies that $t+1 = (q^{n/2} - 1)/(q-1)$ and $|\mathcal{P}| \ge {n/2 \brack n/2}_q
\prod_{i=1}^{n/2}(q^{n/2-i} + 1) > q^{3n/2} \ge (t+1)^3$ when $n \ge 13$, a contradiction to Lemma
\ref{lem:Pbounds}(i).

Next, we assume that $T_P$ is a subgroup of type $\Sp_2(2) \times \Sp_{n/2}(2)$ in $\POmega_n^+(2)$.
This implies that $t+1 \le 2^{n/4 -1}(2^{n/4} + 1) < 2^{n/2}$. Again, we use \cite[Lemma
4.3]{largesubs} to see that $$|\mathcal{P}| = |T:T_P| > 2^{\frac{3n^2}{4} - \frac{3n}{2} - 6}.$$ By
Lemma \ref{lem:Pbounds}, this means that $$2^{\frac{3n^2}{4} - \frac{3n}{2} - 6} < |\mathcal{P}| <
(t+1)^3 < 2^{\frac{3n}{2}},$$ which does not hold for $n \ge 13$.

We now consider the case that $T_P$ is a subgroup of type $\rmO_4^+(q) \otimes
\rmO_{n/4}^{\epsilon}(q)$, where $q$ is odd. By Lemma \ref{stab}, since $n \ge 13$, $n/4$ must be
$4$ or $5$. However, if $n = 20$, then $\PSp_4(q)$ with $q$ odd is a composition factor, which does
not happen in a $2$-transitive action that is not affine. Thus $n = 16$ and $T_P$ is of type
$\rmO_4^+(q) \otimes \rmO_{n/4}^{-}(q)$. Thus $t+1 \le q^2+1$, but by \cite[Lemma 4.3]{largesubs}
$|\mathcal{P}| \ge (q^{120}/4)/((q^4)^2 \cdot q^8) > (t+1)^3$, a contradiction to Lemma
\ref{lem:Pbounds}.

Finally, for all other novelty subgroups listed in \cite[Table 3.5.H]{KleidmanLiebeck} with a
$2$-transitive action and no composition factor that is not allowed by Lemma \ref{stab}, we find
that $t+1 \le n$ when the degree of $T$ is at least $n$. These groups are all ruled out in similar
fashion to those considered above, and we conclude that $G_P$ cannot be a novelty subgroup of $G$
when $n \ge 13.$
\end{proof}

 \section{Exceptional groups of Lie type}

The purpose of this section is to rule out exceptional groups from acting primitively on both the
points and lines of a finite antiflag-transitive generalized quadrangle. As for the classical
groups, we will assume Hypothesis \ref{hyp:3arcAlmostSimple}.

\begin{prop}\label{excep-Lie}
Assume Hypothesis \ref{hyp:3arcAlmostSimple}. Then $T$ cannot be an exceptional group of Lie type.
\end{prop}

\begin{proof}
Let $G$ be an almost simple exceptional group of Lie type, and let $T=\soc(G)$. Let $P$ be a point
of $\mathcal{Q}$. By \cite[Theorem~7]{largesubs}, one of the following holds:
\begin{enumerate}[(A)]
\item the pair $(T, T_P)$ lies in \cite[Table~3]{largesubs},
\item $G_P$ is a novelty maximal subgroup of $G$, where $T_P$ is a non-maximal parabolic subgroup
  of $T$,
\item the pair $(T,T_P)$ lies in \cite[Table~2]{largesubs}, or
\item $G_P$ is a maximal parabolic subgroup of $G$.
\end{enumerate}

\begin{description}[leftmargin=12pt, style=unboxed, itemsep=1ex]
\item[Case (A)] We consider first the possibility that the pair $(T, T_P)$ lies in
  \cite[Table~3]{largesubs}. By Lemma \ref{stab} (and noting that the local action is not affine),
  the only possibility is $T = \F_4(2)$ and $T_P = S_6 \Wr S_2$. However, this implies that $t+1 \le
  10$, and by Lemma \ref{lem:Pbounds}(i), that $|T:T_P| < 1000$; a contradiction.

\item[Case (B)] Next, we consider the possibility that $T_P$ is a non-maximal parabolic subgroup
  of $T$. Any such subgroup $T_P$ arises from a graph automorphism, and the only possibilities are:
\begin{enumerate}[(i)]
\item $T = \E_6(q)$ and $T_P = [q^{24}].\frac{1}{d}(\rmO_8^+(q).C_{q-1})$, where $d =
  \gcd(3,q-1)$;
\item $T = \E_6(q)$ and $T_P = [q^{31}].\frac{1}{d}((\SL_3(q) \times \SL_2(q) \times
  \SL_2(q)).C_{q-1}^2)$, where $d = \gcd(3,q-1)$;
\item $T= \F_4(q)$ and $T_P = [q^{20}].\Sp_4(q).C_{q-1}^2$, where $q = 2^n$;
\item $T = \F_4(q)$ and $T_P = [q^{22}].(\SL_2(q) \times \SL_2(q)).C_{q-1}^2$, where $q = 2^n$;
\item $T= \mathrm{G}_2(q)$ and $T_P = [q^6]:C_{q-1}^2$, where $q = 3^n$.
\end{enumerate}

The first case (i) is ruled out by Lemma \ref{stab}. In the second case (ii), by the Classification
of Finite $2$-Transitive Groups we have that $t+1 \le q^3$. However, $|T:T_P| > q^9$, a
contradiction to Lemma \ref{lem:Pbounds}(i). In the third and fourth cases where $T= \F_4(q)$ and
$T_P = [q^{20}].\Sp_4(q).C_{q-1}^2$ or $[q^{22}].(\SL_2(q) \times \SL_2(q)).C_{q-1}^2$, by the
Classification of Finite $2$-Transitive Groups we have that $t+1 \le q^4$. However, $|T:T_P| >
q^{12}$, a contradiction to Lemma \ref{lem:Pbounds}(i). Finally, if $T= \mathrm{G}_2(q)$ and $T_P =
[q^6]:C_{q-1}^2$, then by the Classification of Finite $2$-Transitive Groups, we have that $t+1 \le
q$ and $|T:T_P| > q^3$, a contradiction to Lemma \ref{lem:Pbounds}(i). Therefore, $T_P$ is a maximal
subgroup of $T$.

\item[Case (C)] We now look at the candidates from Table~2 of \cite{largesubs}.

\begin{description}[leftmargin=12pt, style=unboxed, itemsep=1ex]

\item[(C1): $T=\soc(G)=\E_8(q)$, $\E_7(q)$, $\E_6(q)$, or $^2{}\E_6(q)$]

By Lemmas~\ref{comp-factor} and \ref{stab}, these candidates are excluded except for $T=\E_7(q)$ and
$T_P\rhd\PSL_8(q)$. In this case, $t+1=|\Gamma(P)|={q^8-1\over q-1}$, and
\begin{align*}
|\mathcal{P}|=|T:T_P| &= {q^{63}(q^{18}-1)(q^{14}-1)(q^{12}-1)(q^{10}-1)(q^8-1)(q^6-1)(q^2-1)\over
 q^{28}(q^8-1)(q^7-1)(q^6-1)(q^5-1)(q^4-1)(q^3-1)(q^2-1)}\\
  &>\left({q^8-1\over q-1}\right)^3=(t+1)^3,
 \end{align*}
a contradiction to Lemma \ref{lem:Pbounds}.

\item[(C2): $T=\F_4(q)$] The only possibilities are when $q=2$ and $T_P=\Sp_8(2)$ or $\PSL_4(3)$.
  For the former, $t+1=|\Gamma(P)|=2^7\pm2^3$, and
  \[|\mathcal{P}|=|T:T_P|={2^{24}(2^{12}-1)(2^8-1)(2^6-1)(2^2-1)\over
  2^{16}(2^8-1)(2^6-1)(2^4-1)(2^2-1)}>(2^7\pm2^3)^3,\] which is not possible. Similarly, the latter
  candidate is not possible.

\item[(C3): $T={}^2\F_4(q)$] Then either $T_P=\Sz(q)\times\Sz(q)$, or $q=2$ and $T_P = \Sp_4(2)$.
  For the former, $t+1=|\Gamma(P)|=q^2+1$, and
  \[|\mathcal{P}|=|T:T_P|={q^{12}(q^6+1)(q^6-1)(q^3+1)(q-1)\over q^4(q^2+1)^2(q-1)^2}>(q^2+1)^3,\]
  which is not possible. For the latter, $T_P=\Sp_4(2)\cong S_6$, and $t+1=|\Gamma(P)|=6$ or 10. A
  simple calculation shows that $|\mathcal{P}|=|T:T_P|>10^3$, which is a contradiction.

\item[(C4): $T={}^3\D_4(q)$] In this case, by Lemma \ref{stab} we need only consider the cases
  when $T_P$ is one of $2.(\PSL_2(q^3)\times\PSL_2(q)).2$, $\SL_3(q).C_{q^2+q+1}.2$,
  $\SU_3(q).C_{q^2-q+1}.2$, or $7^2:\SL_2(3)$ when $q=2$. If $T_P =
  2.(\PSL_2(q^3)\times\PSL_2(q)).2$, then $t+1 = q^3 + 1$ by Lemma \ref{comp-factor}. If $T_P =
  \SL_3(q).C_{q^2+q+1}.2$, then $t+1 \le q^2 + q + 1$ (unless $q = 2$, in which case $t+1 \le 8$).
  If $T_P = \SU_3(q).C_{q^2-q+1}.2$, then $t+1 \le q^3 + 1$. If $T_P = 7^2:\SL_2(3)$, then $t+1 \le
  49$. In any case, we have that $$|\mathcal{P}| = |T:T_P| > (t+1)^3,$$ a contradiction to Lemma
  \ref{lem:Pbounds}(i).

\item[(C5): $T=\G_2(q)$, $q>2$] Again using Lemma \ref{stab} to narrow down the possibilities and
  noting that $T_P^{\Gamma(P)}$ is not affine, we have that $T_P = \SL_3(q):2$, $\SU_3(q):2$,
  $\PSL_2(q) \times \PSL_2(q)$ if $q = 2^n$, $2.(\PSL_2(q) \times \PSL_2(q)):2$ if $q$ is odd,
  $\Ree(q)$ if $q = 3^{2n+1}$, $\PSU_3(3):2$ if $q = 5,7$, $\PSL_2(13)$ if $q = 3,4$, or
  $2^3:\SL_3(2)$ if $q = 3$.

If $T_P = \SL_3(q):2$, then $t+1 = q^2 + q + 1$ and $$|\mathcal{P}| = \frac{q^3(q^3 + 1)}{2} =
(s+1)(s(q^2+q) + 1).$$ This implies that $q^3$ (or $q^3/2$ if $q$ is even) divides $s+1$, a
contradiction since $s \le t$. If $T_P = \SU_3(q):2$, then $t+1 = q^3 + 1$ and $$|\mathcal{P}| =
\frac{q^3(q^3 - 1)}{2} = (s+1)(sq^3 + 1).$$ This implies that $q^3$(or $q^3/2$ if $q$ is even)
divides $s+1$, a contradiction since $s^2(t+1) < |\mathcal{P}|$ by Lemma \ref{lem:Pbounds} (ii). If
$T_P = \PSL_2(q) \times \PSL_2(q)$ or $2.(\PSL_2(q) \times \PSL_2(q)):2$, then $|\mathcal{P}| =
q^4(q^4+q^2+1) > (t+1)^3$ (we have $t+1 \le q+ 1$ unless $(q, t+1) = (4, 6)$, $(8, 28)$), a contradiction to
Lemma \ref{lem:Pbounds}(i). If $T_P = \Ree(q)$, then $t+1 = q^3+1$ and $$|\mathcal{P}| = q^3(q^3 -
1)(q+1) = (s+1)(sq^3 + 1).$$ However, this implies that $q^3$ divides $s+1$, which in turn implies
that $s+1 = q^3$ since $s \le t$. Substituting this value of $s$ back in yields $q^3(q^3 - 1)(q+ 1)
= q^3(q^3(q^3 - 1) + 1)$, which has no positive integral solutions. Finally, in the remaining cases,
we have: $t+1 = 28$ if $T_P = \PSU_3(3):2$, $t+1 = 14$ if $T_P = \PSL_2(13)$, and $t+1 = 8$ if $T_P
= 2^3:\SL_3(2)$. In each of these cases, $|T:T_P| > (t+1)^3$, a contradiction to Lemma
\ref{lem:Pbounds}(i).

\item[(C6): $T=\Sz(q)$] In this case, we get $q = 8$ and $T_P = C_{13}:C_4$. However, this means
  that $t+1 = 13$, and $(s+1)(12s+1) = |T:T_P| = 560$ has no integral roots, a contradiction.

\item[(C7): $T=\Ree(q)$] In this case, we have $T_P = C_2 \times \PSL_2(q)$, where $q = 3^{2n+1}$.
  This implies that $t + 1 = q + 1$, and so $|T:T_P|=q^2(q^2-q+1)$. Thus
  $(s+1)sq+(s+1)=(s+1)(st+1)=|\mathcal{P}|=q^2(q^2-q+1)$. This implies that $q$ divides $s+1$, and
  then $q^2$ divides $s+1$, which is not possible since $s\leqslant t =q$.
\end{description}

Therefore, we have ruled out the candidates from Table~2 of \cite{largesubs}.

\item[Case (D)] Now suppose that $G_P$ is a maximal parabolic subgroup of $G$. Then $G_P$ is an
  extension of a $p$-group by the Chevalley group determined by a maximal subdiagram of the Dynkin
  diagram of $G$.

\begin{description}[leftmargin=12pt, style=unboxed, itemsep=1ex]

\item[(D1): $T=\E_8(q)$]

Since a composition factor of $G_P$ is not $\E_7(q)$, $\E_6(q)$, $\POmega^+_{10}(q)$, or
$\POmega^+_{14}(q)$ by Lemma~\ref{comp-factor}, we conclude that $G_P=R.K.\calO$, where $R$ is the
solvable radical of $G_P$, $\calO\leqslant\Out(K)$, and $K$ is one of the following groups:
\[\PSL_4(q)\times\PSL_5(q),\ \PSL_2(q)\times\PSL_3(q)\times\PSL_5(q),\ \PSL_8(q),\
\PSL_2(q)\times\PSL_7(q).\] By Lemma~\ref{comp-factor}, the corresponding value of $t+1$ is equal to
\[q^6,\ {q^5-1\over q-1},\ q^5,\ {q^5-1\over q-1},\ q^8,\ {q^8-1\over q-1},\ q^7,\ {q^7-1\over
q-1}.\] Calculation shows that index $|T:T_P|$ is larger than $q^{30}-1$, which is not possible
since $|T:T_P|$ should be less than $(t+1)^3$.

\item[(D2): $T=\E_7(q)$]

By Lemma~\ref{comp-factor}, a composition factor of $G_P$ is not $\E_6(q)$, $\POmega^+(10,q)$, or
$\POmega^+(12,q)$, we conclude that $G_P=R.K.\calO$, where $R$ is the solvable radical of $G_P$,
$\calO\leqslant\Out(K)$, and $K$ is one of the following groups: \[\PSL_3(q)\times\PSL_5(q),\
\PSL_2(q)\times\PSL_3(q)\times\PSL_4(q),\ \PSL_7(q),\ \PSL_2(q)\times\PSL_6(q).\] By
Lemma~\ref{comp-factor}, the corresponding value of $t+1$ is equal to \[q^5,\ {q^5-1\over q-1},\
q^4,\ {q^4-1\over q-1}, \ 8 \mbox{ (with $q=2$)},\ q^7,\ {q^7-1\over q-1},\ q^6,\ {q^6-1\over
q-1}.\] Calculation shows that index $|T:T_P|$ is larger than $(q^{18}-1)(q^{14}-1)$, which is not
possible since $|T:T_P|$ should be less than $(t+1)^3$.

\item[(D3): $T=\E_6(q)$]

A parabolic subgroup $G_P=R.K.\calO$, where $R$ is the solvable radical of $G_P$,
$\calO\leqslant\Out(K)$, and $K$ is isomorphic to one of the following groups: \[\POmega^+_{10}(q),\
\PSL_2(q)\times\PSL_5(q),\ \PSL_2(q)\times\PSL_3(q)\times\PSL_3(q),\ \PSL_6(q).\] By
Lemma~\ref{comp-factor}, first candidate is not possible, and the other three have values of $t+1$
as below: \[q^5,\ {q^5-1\over q-1},\ q^3,\ {q^3-1\over q-1}, \ 8 \mbox{ (with $q=2$)},\ q^6,\
{q^6-1\over q-1}.\] Calculation shows that index $|T:T_P|$ is larger than
$(q^{12}-1)(q^6+q^3+1)(q^4+1)$, which is larger than $(t+1)^3$, not possible.

\item[(D4): $T=\F_4(q)$]

By Lemma~\ref{comp-factor}, the candidates for $T_P$ are $R.\Sp_6(q)$ with $q$ even, $R.\Sp_6(2)$
with $q=2$, and $R.(\PSL_2(q)\times\PSL_3(q)).\calO$, where $R$ is the solvable radical of $T_P$ and
$\calO\leqslant\Out(\PSL_2(q)\times\PSL_3(q))$.

If the third case occurs, then $q=2$ and $t+1=q^3, q^2+q+1,$ or $8$, and
$|T:T_P|>(q^{12}-1)(q^8-1)$, which is greater than $(t+1)^3$, not possible. If we are in the second
case, then the valency $t+1=2^6$, $2^5-2^2$ or $2^5+2^2$, and the index
\[(s+1)(st+1)=|T:T_P|=(2^{12}-1)(2^4+1).\] Suppose that $t+1=2^6$. Then
$s^2t+st+s+1=2^{16}+2^{12}-2^4-1$. It implies that $2^6$ divides $2^4+2+s$. Since $s\leqslant
t=2^6-1$, we conclude that $2^6=2^4+2+s$, namely, $s=2^6-2^4-2$. Inserting this value of $s$ into
$(s+1)(st+1)=(2^{12}-1)(2^4-1)$ leads to a contradiction. If $t+1=2^5-2^2$ or $2^5+2^2$, then
$|\mathcal{P}|=(2^{12}-1)(2^4+1)>(2^5+2^2)^3\geqslant(t+1)^3$, which is a contradiction.

Finally, if we are in the first case, then $t+1 = q^6$ and by \cite{fsg} we have $T_P = (q^6 \times
q^{1+8}):\Sp_6(q).(q-1)$. This means that $$|\mathcal{P}| = |T:T_P| =
\frac{(q^{12}-1)(q^8-1)}{(q^4-1)(q-1)} = (s+1)(s(q^6 -1) + 1).$$ Furthermore, from the structure of
$T_P$ and $T_P^{\Gamma(P)}$, we know that $q^9$ divides $|T_\ell|$, since the $q^{1+8}$ must be in
the kernel of the local action. Moreover, $$|\mathcal{L}| = q^6((q^6 - 1)s + 1)$$ is coprime to
$(q^6 - 1)$, and since $|\mathcal{L}| = |T:T_\ell|,$ we have also that $(q^2 - 1)^2(q^6 - 1)^2$
divides $|T_\ell|$. This implies that $|T| < |T_\ell|^3$, and so $T_\ell$ is a large maximal
subgroup of $T$. Reasoning as above for the non-parabolic maximal subgroups of $\F_4(q)$, we
conclude that $q=2$, which was ruled out above. This leaves only parabolic subgroups. Since $s \ge
\sqrt{q^6 - 1}$ by Lemma \ref{lem:GQbasics} (iii), examining the possibilities for $T_\ell$ we see
that this implies that $s+1 = q^3, q^4,$ or $q^6$. However, none of these values of $s$ gives a real
solution to $(s+1)((q^6 - 1)s + 1) = (q^{12} - 1)(q^4+1)/(q-1)$, a final contradiction.

\item[(D5): $T={}^2\E_6(q)$] By Lemma~\ref{comp-factor}, the candidate for $T_P$ are
  $R.(\SL_3(q^2)\times\SL_2(q))$ or $R.(\SL_2(q^2)\times\SL_3(q))$, refer to \cite[p.101]{GLS3}.
  Then $t+1\leqslant q^6$, and $|\mathcal{P}|=|T:T_P|\ge (q^9+1)(q^5+1)(q^4+1)$, which is not
  possible by Lemma~\ref{lem:Pbounds}.

\item[(D6): $T={}^2\F_4(q)$, where $q=2^{2n+1}$]

By \cite{Malle}, a maximal parabolic subgroup of $T$ is isomorphic to $[q^{11}]\colon
(\PSL_2(q)\times C_{q-1})$ or $[q^{10}]\colon (\Sz(q)\times C_{q-1})$. For the former case, the
valency $t+1=|\Gamma(P)|$ equals $q^2$, or $q+1$, or $q$ with $q=5$, $7$ or $11$, or $6$ with $q=9$,
and the number of points $|T:T_P|=(q^6+1)(q^4-1)(q^3+1)/(q-1)$. For the second case,
$t+1=|\mathcal{P}|=q^2+1$, and $|T:T_P|=(q^6+1)(q^2-1)(q^3+1)(q+1)$. In each case, we have
$|\mathcal{P}|=|T:T_P|>(t+1)^3$, which is a contradiction.

\item[(D7): $T = {}^3\D_4(q)$] In this case, $T_P = q^{1+8}\colon \SL_2(q^3).C_{q-1}$ or
  $q^{2+3+6}\colon \SL_2(q).C_{q^3-1}$. If $T_P = q^{2+3+6}\colon \SL_2(q).C_{q^3-1}$, then by the
  Classification of Finite $2$-Transitive Groups, we have that $t+1 \le q^3$. However, in this case
  $|T:T_P| = (q^8 + q^4 + 1)(q^3 + 1) > (t+1)^3$, a contradiction to Lemma \ref{lem:Pbounds}(i).
  Hence we have $T_P = q^{1+8}\colon \SL_2(q^3).C_{q-1}$. In this case, $t+1=q^6$ or $q^3+1$ (note
  that $t+1 \neq q$ by Lemma \ref{comp-factor}), and so $t=q^6-1$ or $q^3$. The cardinality
  $|\mathcal{P}|=(s+1)(st+1)=q^8+q^4+1$. If $t=q^6-1$, then $s\geqslant q^3$, and thus
  \[q^8+q^4+1=(s+1)(st+1)>q^3(q^3(q^6-1)+1),\] which is not possible. For $t=q^3$, we have
  \[q^8+q^4+1=|\mathcal{P}|=(s+1)(sq^3+1)=s^2q^3+sq^3+s+1,\] so $q^8+q^4=s^2q^3+sq^3+s$. It implies
  that $s$ is divisible by $q^3$, and hence $s=q^3=t$. Therefore,
  $q^8+q^4+1=|\mathcal{P}|=(q^3+1)(q^6+1)$, which is a contradiction.

\item[(D8): $T = \G_2(q)$]

In this case, $T_P = [q^5]:\GL_2(q)$. This implies that $|\mathcal{P}| = |T:T_P| = (q^6 - 1)/(q-1)$.
By the Classification of Finite $2$-Transitive Groups, we have that $t+1 = q^2$ or $t+1 \le q+1$. If
$t+1 \le q+1$, then $(t+1)^3 < |\mathcal{P}|$, a contradiction to Lemma \ref{lem:Pbounds}(i). Hence
$t+1 = q^2$. Now, since $T_P^{\Gamma(P)}$ is affine, we know that $T_P^{[1]} \le T_\ell$ contains a
group of order $q^3$. Moreover, since $$|T:T_\ell| = |\mathcal{L}| = q^2((q^2-1)s + 1)$$ is coprime
to $q^2-1$, we have also that $(q^2 - 1)^2$ divides $|T_\ell|$. Hence $q^3(q^2 - 1)$ divides
$|T_\ell|$. However, examining the list of subgroups of $T$ (see, for instance, \cite[Table
4.1]{fsg}), none satisfies this property, a contradiction.

\item[(D9): $T = \Sz(q)$] In this case, $T_P=q^{1+1}.C_{q-1}$, $t+1 =q$, and the index
  $|T:T_P|=q^2+1$. Then \[q^2+1=|T:T_P|=|\mathcal{P}|=(s+1)(st+1)=(s+1)(s(q-1)+1).\] It implies that
  $sq(s+1)-s^2=q^2$. Thus $s$ divides $q^2=2^{2n}$, so $s=2^m$ for some $m\leqslant 2n$. Hence
  $2^{m+n}(2^m+1)=sq(s+1)=s^2+q^2=2^{2m}+2^{2n}$, which is not possible.

\item[(D10): $T = \Ree(q)$] In this case, $T_P = q^{1+1+1}:C_{q-1}$. By the Classification of
  Finite $2$-Transitive Groups, we have that $t+1 \le q$. This implies that $$|\mathcal{P}| =
  |T:T_P| = q^3 + 1 > (t+1)^3,$$ a contradiction to Lemma \ref{lem:Pbounds}(i).

\end{description}
\end{description}
Therefore, $T$ cannot be an exceptional group of Lie type.
\end{proof}

\begin{proof}[Proof of Theorem \ref{maintheorem}]
Suppose that $G$ acts transitively on the antiflags of a finite thick generalized quadrangle $\mathcal{Q}$.  By Theorem \ref{thm:characterization}, if $\mathcal{Q}$ is not the unique generalized quadrangle of order $(3,5)$ or its dual, then $G$ is an almost simple group of Lie type acting primitively on both the points and lines of $\mathcal{Q}$.  By Propositions \ref{prop:linearmaxlgeom}, \ref{prop:unitarymaxlgeom}, \ref{prop:symplecticmaxlgeom}, \ref{prop:orthogonalmaxlgeom}, \ref{prop:notC9}, \ref{prop:leftover2}, \ref{prop:noveltydimatleast13}, and \ref{excep-Lie}, we have that any finite thick generalized quadrangle with a collineation group that is an almost simple group of Lie type acting primitively on points, primitively on lines, and transitively on antiflags is a classical generalized quadrangle or the dual of a classical generalized quadrangle, as desired. 
\end{proof}

\section{Concluding remarks}

In this paper, our classification was made possible due mostly in part to Theorems
\ref{thm:3arcNtrans} and \ref{thm:3arcLimprim} that reduce the problem to the case that $G$ acts
primitively on both points and lines of almost simple type. From there, we showed that the
point-stabilizer is \emph{large} and used the work of Alavi and Burness \cite{largesubs} to
determine the possibilities for the almost simple group and its point-stabilizer. For the much
weaker hypothesis of local $2$-arc-transitivity, we can still give strong structural information. In
a forthcoming paper \cite{Locally2}, we show that if $G$ acts locally-$2$-arc-transitively on the
incidence graph $\Gamma$ of a generalized quadrangle $\mathcal{Q}$, and if $G$ acts quasiprimitively
on both points and lines of $\mathcal{Q}$, then $G$ is almost simple.

\medskip

\subsection*{Acknowledgements} We would like to thank Professor Tim Burness for his aid in
determining the novelty parabolic subgroups of exceptional groups of Lie type.

\bibliographystyle{plain}
\bibliography{Loc2arcQuasiprimGQ}

\end{document}